\documentclass[11pt]{article}

\usepackage{tikz}
\usepackage{subfigure}
\usepackage{scalefnt}
\usepackage{amsmath,amsfonts,amsthm,enumerate}
\numberwithin{equation}{section}
\newtheorem{defn}{{\bf Definition}}[section]
\newtheorem{eg}[defn]{{\bf Example}}
\newtheorem{lemma}[defn]{{\bf Lemma}}
\newtheorem{prop}[defn]{{\bf Proposition}}
\newtheorem{theo}[defn]{{\bf Theorem}}

\newtheorem{remark}[defn]{{\bf Remark}}

\begin{document}

\title{An algorithmic approach to construct crystallizations of $3$-manifolds from  presentations of fundamental groups}
\author{Biplab Basak}
\date{}
\maketitle
\vspace{-10mm}
\begin{center}

\noindent {\small Department of Mathematics, Indian Institute of Science, Bangalore 560\,012, India.}

\noindent {\small {\em E-mail address:} biplab10@math.iisc.ernet.in.}

\medskip

\date{September 13, 2015}
\end{center}

\hrule

\begin{abstract}
We have defined the weight of the pair $(\langle S \mid R \rangle, R)$ 
for a given presentation $\langle S \mid R \rangle$ of a group, where the number of generators is equal to the number of relations. We present an algorithm to construct crystallizations of 3-manifolds whose fundamental group has a presentation with two
generators and two relations. If the weight of $(\langle S \mid R \rangle, R)$ is $n$
then our algorithm constructs all the $n$-vertex crystallizations which yield $(\langle S \mid R \rangle, R)$. As an application, we have constructed some new crystallizations of 3-manifolds.

We have generalized our algorithm for presentations with three generators and a certain class of relations. For $m\geq 3$ and $m \geq n \geq k \geq 2$, our generalized algorithm gives a $2(2m+2n+2k-6+\delta_n^2 + \delta_k^2)$-vertex crystallization of
the closed connected orientable $3$-manifold $M\langle m,n,k \rangle$ having fundamental group $\langle x_1,x_2,x_3 \mid x_1^m=x_2^n=x_3^k=x_1x_2x_3 \rangle$.
These crystallizations are minimal and unique with respect to the given presentations. If `$n=2$'  or `$k\geq 3$ and $m \geq 4$' then our crystallization of $M\langle m,n,k \rangle$
is vertex-minimal for all the known cases.

\end{abstract}

\noindent {\small {\em MSC 2010\,:} Primary 57Q15. Secondary 05C15, 57N10, 57Q05.

\noindent {\em Keywords:} Pseudotriangulations of manifolds, Crystallizations of manifolds, Spherical and hyperbolic $3$-manifolds, Presentations of groups.
}

\medskip

\hrule

\section{Introduction}
For $d\geq 1$, a $(d+1)$-colored graph $(\Gamma,\gamma)$ represents a pure $d$-dimensional simplicial cell 
complex $\mathcal{K}(\Gamma)$ which has $\Gamma$ as dual graph. 
For a certain class of such graphs, the underlying space $|\mathcal{K}(\Gamma)|$ is a closed connected $d$-manifold. In such case, the $(d+1)$-colored graph $(\Gamma,\gamma)$
is called a crystallization of the $d$-manifold $|\mathcal{K}(\Gamma)|$. In \cite{pe74}, Pezzana showed existence of crystallizations for each 
closed connected PL manifold. In \cite{ga79b}, Gagliardi introduced an algorithm to find a presentation of the fundamental group of a closed connected
$d$-manifold $M$ from a crystallization of $M$. The components of the graph restricted over two colors give the relations, and
the components of the graph restricted over remaining colors give the generators of the presentation. In \cite{ep61}, Epstein proved that the fundamental group of a 
$3$-manifold has a presentation which has the number of relations is less than or equal to the number of generators. For a pair $(\langle S \mid R \rangle, R)$ with $\#S=\#R$,
we have defined its weight $\lambda(\langle S \mid R \rangle, R)$ in Definition \ref{def:presentation}.  If $(\Gamma,\gamma)$ is a crystallization of a closed connected orientable $3$-manifold and yields a presentation $(\langle S \mid R\rangle, R)$ then, from Lemma \ref{lemma:vertex}, $\#V(\Gamma) \geq
\lambda ( \langle S \mid R\rangle, R)$. Given a presentation 
$\langle S \mid R\rangle$ with two generators and two relations, our aim is to construct all  crystallizations which yield $(\langle S \mid R\rangle, R)$ and have $\lambda ( \langle S \mid R\rangle, R)$ vertices. 

For such a presentation $\langle S \mid R\rangle$ of a group, 
we have presented an algorithm 
(Algorithm 1 in Subsection \ref{subsec:algorithm1}) which gives all crystallizations such that the crystallizations yield the pair $(\langle S \mid R \rangle, R)$ and are minimal (cf. Definition \ref{def:minimal})
with respect to the pair $(\langle S \mid R \rangle, R)$. In particular, the algorithm determines whether such a crystallization exists or not. 

Let $M \not \cong L(p,q)$ be a closed connected orientable prime $3$-manifold and the fundamental group of $M$ has a presentation $\langle S \mid R\rangle$ with two generators and two relations. Using  Algorithm 1, we have constructed all possible crystallizations of $M$ which yield  $(\langle S \mid R \rangle, R)$ and are minimal with respect to the pair
$(\langle S \mid R \rangle, R)$ (cf. Theorem \ref{theorem:algorithm}). If $M \cong L(p,q)$ then the algorithm gives all such crystallizations of  $L(p,q^{\prime})$ for some $q^{\prime} \in \{1, \dots, p-1\}$.

As an application of Algorithm 1, we have constructed such crystallizations of some closed connected orientable $3$-manifolds, namely, $M\langle m,n,2 \rangle$ for all $m,n\geq 2$, 
lens spaces and a hyperbolic $3$-manifold. (Here $M\langle m,n,k \rangle$ is as in Subsection \ref{subsec:quaternion}.)

We have also generalized this algorithm for presentations with three generators and a certain class of relations (Algorithm 2 in Subsection \ref{subsec:algorithm2}). 
As an application of this, we have constructed a $4(m+2)$-vertex crystallization of the generalized quaternion space $S^3/Q_{4m}$
and a $4(m+n+k-3)$-vertex crystallization of the $3$-manifold $M\langle m,n,k \rangle$ for $m,n,k \geq 3$. (Here $S^3/Q_{4m}$ is as in Subsection \ref{subsec:quaternion}.) For $(m,n,k) \neq (3,3,3)$, these crystallizations are
vertex-minimal, when the number of vertices are at most $30$. In fact, there are no known crystallizations of these manifolds which have less number of vertices than our constructed ones (cf. Remark \ref{remark:minimal}). We have also constructed  a $(4m+4n-2)$-vertex crystallization of the $3$-manifold $M\langle m,n,2 \rangle$ for $m,n \geq 3$.
The crystallizations of the $3$-manifolds  $S^3/Q_{4m}$, $M\langle m,n,2 \rangle$ and $M\langle m,n,k \rangle$ for $m,n,k \geq 3$ are minimal and unique with respect to the given presentations (cf. Theorems \ref{theorem:q4n}, \ref{theorem:mn2} and  \ref{theorem:mnk}).

\section{Preliminaries}

\subsection{Colored graphs}
A multigraph $\Gamma = (V(\Gamma), E(\Gamma))$ is a  finite connected graph which can have multiple edges but no loops, where $V(\Gamma)$ and $E(\Gamma)$ denote the sets of
vertices and edges of $\Gamma$ respectively.
For $n\geq 1$, an $n$-path is a tree with $(n+1)$ distinct vertices and $n$ edges. 
If $a_i$ and $a_{i+1}$ are adjacent in an $n$-path for $1\leq i\leq n$ then the $n$-path is denoted by $P_{n}(a_1, a_2, \dots, a_{n+1})$.
For $n\geq 2$, an $n$-cycle is a closed path with $n$ distinct vertices and $n$ edges. 
If vertices $a_i$ and $a_{i+1}$ are adjacent in an $n$-cycle for $1\leq i\leq n$ (addition is modulo $n$) then the $n$-cycle 
is denoted by $C_n(a_1, a_2, \dots, a_n)$. By $kC_n$ we mean a graph consists of $k$ disjoint $n$-cycles. The disjoint union of the graphs $G$ and $H$ is denoted by $G \sqcup H$. A graph $\Gamma$ is called {\em $(d+1)$-regular} if the number of edges adjacent to each vertex is $(d+1)$. 

First we call $\Delta_d=\{0,1,\dots,d\}$ the color set.
An {\em edge coloring} with $(d+1)$ colors on the graph $\Gamma = (V(\Gamma), E(\Gamma))$ is a map $\gamma \colon E(\Gamma) \to \Delta_d$ such that $\gamma(e) \neq
\gamma(f)$ whenever $e$ and $f$ are adjacent (i.e., $e$ and $f$ are adjacent to a common vertex).
An $(d+1)$-colored graph is a pair $(\Gamma, \gamma)$ where  $\Gamma$ is a  multigraph and $\gamma$ is a 
edge coloring on the graph $\Gamma$ with $(d+1)$ colors. Two vertices are called $i$-adjacent to each other if they are joined by an edge of color $i$. 

Let $(\Gamma,\gamma)$ be a $(d+1)$-colored connected graph with color set $\Delta_d$. If $B \subseteq \Delta_d$ with $k$ elements then the
graph $(V(\Gamma), \gamma^{-1}(B))$ is a $k$-colored graph with coloring $\gamma|_{\gamma^{-1}(B)}$. This colored
graph is denoted by $\Gamma_B$. If $\Gamma_{\Delta_d \setminus \{c\}}$ is connected for all $c\in \Delta_d$ then  $(\Gamma,\gamma)$ is called {\em contracted}.
For standard terminology on graphs see \cite{bm08}.

\subsection{Spherical and hyperbolic 3-manifolds}

A  $3$-manifold $M$ is called a {\em spherical $3$-manifold} if $M \cong S^3/\Gamma$ where $\Gamma$ is a finite subgroup of $SO(4)$ 
acting freely by rotations on the $3$-sphere $S^3$. Therefore, spherical $3$-manifolds are prime, orientable and closed. 
Spherical $3$-manifolds are sometimes called {\em elliptic $3$-manifolds} or {\em Clifford-Klein $3$-manifolds}. In \cite[Chapter 3]{th80}, Thurston
conjectured that a closed 3-manifold with finite fundamental group is spherical, which is also known as {\em elliptization conjecture}.
In \cite{pe03}, Perelman proved  the elliptization conjecture.

Consider the 3-sphere $S^{\hspace{.2mm}3} = \{(z_1, z_2) \in \mathbb{C}^{\hspace{1mm}2} \, : \, |z_1|^2 + |z_2|^2
= 1\}$. Let $p$ and $q$ be relatively prime integers. Then the action of $\mathbb{Z}_p = \mathbb{Z}/p\mathbb{Z}$
on $S^3$ generated by $[1].(z_1, z_2) := (e^{2\pi i/p} \cdot z_1, e^{2\pi iq/p} \cdot z_2)$ is free and hence
properly discontinuous. Therefore the quotient space $L(p, q) := S^{\hspace{.3mm}3}/\mathbb{Z}_p$ is a 3-manifold
whose fundamental group is isomorphic to $\mathbb{Z}_p$. The 3-manifolds $L(p, q)$ are called the lens spaces. It
is a classical theorem of Reidemeister that $L(p, q^{\hspace{.3mm}\prime})$ is homeomorphic to $L(p, q)$ if and
only if $q^{\hspace{.3mm}\prime} \equiv \pm q^{\pm 1}$ (mod $p$).

A $3$-manifold is called a {\em hyperbolic $3$-manifold} if it is equipped with a complete Riemannian metric of constant sectional curvature $-1$.
In other words, it is the quotient of three-dimensional hyperbolic space by a subgroup of hyperbolic isometries acting freely and properly discontinuously.
From \cite[Theorem 2.2]{afw13}, we know the following.

\begin{prop} \label{prop:hyperbolic}
Let M and N be two orientable, closed, prime $3$-manifolds and let $\varphi : \pi_1(M, \ast) \to \pi_1(N, \ast)$ be an isomorphism. 
\begin{enumerate}[{\rm (i)}]
\item If M and N are not lens spaces then M and N are homeomorphic.
\item If M and N are not spherical then there exists a homeomorphism which induces $\varphi$.
\end{enumerate}

\end{prop}

\subsection{Weights of  presentations of groups}\label{subsec:presentation}

Given a set $S$, let $F(S)$ denote the free group generated by $S$. So, any element $w$ of $F(S)$ is of the form
$w = x_1^{\varepsilon_1} \cdots x_m^{\varepsilon_m}$, where $x_1, \dots, x_m\in S$ and $\varepsilon_i = \pm 1$
for $1\leq i\leq m$ and $(x_{j+1}, \varepsilon_{j+1}) \neq (x_j, -\varepsilon_j)$ for $1\leq j \leq m-1$. For $R
\subseteq F(S)$, let $N(R)$ be the smallest normal subgroup of $F(S)$ containing $R$. Then, the quotient group
$F(S)/N(R)$ is denoted by $\langle S \, | \, R \rangle$. For a presentation $P=\langle S \, | \, T \rangle$ with  $N(T) = N(R)$,
the pair $(P,R)$ denotes the presentation $P$ with the relation set $R$. So, if $T \neq R$ and $N(T) = N(R)$ then $\langle S \, | \, T \rangle = \langle S \, | \, R
\rangle$ but as a pair $(\langle S \, | \, T \rangle,T) \neq (\langle S \, | \, R\rangle, R)$.  Two elements $w_1, w_2\in F(S)$ are said to be {\em
independent} (resp., {\em dependent}) if $N(\{w_1\}) \neq N(\{w_2\})$ (resp., $N(\{w_1\}) = N(\{w_2\})$).

For a finite subset $R$ of $F(S)$, let
\begin{align}\label{R}
\overline{R} := \{w \in N(R) \, : \, N((R \setminus \{r\}) \cup \{w\}) = N(R) \mbox{ for each } r\in R\}.
\end{align}
Observe that $\overline{\emptyset} = \emptyset$ and if $R\neq \emptyset$ is a finite set then $w :=\prod_{r\in R}
r \in \overline{R}$ and hence $\overline{R} \neq \emptyset$.

For $w = x_1^{\varepsilon_1}\cdots x_m^{\varepsilon_m} \in F(S)$, $m\geq 1$, let

\begin{align}
\varepsilon(w) :=
\left\{ \begin{array}{lcl}
0 & \mbox{if} & m = 1, \\
|\varepsilon_1-
\varepsilon_2| + \cdots + |\varepsilon_{m-1}-\varepsilon_m| +
|\varepsilon_m - \varepsilon_1| & \mbox{if}
& m \geq 2.
\end{array}\right. \nonumber
\end{align}
Consider the map $\lambda \colon F(S) \to \mathbb{Z}^{+}$ define inductively as follows.
\begin{eqnarray}\label{l(w)}
\lambda(w) := \left\{ \begin{array}{lcl}
2 & \mbox{if} & w = \emptyset, \\
2m-\varepsilon(w) & \mbox{if}
& w = x_1^{\varepsilon_1}\cdots
x_m^{\varepsilon_m}, \, (x_m, \varepsilon_m) \neq (x_1, -\varepsilon_1), \\
\lambda(w^{\hspace{.1mm}\prime})
& \mbox{if} & w = x_1^{\varepsilon_1} w^{\hspace{.1mm}\prime} x_1^{-\varepsilon_1}.
\end{array}\right.
\end{eqnarray}

Since $|\varepsilon_i - \varepsilon_j|=0$ or 2, $\varepsilon(w)$
is an even integer and hence $\lambda(w)$ is
also even. For $w\in F(S)$, $\lambda(w)$ is said to be the {\em weight} of $w$. Observe that $\lambda(w_1w_2) =
\lambda(w_2w_1)$ for $w_1, w_2\in F(S)$. In \cite{ep61}, Epstein proved that the fundamental group of a 3-manifold has a presentation where
the number of relations is less than or equal to the number of generators. Here, we are interested in those presentations $\langle S \mid R\rangle$ for which $\#S=\#R<\infty$.

\begin{defn} \label{def:presentation}
{\rm
 Let $S = \{x_1, \dots, x_s\}$ and $R= \{r_1, \dots, r_s\} \subseteq F(S)$. Let $r_{s+1}$ be an
element in $\overline{R}$ of minimum weight. Then, the number
$$\lambda(\langle S \mid R\rangle, R) := \lambda(r_1) +
\cdots + \lambda(r_s) + \lambda(r_{s+1}).$$
is called the  {\em weight} of the pair $(\langle S \mid R\rangle, R)$.}
\end{defn}

Let $w=\alpha_1^{\varepsilon_1}\alpha_2^{\varepsilon_2} \cdots \alpha_m^{\varepsilon_m}  \in  F(S:=\{x_1, \dots ,x_s\})$ where
$\varepsilon_i \in\{+1,-1\}$ for $1\leq i \leq m$. Then, we define 

\begin{enumerate}[{\rm (i)}]
\item $w_{ij}^{(2)}:=$ total number of appearances of $x_i^{-1}x_j$ and $x_j^{-1}x_i$ in $\alpha_m^{\varepsilon_m}\alpha_1^{\varepsilon_1}\alpha_2^{\varepsilon_2} \cdots \alpha_m^{\varepsilon_m}$, for $x_i, x_j \in S$ and $1\leq i<j \leq s$,
\item $w_{ij}^{(3)}:=$ total number of appearances of $x_ix_j^{-1}$ and $x_jx_i^{-1}$ in $\alpha_m^{\varepsilon_m}\alpha_1^{\varepsilon_1}\alpha_2^{\varepsilon_2} \cdots \alpha_m^{\varepsilon_m}$, for $x_i, x_j \in S$ and $1\leq i<j \leq s$,
\item $w_{i\,(s+1)}^{(2)}:=$ total number of appearances of $x_i^{-1}x_j^{-1}$ and $x_jx_i$ in $\alpha_m^{\varepsilon_m}\alpha_1^{\varepsilon_1}\alpha_2^{\varepsilon_2} \cdots \alpha_m^{\varepsilon_m}$, for $x_i, x_j \in S$ and $1 \leq i \neq j \leq s$,
\item $w_{i \,(s+1)}^{(3)}:=$ total number of appearances of $x_j^{-1}x_i^{-1}$ and $x_ix_j$ in $\alpha_m^{\varepsilon_m}\alpha_1^{\varepsilon_1}\alpha_2^{\varepsilon_2} \cdots \alpha_m^{\varepsilon_m}$, for $x_i, x_j \in S$ and $1 \leq i \neq j \leq s$. 
\end{enumerate}
Observe that $\lambda(w)= \sum w_{ij}^{(c)}$ is the sum over $1\leq i < j \leq s+1$ and $2 \leq c \leq 3$.

\subsection{Binary polyhedral groups and generalized quaternion spaces}\label{subsec:quaternion}
A group is called a {\em binary polyhedral group} if it has a presentation of the form $\langle x_1,x_2,x_3 \mid x_1^{m} = x_2^{n} = x_3^{k}=x_1x_2x_3 \rangle$ 
for some integer $m,n, k \geq 2$. This group is denoted by $\langle m,n,k \rangle$. This group is known to be fundamental group of the $3$-manifold $\mathcal{L}/(m,n,k)$, where $\mathcal{L}$ is the connected Lie group of orientation preserving
isometries of a plane $P$ (cf. \cite{li95, mil75}) and $(m,n,k)= \langle x_1,x_2 ,x_3 \mid x_1^{m} = x_2^{n} = x_3^{k}= 1\rangle$.
Since $\langle m,n,k \rangle$ is not a free product and not isomorphic to $\mathbb{Z}_p$, any $3$-manifold which has fundamental group $\langle m,n,k \rangle$,
is prime and not homeomorphic to lens space. Therefore, by Proposition \ref{prop:hyperbolic}, any two closed connected orientable manifolds with same
fundamental group $\langle m,n,k \rangle$, are homeomorphic. In this article, we will denote such a $3$-manifold by $M\langle m,n,k \rangle$. Observe that
$M\langle \tilde{m},\tilde{n},\tilde{k} \rangle \cong M\langle m,n,k \rangle$ for every permutation $\tilde{m}\tilde{n}\tilde{k}$ of $mnk$. Thus,
we can assume that $m \geq n \geq k$. Clearly, the group $\langle x_1,x_2 \mid x_1^{m}x_2^{-n},x_2^{nk-n-k}x_1^{-k}, x_1x_2x_1^{-1}x_2^{-1}\rangle$ 
is isomorphic to the abelianized group of $\langle m,n,k\rangle$. Therefore, $(m,n,k)=(5,3,2)$ or $(7,3,2)$ implies that abelianization of $\langle m,n,k \rangle$ is trivial.
Thus, $M\langle 5,3,2 \rangle$ and $M\langle 7,3,2 \rangle$ are homology spheres, in fact, $M\langle 5,3,2 \rangle$ is the Poincar\'{e} homology sphere.
Since $\langle m,2,2 \rangle ~ (\cong Q_{4m})$, $P_{24}:=\langle 3,3,2 \rangle$, $P_{48}:=\langle 4,3,2 \rangle$ and $P_{120}:= \langle 5,3,2 \rangle$ are finite groups, by the proof of
elliptization conjecture of Perelman, $M\langle m,n,k \rangle$ is spherical, i.e.,  $M\langle m,n,k \rangle \cong S^3/\langle m,n,k \rangle$ for these groups  $\langle m,n,k \rangle$. It is not difficult to prove that, the abelianization of $\langle m,n,k \rangle \cong \mathbb{Z} \oplus H$ for some group $H$ if and only if
$(m,n,k)=(6,3,2)$, $(4,4,2)$ or $(3,3,3)$. Therefore, in these three cases,
the $3$-manifold $M\langle m,n,k \rangle$ has a handle and in all the other cases, $M\langle m,n,k \rangle$ is handle-free.

A group is called a {\em generalized quaternion group} or {\em dicyclic group} if it has a presentation of the form
$\langle x_1,x_2 \mid x_1^{2m} = x_2^4 = 1, x_1^m = x_2^2, x_2^{-1}x_1x_2 = x_1^{-1}\rangle$ for some integer $m \geq 2$. 
This group has order $4m$ and is denoted by $Q_{4m}$.

\noindent{\em Claim:} For $m \geq 2$, $Q_{4m}$ has a presentation
$\langle S \mid R \rangle$, where $S =\{x_1,x_2, x_3\}$ and $R=\{x_1^{m-1}x_3^{-1}x_2^{-1}, x_2x_1^{-1}x_3^{-1}, x_3x_2^{-1}x_1^{-1}\}$.

Observe that $x_2x_1^{-1}x_3^{-1}=1=x_3x_2^{-1}x_1^{-1}$ implies $x_2x_1^{-1}= x_3 =x_1x_2$, i.e., $x_1^{-1}=x_2^{-1} x_1x_2$.
Again, $x_1^{m-1}x_3^{-1}x_2^{-1}=1$ and $x_2x_1^{-1}x_3^{-1}=1$ implies that $x_1^m=x_2^2$. 
Now, $x_2=x_1x_2x_1=x_1(x_1x_2x_1)x_1=x_1^2x_2x_1^2=\cdots=x_1^mx_2x_1^m=x_2^2x_2x_2^2=x_2^5$ implies
$x_2^4=1$. Since $x_1^m=x_2^2$, $x_1^{2m}=x_2^4=1$. Thus, $\langle S \mid R \rangle \cong \langle x_1,x_2 \mid x_1^{2m} = x_2^4 = 1, x_1^m = x_2^2, x_2^{-1}x_1x_2 = x_1^{-1}
\rangle$. 
This proves the claim.

The $3$-manifold $M\langle m,2,2 \rangle$ is called {\em generalized quaternion space}.
Then, by the proof of elliptization conjecture of Perelman, $M\langle m,2,2 \rangle$ is spherical and homeomorphic to $S^3/Q_{4m}$.

\subsection{Crystallizations} \label{crystal}

A CW-complex $X$ is said to be {\em regular} if the attaching maps which define the incidence structure of $X$ are homeomorphisms. Given a regular CW-complex $X$, let ${\mathcal X}$ be the set of all closed cells of $X$ together with the empty set. Then, ${\mathcal X}$ is a poset, where the partial ordering is the set inclusion. This poset ${\mathcal X}$ is said to be the {\em face poset} of $X$. Clearly, if $X$ and $Y$ are two finite regular CW-complexes with isomorphic face posets then $X$ and $Y$ are homeomorphic. A regular CW-complex $X$ is said to be {\em simplicial} if the boundary of each cell in $X$ is isomorphic (as a poset) to the boundary of a simplex of same dimension. 
A {\em simplicial cell complex} $K$ of dimension $d$ is a poset, isomorphic to the face poset ${\mathcal X}$ of a $d$-dimensional simplicial CW-complex $X$. The topological space $X$ is called the {\em geometric carrier} of $K$ and is also denoted by $|K|$. If a topological space $M$ is homeomorphic to $|K|$, then 
$K$ is said to be a {\em pseudotriangulation} of $M$. 

Let $K$ be a simplicial cell complex with partial ordering $\leq$. If $\beta\leq \alpha\in K$ 
then we say $\beta$ is a {\em face} of $\alpha$. If all the maximal cells of a $d$-dimensional simplicial cell complex $K$ are $d$-cells then it is called {\em pure}. 
Maximal cells in a pure simplicial cell complex $K$ are called the {\em facets} of $K$.
The 0-cells in a simplicial cell complex $K$ are said to be the {\em vertices} of $K$. If $u$ is a face of $\alpha$ and $u$ is a vertex then we say $u$ 
is a {\em vertex of} $\alpha$. 
Clearly, a $d$-dimensional simplicial cell complex $K$ has at least $d+1$ vertices. If  a $d$-dimensional simplicial cell complex $K$ has exactly $d+1$ 
vertices then $K$ is called {\em contracted}.

Let $K$ be a pure $d$-dimensional simplicial cell complex. Consider the graph 
$\Lambda(K)$ whose vertices are the facets of $K$ and  whose edges are the ordered pairs 
$(\{\sigma_1, \sigma_2\}, \gamma)$, where $\sigma_1$, $\sigma_2$ are facets, $\gamma$ is a $(d-1)$-cell and is a common face of $\sigma_1$, $\sigma_2$. The graph $\Lambda(K)$ is said to be the {\em dual graph} of $K$. Observe that $\Lambda(K)$ is in general a multigraph without loops. On the other hand, for $d\geq 1$, 
if $(\Gamma, \gamma)$ is a $(d+1)$-colored graph with color set $\Delta_d = \{0, \dots, d\}$ then we define a $d$-dimensional simplicial cell complex ${\mathcal K}(\Gamma)$ as follows. For each $v\in V(\Gamma)$, we take a $d$-simplex $\sigma_v$ and label its vertices by $0, \dots, d$. If $u, v \in V(\Gamma)$ are joined by an edge $e$ and $\gamma(e) =  {i}$, then we identify the $(d-1)$-faces of $\sigma_u$ and $\sigma_v$ opposite to the vertices labeled by ${i}$, so that equally labeled vertices are identified together. Since there is no identification within a $d$-simplex, this gives a simplicial CW-complex $W$ of dimension $d$. So, the face poset (denoted by ${\mathcal K}(\Gamma)$) of $W$ is a pure $d$-dimensional simplicial cell complex. We say that $(\Gamma, \gamma)$ {\em represents} the simplicial cell complex ${\mathcal K}(\Gamma)$. Clearly, the number of $i$-labeled vertices of ${\mathcal K}(\Gamma)$ is equal to the number of components of $\Gamma_{\Delta_d\setminus\{{i}\}}$ for each $ {i}\in \Delta_d$. Thus, the simplicial cell complex ${\mathcal K}(\Gamma)$ is contracted if and only if $\Gamma$ is contracted  (cf. \cite{fgg86}).

A {\em crystallization} of a connected closed $d$-manifold $M$ is a $(d+1)$-colored contracted graph 
$(\Gamma, \gamma)$ such that the simplicial cell complex ${\mathcal K}(\Gamma)$ is a pseudotriangulation of $M$. Thus, if $(\Gamma, \gamma)$ is a 
crystallization of a $d$-manifold $M$ then the number of vertices in ${\mathcal K}(\Gamma)$ is $d+1$. On the other hand, if $K$ is a contracted pseudotriangulation 
of $M$ then the dual graph $\Lambda(K)$ gives a crystallization of $M$. Clearly, if $(\Gamma, \gamma)$ is a crystallization of a closed $d$-manifold $M$ then, either 
$\Gamma$ has two vertices (in which case $M$ is $S^d$) or the number of edges between two vertices is at most $d-1$. 
In \cite{pe74}, Pezzana showed the following.

\begin{prop}[Pezzana] \label{prop:pezzana74}
Every connected closed PL manifold admits a crystallization.
\end{prop}

Thus, every connected closed PL $d$-manifold has a {\em contracted} pseudotriangulation, i.e., a
pseudotriangulation with $d+1$ vertices. From  \cite{cgp80}, we know the following.

\begin{prop}[Cavicchioli-Grasselli-Pezzana] \label{prop:ca-ga-pe80}
Let $(\Gamma,\gamma)$ be a crystallization of a PL manifold $M$. Then $M$ is orientable if and only if 
$\Gamma$ is bipartite.
\end{prop}

Let $\Delta_d=\{0, \dots, d\}$ be the color set of a $(d+1)$-colored graph $(\Gamma, \gamma)$. For $0\leq i \neq j\leq d$, $g_{ij}$ denote the number of connected components of the graph $\Gamma_{\{i,j\}}$. In \cite{ga79a}, Gagliardi proved the following.

\begin{prop}[Gagliardi] \label{prop:gagliardi79a}
Let $(\Gamma,\gamma)$ be a contracted $4$-colored graph with the color set $\Delta_3$. Then, $(\Gamma,\gamma)$ is a
crystallization of a connected closed $3$-manifold if and only if
\begin{enumerate}[{\rm (i)}]
\item
$g_{ij}=g_{kl}$ for  $\{i,j,k,l\}=\Delta_3$, and
\item $g_{01}+g_{02}+g_{03}=2+ \frac{\#V(\Gamma)}{2}$.
\end{enumerate}
\end{prop}

Let  $(\Gamma, \gamma)$ be a crystallization (with color set $\Delta_d$) of a connected closed $d$-manifold $M$. So,
$\Gamma$ is a $(d+1)$-regular graph. Choose two colors, say, $i$ and $j$ from $\Delta_d$. Let $\{G_1, \dots, G_{s+1}\}$
be the set of all connected components of $\Gamma_{\Delta_d\setminus \{i,j\}}$ and $\{H_1, \dots, H_{t+1}\}$ be the set
of all connected components of $\Gamma_{\{i,j\}}$. Since $\Gamma$ is regular, each $H_p$ is an even cycle. Note
that, if $d=2$ then $\Gamma_{\{i,j\}}$ is connected and hence $H_1= \Gamma_{\{i,j\}}$. Consider a set $\widetilde{S} =
\{x_1, \dots, x_s, x_{s+1}\}$ of $s+1$ elements. For $1\leq k\leq t+1$, consider the word $\tilde{r}_k$ in
$F(\widetilde{S})$ as follows. Choose a starting vertex $v_1$ in $H_k$. Let $H_k = v_1 e_{1}^i v_2 e_{2}^j v_3 e_{3}^i
v_{4} \cdots e_{2l-1}^i v_{2l}e_{2l}^jv_1$, where $e_{p}^i$ and  $ e_{q}^j$ are edges with colors $i$ and $j$
respectively. Define
\begin{align} \label{tildar}
\tilde{r}_k := x_{k_2}^{+1} x_{k_3}^{-1}x_{k_4}^{+1}  \cdots
x_{k_{2l}}^{+1}x_{k_1}^{-1},
\end{align}
where $G_{k_h}$ is the component of $\Gamma_{\Delta_d\setminus \{i,j\}}$ containing $v_h$.  For $1\leq k\leq t+1$, let
$r_k$ be the word obtained from $\tilde{r}_k$ by deleting $x_{s+1}^{\pm 1}$'s in $\tilde{r}_k$. So, $r_k$ is a
word in $F(S)$, where $S = \widetilde{S}\setminus \{x_{s+1}\}$. In \cite{ga79b}, Gagliardi proved the following.

\begin{prop}[Gagliardi] \label{prop:gagliardi79b}
For $d\geq 2$, let  $(\Gamma, \gamma)$ be a crystallization of a connected closed PL manifold $M$. For two
colors $i, j$, let $s$, $t$, $x_p$, $r_q$ be as above. If $\pi_1(M, x)$ is the fundamental group of $M$ at a
point $x$, then
$$
\pi_1(M, x) \cong \left\{ \begin{array}{lcl}
\langle {x_1, x_2,\dots, x_s} ~ | ~ {r_1} \rangle & \mbox{if}
& d=2,   \\
\langle {x_1, x_2, \dots, x_s} ~ | ~ {r_1, \dots, r_t} \rangle
& \mbox{if} & d\geq 3.
\end{array}\right.
$$
\end{prop}

In this case, we will say $(\Gamma,\gamma)$ yields $(\langle S \mid R \rangle, R)$, where $S=\{x_1,\dots, x_s\}$ and $R=\{r_1, \dots, r_t\}$.
From Proposition \ref{prop:gagliardi79a}, it is clear that, if $(\Gamma,\gamma)$ is a crystallization of a $3$-manifold then $s=t$. 
Note that, there may have a relation $r \in R$ such that $r \in \overline{R \setminus \{r\}}$ and in that case $\langle S \mid R \rangle \cong \langle S \mid R \setminus \{r\} \rangle$.

\begin{lemma}\label{lemma:vertex}
If $(\Gamma,\gamma)$ is a crystallization of a $3$-manifold such that $(\Gamma,\gamma)$ yields $(\langle S \mid R \rangle, R)$ then $\#V(\Gamma)\geq \lambda(\langle S \mid R \rangle, R)$.
\end{lemma}

\begin{proof}
Since the crystallization yields $(\langle S \mid R \rangle, R)$, from the above discussion, we know the 
crystallization yields the relations in $R\cup\{w\}$ where
$w \in \overline{R}$. Thus, the lemma follows from the construction of $\tilde{r}$  as in Eq. \eqref{tildar}
and Definition \ref{def:presentation}.
\end{proof}

\begin{defn} \label{def:minimal}
{\rm
A crystallization $(\Gamma,\gamma)$ of a $3$-manifold is called {\em minimal with respect to the pair $(\langle S \mid R \rangle, R)$} if 
$\#V(\Gamma) = \lambda(\langle S \mid R \rangle, R)$.}
\end{defn}

\section{Constructions of crystallizations and Algorithm 1}\label{sec:algorithm}
Here we are interested in orientable $3$-manifolds. Thus, by Proposition \ref{prop:ca-ga-pe80}, corresponding crystallizations are bipartite. We use black dots `$\bullet$' for vertices in one part and white dots `$\circ$' for vertices in the other part of the crystallizations.

\subsection{Constructions} 
Let $\langle S \mid R\rangle = \langle x_1,x_2 \mid r_1, r_2\rangle$ be a presentation of a group. Now we 
construct all possible crystallizations from the pair $(\langle S \mid R \rangle, R)$ by the following steps.

\medskip 

\noindent\textbf{Step 1:} 
If a crystallization yields $(\langle S \mid R \rangle, R)$ then the crystallization yields the relations in $R\cup\{w\}$, where $w$ is an element
$\overline{R}$. Since we are interested in $\lambda(\langle S \mid R\rangle, R)$-vertex crystallizations, $w \in \overline{R}$
is of minimum weight. Let $\{w_i \in \overline{R}, 1\leq i \leq k\}$ be the set of all  independent words with minimum weight (as there are only finite number of independent
words in $\overline{R}$ with minimum weight). 
Let $R_{i} = R \cup \{w_i\}$ for $1\leq i \leq k$. 
For each $\mathcal{R} \in \{R_1,\dots, R_k\}$, we will construct all possible crystallizations which yield the relations in $\mathcal{R}$. Choose a $\mathcal{R} \in \{R_1,\dots, R_k\}$.
   
\medskip

\noindent\textbf{Step 2:} If possible, let $(\Gamma,\gamma)$ be a crystallization of a $3$-manifold which yields the relations in $\mathcal{R}$. 
Without loss of generality, let $G_1, G_2, G_3$  be the components of $\Gamma_{\{0,1\}}$ such that $G_i$ represents the generator
$x_i$ for $1 \leq i \leq 2$ and $G_3$ represents $x_3$ (cf. Eq. \eqref{tildar} for construction of $\tilde{r}$). Let $n_i$ be the total number of appearance of $x_i$ in 
the three relations in $\mathcal{R}$ for $1\leq i \leq 2$ and $n_3=\lambda(\langle S \mid R \rangle, R)-(n_1+n_2)$. Then, the total number of vertices in $G_i$ is $n_i$ and let $G_i=C_{n_{i}}(x^{(1)}_i, \dots, x^{(n_i)}_i)$ for $1\leq i \leq 3$.
Clearly, each $n_i$ is even and $n_{1}+n_{2}+n_{3}=\#V(\Gamma)$. Without loss of generality, we can assume that $x^{(2j-1)}_ix^{(2j)}_i \in \gamma^{-1}(1)$ and 
$x^{(2j)}_ix^{(2j+1)}_i \in \gamma^{-1}(0)$
with $x^{(n_i+1)}_i=x^{(1)}_i$ for $1 \leq j \leq n_i/2$ and $1\leq i \leq 3$. Here and after, the additions and subtractions
at the point `$\ast$' in $x_i^{(\ast)}$ are modulo $n_i$ for $1\leq i \leq 3$.

\medskip

\noindent\textbf{Step 3:} From the fact $\#V(\Gamma)=\lambda(\langle S \mid R \rangle, R)$, we know that $\#V(\Gamma)$ is always even and there is no $2$-cycle in $\Gamma_{\{i,j\}}$ for $0\leq i \leq 1$ and 
$2\leq j \leq 3$. If $\#V(\Gamma)$ is of the form $4n$ for some $n\in \mathbb{N}$, then $g_{01}+g_{02}+g_{03}=2n+2$. 
This implies, $g_{02}+g_{03}=2n-1$.
Without loss of generality, consider $g_{02}=g_{13}=n-1$ and $g_{03}=g_{12}=n$. Therefore, all the components of $\Gamma_{\{0,3\}}$ (resp., $\Gamma_{\{1,2\}}$) are $4$-cycles. 
But, there are two choices for $\Gamma_{\{0,2\}}$ (resp., $\Gamma_{\{1,3\}}$). Either $\Gamma_{\{0,2\}}$ has one $8$-cycle and remaining $4$-cycles or
$\Gamma_{\{0,2\}}$ has two $6$-cycles and remaining $4$-cycles. Similar arguments hold for $\Gamma_{\{1,3\}}$. On the other hand, if $\#V(\Gamma)$ is of the form $4n+2$ for
some $n\in \mathbb{N}$, then $g_{01}+g_{02}+g_{03}=2n+3$. This implies, $g_{02}+g_{03}=2n$ and hence $g_{02}=g_{13}=g_{03}=g_{12}=n$. 
Therefore, $\Gamma_{\{i,j\}}$ has one $6$-cycle and remaining all $4$-cycles for $0\leq i \leq 1$ and $2\leq j \leq3$.

\medskip

\noindent\textbf{Step 4:} Since components of $\Gamma_{\{2,3\}}$ yield the relations in $\mathcal{R}$, without loss of generality,
let the colors $2$ and $3$ be the colors `$i$' and `$j$' respectively as in the construction of $\tilde{r}$ for $r\in \mathcal{R}$ (cf. Eq. \eqref{tildar}).
Let $m^{(c)}_{ij}:= \sum_{w\in \mathcal{R}} w^{(c)}_{ij}$ for $1\leq i < j \leq 3$ and $2 \leq c \leq 3$, where  $w^{(c)}_{ij}$ as in 
Subsection \ref{subsec:presentation}.
Then, the number of edges of color $c$ between $G_i$ and $G_j$ is $m_{ij}^{(c)}$
for $2\leq c \leq 3$ and $1 \leq i < j \leq 3$. Therefore, $2(m_{12}^{(c)}+m_{13}^{(c)}+m_{23}^{(c)}) =\#V(\Gamma)$ for $2\leq c \leq 3$.
Thus, $g_{01}+g_{02}+g_{03}=\#V(\Gamma)/2+2=m_{12}^{(c)}+m_{13}^{(c)}+m_{23}^{(c)}+2$. Since $g_{01}=3$, $g_{02}=g_{13}$ and 
$g_{03}=g_{12}$, we have $g_{0c}+g_{1c}=m_{12}^{(c)}+m_{13}^{(c)}+m_{23}^{(c)}-1$ for $2\leq c \leq 3$.

\medskip

\noindent\textbf{Step 5:} Since $\Gamma_{\{0,1,c\}}$ is connected for $2\leq c \leq 3$, $\#\{m_{ij}^{(c)} \geq 1, 1 \leq i < j \leq 3\}\geq 2$.

\noindent Case 1: Let $\#\{m_{ij}^{(c)} \geq 1, 1 \leq i < j \leq 3\}=3$, i.e., $m_{ij}^{(c)} \geq 1$, where $1 \leq i < j \leq 3$ for some $c\in\{2,3\}$. 
Then, the maximum number of bi-colored $4$-cycles in $\Gamma_{\{0,1,c\}}$
with two edges of color $c$ is $(m_{12}^{(c)}-1)+(m_{13}^{(c)}-1)+(m_{23}^{(c)}-1)= m_{12}^{(c)}+m_{13}^{(c)}+m_{23}^{(c)}-3$.
Since $g_{0c}+g_{1c}=m_{12}^{(c)}+m_{13}^{(c)}+m_{23}^{(c)}-1$, from the arguments in Step 3, $\Gamma_{\{0,1,c\}}$ must have $m_{12}^{(c)}+m_{13}^{(c)}+m_{23}^{(c)}-3$ 
 bi-colored $4$-cycles and two $6$-cycles with some edges of color $c$.
Therefore, from the arguments in Step 3, if $\#V(\Gamma)=4n$ for some $n \in \mathbb{N}$ then
$\Gamma_{\{0,3\}}$ (resp., $\Gamma_{\{1,2\}}$) is a union of $4$-cycles and $\Gamma_{\{c-2,c\}}$ is
of the form $2C_6 \sqcup (n-3)C_4$.
But, if $\#V(\Gamma)=4n+2$ for some $n \in \mathbb{N}$ then $\Gamma_{\{i,c\}}$ is
of the form $C_6 \sqcup (n-1)C_4$ for $0\leq i\leq1$.

\noindent Case 2: Let  $\#\{m_{ij}^{(c)} \geq 1, 1 \leq i < j \leq 3\}=2$ for some $c\in\{2,3\}$. Then, assume $\{i,j,l\} = \{1,2,3\}$ such that $m_{jl}^{(c)} =0$. 
Then, the maximum number of bi-colored $4$-cycles in $\Gamma_{\{0,1,c\}}$ with two edges of color $c$ is $(m_{ij}^{(c)}-1)+(m_{il}^{(c)}-1)=m_{ij}^{(c)}+m_{il}^{(c)}-2$.
Since $\#V(\Gamma)=\lambda(\langle S \mid R \rangle,R)$ and $m_{jl}^{(c)}=0$, $\Gamma_{\{0,1,c\}}$ does not have a  bi-colored $6$-cycle with some edges of color $c$.
Therefore, $g_{0c}+g_{1c}=m_{ij}^{(c)}+m_{il}^{(c)}-1$ and the arguments in Step 3 implies that $\Gamma_{\{0,1,c\}}$ must have
$m_{ij}^{(c)}+m_{il}^{(c)}-2$  bi-colored $4$-cycles and one $8$-cycle with some edges of color $c$. Thus, $\Gamma_{\{0,3\}}$ (resp., $\Gamma_{\{1,2\}}$) is a union of $4$-cycles and $\Gamma_{\{c-2,c\}}$  is
of the form $C_8 \sqcup (n-2)C_4$. In this case, $\#V(\Gamma)=4n$ for some $n \in \mathbb{N}$. 
Therefore, it is clear that, $m_{ij}^{(c)}$ edges of color $c$ between  $G_i$ and $G_j$ yield $m_{ij}^{(c)}-1$ 
bi-colored $4$-cycles for $1 \leq i < j \leq 3$ and $2\leq c\leq3$.

\medskip 

\noindent\textbf{Step 6:} Now, we will construct  $\Gamma_{\{0,1,2\}}$ and will show that $\Gamma_{\{0,1,2\}}$ is unique up to an isomorphism.
Choose $\{n_i,n_j,n_l\} =\{n_1,n_2,n_3\}$ such that $n_i \geq n_j$ and $n_i \geq n_l$. Clearly, there are edges of color $2$ between each of the pairs $(G_i,G_j)$ and $(G_i,G_l)$. 
Without loss of generality, let $x^{(1)}_ix^{(1)}_l,x^{(n_i)}_ix^{(1)}_j \in \gamma^{-1}(2)$. 

\noindent Case 1: If $m^{(2)}_{jl}=0$ then the path $P_5(x^{(n_j)}_j,x^{(1)}_j,x^{(n_i)}_i,x^{(1)}_i,x^{(1)}_l,x^{(n_l)}_l)$ must be a part of the $8$-cycle of $\Gamma_{\{0,2\}}$. 
Since the $m_{ij}^{(2)}$ (resp., $m_{il}^{(2)}$) edges of color $c$ between the pair $(G_i,G_j)$ (resp., $(G_i,G_l)$) yield $m_{ij}^{(2)}-1$ (resp., $m_{il}^{(2)}-1$) 
bi-colored $4$-cycles, we have
$x^{(1)}_ix^{(1)}_l,\dots, x^{(m_{il}^{(2)})}_ix^{(m_{il}^{(2)})}_l$ and $x^{(n_i)}_ix^{(1)}_j,\dots, x^{(n_i+1-m_{ij}^{(2)})}_ix^{(m_{ij}^{(2)})}_j  \in \gamma^{-1}(2)$.
In this case, $m_{ij}^{(2)}=n_j$, $m_{il}^{(2)}=n_l$ and $n_i=n_j+n_l$.

\noindent Case 2: If $m^{(2)}_{jl}\geq 1$ then $x_j^{(n_j)}x_l^{(n_l)}\in \gamma^{-1}(2)$ completes the $6$-cycle in $\Gamma_{\{0,2\}}$. Therefore,
by similar reasons as above, $\{x^{(1)}_ix^{(1)}_l,\dots, x^{(m_{il}^{(2)})}_ix^{(m_{il}^{(2)})}_l\}$, $\{x^{(n_i)}_ix^{(1)}_j,\dots, x^{(n_i+1-m_{ij}^{(2)})}_ix^{(m_{ij}^{(2)})}_j\}$ 
and $\{x^{(n_j)}_jx^{(n_l)}_l,\dots, x^{(n_j+1-m_{jl}^{(2)})}_jx^{(n_l+1-m_{jl}^{(2)})}_l\}$ are the sets of edges of color $2$.
In this case, $m_{ij}^{(2)}+m_{il}^{(2)}=n_i$, $m_{ij}^{(2)}+m_{jl}^{(2)}=n_j$ and $m_{il}^{(2)}+m_{jl}^{(2)}=n_l$.

Thus, $\Gamma_{\{0,1,2\}}$ is unique up to an isomorphism.
We use white dot `$\circ$' for vertex $x_i^{(1)}$ and black dot `$\bullet$' for vertex $x_i^{(n_i)}$.
Since $\Gamma_{\{0,1,2\}}$ is bipartite graph, $x_i^{(2p_i-1)}, x_j^{(2p_j-1)}$, $x_l^{(2p_l)}$ are denoted by white dots `$\circ$' 
and  $x_i^{(2p_i)}, x_j^{(2p_j)},x_l^{(2p_l-1)}$ are denoted by black dots `$\bullet$' for $1 \leq p_i \leq n_i/2$, $1 \leq p_j \leq n_j/2$ and $1 \leq p_l \leq n_l/2$.

\medskip

\noindent\textbf{Step 7:} Now, we are ready to construct a crystallization $(\Gamma,\gamma)$ which yields the relations in $\mathcal{R}$. For a given set of relations $\mathcal{R}$,
we constructed $\Gamma_{\{0,1,2\}}$ uniquely. By similar arguments as above, we have $m_{ij}^{(3)}, m_{il}^{(3)} \geq 1$.
Choose an edge $x_i^{(2q_i-1)}x_i^{(2q_i)} \in \gamma^{-1}(1)$. Since $1\leq q_i\leq n_i/2$, there are $n_i/2$ choices for 
such an edge.
Now, choose two edges $x_j^{(2q_j-1)}x_j^{(2q_j)}$ and $x_l^{(2q_l-1)}x_l^{(2q_l)} \in \gamma^{-1}(1)$ in $G_j$ and $G_l$ respectively.
Then, either $x_i^{(2q_i-1)}x_j^{(2q_j)}, x_i^{(2q_i)}x_l^{(2q_l)} \in \gamma^{-1}(3)$ or  $x_i^{(2q_i-1)}x_l^{(2q_l-1)}, x_i^{(2q_i)}x_j^{(2q_j-1)} \in \gamma^{-1}(3)$. 
Thus, either $P_5(x_j^{(2q_j-1)}$, $x_j^{(2q_j)}$, $x_i^{(2q_i-1)}$, $x_i^{(2q_i)}$, $x_l^{(2q_l)}, x_l^{(2q_l-1)})$ or 
$P_5(x_j^{(2q_j)}, x_j^{(2q_j-1)}, x_i^{(2q_i)}, x_i^{(2q_i-1)}, x_l^{(2q_l-1)}, x_l^{(2q_l)})$ is a path in $\Gamma_{\{1,3\}}$. Therefore, by  similar arguments as in Step 6,
there is a unique way to choose the remaining edges of color $3$. Since $1\leq q_i\leq n_i/2$, $1\leq q_j\leq n_j/2$ and $1\leq q_l\leq n_l/2$,
we have $2 \times \frac{n_i}{2} \times \frac{n_j}{2} \times {n_l}{2}=\frac{n_1n_2n_3}{4}$ choices for the $4$-colored graph.

\medskip

\noindent\textbf{Step 8:} For a set $\mathcal{R}$ of relations there are  $(n_1n_2n_3)/4$ choices for the $4$-colored graph. If there is a choice, for which $(\Gamma,\gamma)$ yields
the  relations in $\mathcal{R}$ then $(\Gamma,\gamma)$ is a regular bipartite $4$-colored graph which satisfies all the properties of Proposition \ref{prop:gagliardi79a}. 
Therefore, $(\Gamma,\gamma)$ is a
crystallization of a closed connected orientable $3$-manifold $M$ whose fundamental group is $(\langle S \, | \, R\rangle, R)$.
Now, we choose a different $\mathcal{R} \in \{R_1,\dots,R_{k}\}$ and repeat the process from Step 2 to find all possible  
$\lambda(\langle S \, | \, R\rangle, R)$-vertex crystallizations which yield  $(\langle S \, | \, R\rangle, R)$.
If there is no such $4$-colored graph for each $\mathcal{R} \in \{R_1,\dots,R_{k}\}$ then
there is no crystallization of a closed, connected orientable $3$-manifold, which yields $(\langle S \, | \, R\rangle, R)$ and is minimal with respect 
to $(\langle S \, | \, R\rangle, R)$.

\begin{theo}\label{theorem:algorithm}
Let $M$ be a closed connected orientable prime $3$-manifold with fundamental group $\langle S \mid R\rangle$, where $\#S=\#R=2$.
Let $(\Gamma,\gamma)$ be a crystallization constructed from the pair $(\langle S \mid R\rangle, R)$ by using the above construction. Then, we have the following.

\begin{enumerate}[{\rm (i)}]
\item If $\langle S \mid R\rangle \not \cong \mathbb{Z}_p$ then $|\mathcal{K}(\Gamma)| \cong M$.
\item If $\langle S \mid R\rangle \cong \mathbb{Z}_p$ then $|\mathcal{K}(\Gamma)| \cong L(p,q')$ for some $q' \in \{1, \dots, p-1\}$.
\item $(\Gamma,\gamma)$ is minimal with respect to the pair $(\langle S \mid R\rangle, R)$.
\end{enumerate}
\end{theo}

\begin{proof}
Since $(\Gamma,\gamma)$ yields the presentation $\langle S \mid R\rangle$, by Proposition \ref{prop:gagliardi79b}, 
$\pi_1(|\mathcal{K}(\Gamma)|,\ast) \cong$ $\langle S \mid R\rangle$. Since $\Gamma$ is regular and bipartite, $|\mathcal{K}(\Gamma)|$ is a closed connected orientable
$3$-manifold. Since $M$ is prime 3-manifold, the fundamental group of $M$ is not a free product of two groups. Since $M$ and 
$|\mathcal{K}(\Gamma)|$ have same fundamental group, it follows that  $|\mathcal{K}(\Gamma)|$ is also prime 3-manifold. Part (i) now follows from Proposition \ref{prop:hyperbolic}.

If $\pi_1(|\mathcal{K}(\Gamma)|,\ast) \cong \pi_1(M,\ast)\cong \mathbb{Z}_p$ then, by the proof of
elliptization conjecture, $|\mathcal{K}(\Gamma)|$ is spherical and hence $|\mathcal{K}(\Gamma)|\cong S^3/\mathbb{Z}_p\cong L(p,q')$ for some $q' \in \{1, \dots, p-1\}$.
This proves part (ii).

Part (iii) follows from the construction.
\end{proof}

\subsection{Algorithm 1}\label{subsec:algorithm1}
Now, we present our algorithm which finds crystallizations of $3$-manifolds for a presentation $(\langle S \mid R \rangle, R)$ with two generators and two relations, 
such that the crystallizations yield the pair and have the number of vertices is equal to the
weight of the pair $(\langle S \mid R \rangle, R)$.
\begin{enumerate}[(i)]
\item Find the set $\{w_i \in \overline{R}, 1\leq i \leq k\}$ of independent words such that $\lambda(w_i)$ is minimum.
Let $\mathcal{R} = R \cup \{w_1\}$ and consider a class of graphs $\mathcal{C}$  which is empty.

\item For $\mathcal{R}$, $(a)$ find $m_{ij}^{(c)}$ for $2\leq c \leq 3$ and $1 \leq i < j \leq 3$, $(b)$ find $n_1, n_2, n_3$
and $(c)$ choose $i,j,l$ such that $n_i \geq n_j$ and $n_i \geq n_l$.

\item Consider three bi-colored cycles $G_i=C_{n_{i}}(x^{(1)}_i, \dots, x^{(n_i)}_i)$ for $1\leq i \leq 3$ such that $x^{(2j-1)}_ix^{(2j)}_i$ has color $1$ and $x^{(2j)}_ix^{(2j+1)}_i$ has color $0$ with the consideration $x^{(n_i+1)}_i=x^{(1)}_i$ for $1 \leq j \leq n_i/2$ and $1\leq i \leq 3$.

\item  The edges  $x^{(1)}_ix^{(1)}_l,\dots, x^{(m_{il}^{(2)})}_ix^{(m_{il}^{(2)})}_l$ and $x^{(n_i)}_ix^{(1)}_j,\dots, x^{(n_i+1-m_{ij}^{(2)})}_ix^{(m_{ij}^{(2)})}_j$ have color $2$. If $n_j + n_l \neq n_i$ then
the edges $x^{(n_j)}_jx^{(n_l)}_l,\dots, x^{(n_j+1-m_{jl}^{(2)})}_jx^{(n_l+1-m_{jl}^{(2)})}_l$ have also color $2$.

\item For $1\leq q_i\leq n_i/2$, $1\leq q_j\leq n_j/2$ and $1\leq q_l\leq n_l/2$, choose a bi-colored path of colors $1$ and $3$ from the two paths $P_5(x_j^{(2q_j-1)}$, $x_j^{(2q_j)}$, $x_i^{(2q_i-1)}$, $x_i^{(2q_i)}$, $x_l^{(2q_l)}, x_l^{(2q_l-1)})$ and 
$P_5(x_j^{(2q_j)}, x_j^{(2q_j-1)}, x_i^{(2q_i)}, x_i^{(2q_i-1)}, x_l^{(2q_l-1)}, x_l^{(2q_l)})$ and join the remaining vertices by edges of color 3
as there is unique way to choose these edges with the chosen  path. There are $(n_1n_2n_3)/4$ choices for $4$-colored graphs. If some graphs yield $(\langle S \mid R \rangle, R)$ then put them in the class $\mathcal{C}$.

\item If $\mathcal{R} = R \cup\{w_i\}$, for some $i \in \{1, \dots, k-1\}$, choose $\mathcal{R} = R \cup\{w_{i+1}\}$ and go to the step (ii).
If $\mathcal{R} = R \cup\{w_{k}\}$ then $\mathcal{C}$ is the collection of all crystallizations which yield  $(\langle S \mid R \rangle, R)$ and are minimal with respect to 
$(\langle S \mid R \rangle, R)$. If $\mathcal{C}$ is empty then such a
crystallization does not exist.
\end{enumerate}

\section{Applications of Algorithm 1}

From a given presentation 
$\langle S \mid R\rangle$ with two generators and two relations, using our algorithm, we construct all the possible 
$\lambda(\langle S \mid R\rangle, R)$-vertex crystallizations which yield $(\langle S \mid R\rangle, R)$. We also discuss the cases, where no such crystallization exists.

\subsection{Constructions of some crystallizations}
Here we consider the cases where Algorithm $1$ gives crystallizations for a pair $(\langle S \mid R\rangle, R)$.

\begin{eg}[\textbf{\boldmath{Crystallizations of $M\langle m,n,2 \rangle$ for $m,n\geq 2$}}]\label{eg:polyhedral}
{\rm Recall that binary polyhedral group $\langle m,n,k \rangle$ has a presentation $\langle x_1,x_2,x_3 \mid x_1^{m} = x_2^{n} = x_3^{k}=x_1x_2x_3 \rangle$ 
for some integer $m,n,k \geq 2$. 
If $k=2$ then $x_3=x_1x_2$ and hence $x_1^{m}=x_1x_2x_1x_2$ and $x_2^{n}=x_1x_2x_1x_2$.
Therefore,  $\langle m,n,2 \rangle$ has a presentation  $\langle S \mid R \rangle$, where $S=\{x_1,x_2\}$ and $R=\{x_1^{m-1}x_2^{-1}x_1^{-1}x_2^{-1},
x_2^{n-1}x_1^{-1}x_2^{-1}x_1^{-1}\}$.
It is not difficult to prove that, $x_1^{m}x_2^{-n}$ is the only independent element in $\overline{R}$ of minimum weight.

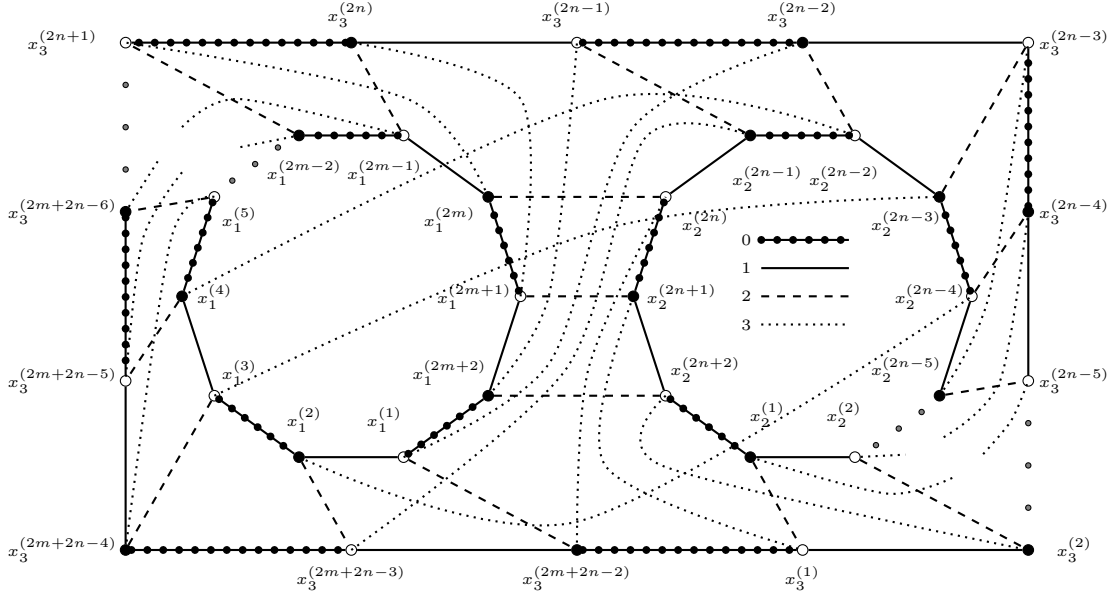
\begin{figure}[ht]
\tikzstyle{vert}=[circle, draw, fill=black!100, inner sep=0pt, minimum width=4pt] \tikzstyle{vertex}=[circle,
draw, fill=black!00, inner sep=0pt, minimum width=4pt] \tikzstyle{ver}=[] \tikzstyle{extra}=[circle, draw,
fill=black!50, inner sep=0pt, minimum width=2pt] \tikzstyle{edge} = [draw,thick,-] \centering
\begin{tikzpicture}[scale=0.75]

\begin{scope}[shift={(2,4)}]
\node[ver] (3) at (1,-4.5){\tiny{$3$}}; 
\node[ver] (2) at (1,-4){\tiny{$2$}};
\node[ver](1) at (1,-3.5){\tiny{$1$}}; 
\node[ver](0) at (1,-3){\tiny{$0$}}; 
\node[ver] (8) at (3,-4.5){}; 
\node[ver](7) at (3,-4){}; 
\node[ver](6) at (3,-3.5){}; 
\node[ver] (5) at (3,-3){};
\end{scope}

\begin{scope}[shift={(-4,0)}, rotate=36]
\foreach \x/\y in {72/x_1^{(2m-2)},144/x_1^{(4)},216/x_1^{(2)},288/x_1^{(2m+2)},0/x_1^{(2m)}}
{ \node[ver](\y) at (\x:2.3){\tiny{$\y~~$}};
    \node[vert] (\y) at (\x:3){};} 
    \foreach \x/\y in {108/x_1^{(5)},180/x_1^{(3)},252/x_1^{(1)},324/x_1^{(2m+1)},36/x_1^{(2m-1)}}{ 
    \node[ver] (\y) at (\x:2.3){\tiny{$\y~~$}};
    \node[vertex] (\y) at (\x:3){};}
    \foreach \x/\y in {x_1^{(1)}/x_1^{(2)},x_1^{(3)}/x_1^{(4)},x_1^{(2m-1)}/x_1^{(2m)},x_1^{(2m+1)}/x_1^{(2m+2)},0/5,1/6}{
\path[edge] (\x) -- (\y);} 

\foreach \x/\y in {x_1^{(1)}/x_1^{(2m+2)},x_1^{(3)}/x_1^{(2)},x_1^{(5)}/x_1^{(4)},x_1^{(2m+1)}/x_1^{(2m)},x_1^{(2m-1)}/x_1^{(2m-2)}}
{\path[edge] (\x) -- (\y);} 

\foreach \x/\y in {x_1^{(1)}/x_1^{(2m+2)},x_1^{(3)}/x_1^{(2)},x_1^{(5)}/x_1^{(4)},x_1^{(2m+1)}/x_1^{(2m)},x_1^{(2m-1)}/x_1^{(2m-2)},0/5}
{\draw [line width=3pt, line cap=round, dash pattern=on 0pt off 2\pgflinewidth]  (\x) -- (\y);} 
\foreach \x/\y in {-0.5/2.9,0/2.9,0.5/2.9}
{\node[extra] () at (\x,\y){};} 

\end{scope}

\begin{scope}[shift={(4,0)}, rotate=-36]
\foreach \x/\y in {72/x_2^{(2n-3)},144/x_2^{(2n-1)},216/x_2^{(2n+1)},288/x_2^{(1)},0/x_2^{(2n-5)}}
{ \node[ver](\y) at (\x:2.2){\tiny{$~\y$}};
    \node[vert] (\y) at (\x:3){};} 
    \foreach \x/\y in {108/x_2^{(2n-2)},180/x_2^{(2n)},252/x_2^{(2n+2)},324/x_2^{(2)},36/x_2^{(2n-4)}}{ 
    \node[ver] (\y) at (\x:2.2){\tiny{$~\y$}};
    \node[vertex] (\y) at (\x:3){};}
    \foreach \x/\y in {x_2^{(2n+2)}/x_2^{(2n+1)},x_2^{(2n)}/x_2^{(2n-1)},x_2^{(2n-4)}/x_2^{(2n-5)},x_2^{(2)}/x_2^{(1)}}{
\path[edge] (\x) -- (\y);} 

\foreach \x/\y in {x_2^{(2n+2)}/x_2^{(1)},x_2^{(2n)}/x_2^{(2n+1)},x_2^{(2n-2)}/x_2^{(2n-1)},x_2^{(2n-3)}/x_2^{(2n-2)},x_2^{(2n-4)}/x_2^{(2n-3)}}
{\path[edge] (\x) -- (\y);} 

\foreach \x/\y in {x_2^{(2n+2)}/x_2^{(1)},x_2^{(2n)}/x_2^{(2n+1)},x_2^{(2n-2)}/x_2^{(2n-1)},x_2^{(2n-4)}/x_2^{(2n-3)}}
{\draw [line width=3pt, line cap=round, dash pattern=on 0pt off 2\pgflinewidth]  (\x) -- (\y);} 
\end{scope}

\begin{scope}[shift={(4,0)}, rotate=-144]
\foreach \x/\y in {-0.5/2.9,0/2.9,0.5/2.9}
{\node[extra] () at (\x,\y){};} 
\end{scope}

\begin{scope}[]
\foreach \x/\y/\z/\w in {8/4.5/8.7/x_3^{(2n-3)},8/-1.5/8.7/x_3^{(2n-5)}}
{\node[ver](\w) at (\z,\y){\tiny{$~~\w$}};
\node[vertex] (\w) at (\x,\y){};} 
\foreach \x/\y/\z/\w in {8/1.5/8.7/x_3^{(2n-4)},8/-4.5/8.7/x_3^{(2)}}
{\node[ver](\w) at (\z,\y){\tiny{$~~\w$}};
\node[vert] (\w) at (\x,\y){};} 
\foreach \x/\y/\z/\w in {-8/4.5/-9/x_3^{(2n+1)},-8/-1.5/-9/x_3^{(2m+2n-5)}}
{\node[ver](\w) at (\z,\y){\tiny{$\w~~$}};
\node[vertex] (\w) at (\x,\y){};} 
\foreach \x/\y/\z/\w in {-8/1.5/-9/x_3^{(2m+2n-6)},-8/-4.5/-9/x_3^{(2m+2n-4)}}
{\node[ver](\w) at (\z,\y){\tiny{$\w~~$}};
\node[vert] (\w) at (\x,\y){};} 
\foreach \x/\y/\z/\w in {-4/-4.5/-5/x_3^{(2m+2n-3)},4/-4.5/-5/x_3^{(1)},0/4.5/5/x_3^{(2n-1)}}
{\node[ver](\w) at (\x,\z){\tiny{$\w$}};
\node[vertex] (\w) at (\x,\y){};} 
\foreach \x/\y/\z/\w in {-4/4.5/5/x_3^{(2n)},4/4.5/5/x_3^{(2n-2)},0/-4.5/-5/x_3^{(2m+2n-2)}}
{\node[ver](\w) at (\x,\z){\tiny{$\w$}};
\node[vert] (\w) at (\x,\y){};} 
\foreach \x/\y in {-8/2.25,-8/3,-8/3.75,8/-2.25,8/-3,8/-3.75}
{\node[extra] () at (\x,\y){};} 

\foreach \x/\y in {x_3^{(2n)}/x_3^{(2n+1)},x_3^{(2n-1)}/x_3^{(2n-2)},x_3^{(2m+2n-6)}/x_3^{(2m+2n-5)},x_3^{(2m+2n-4)}/x_3^{(2m+2n-3)},x_3^{(2m+2n-2)}/x_3^{(1)},x_3^{(2n-2)}/x_3^{(2n-3)},x_3^{(2n-4)}/x_3^{(2n-3)}}
{\path[edge] (\x) -- (\y);}

\foreach \x/\y in {x_3^{(2n)}/x_3^{(2n+1)},x_3^{(2n-1)}/x_3^{(2n-2)},x_3^{(2m+2n-6)}/x_3^{(2m+2n-5)},x_3^{(2m+2n-4)}/x_3^{(2m+2n-3)},x_3^{(2m+2n-2)}/x_3^{(1)},x_3^{(2n-4)}/x_3^{(2n-3)}}
{\draw [line width=3pt, line cap=round, dash pattern=on 0pt off 2\pgflinewidth]  (\x) -- (\y);} 

\foreach \x/\y in {x_3^{(2n)}/x_3^{(2n-1)},x_3^{(2m+2n-4)}/x_3^{(2m+2n-5)},x_3^{(2m+2n-2)}/x_3^{(2m+2n-3)},x_3^{(2)}/x_3^{(1)},x_3^{(2n-4)}/x_3^{(2n-5)}}
{\path[edge] (\x) -- (\y);}

\end{scope}

\foreach \x/\y in {x_1^{(2m)}/x_2^{(2n)},x_1^{(2m+1)}/x_2^{(2n+1)},x_1^{(2m+2)}/x_2^{(2n+2)},x_2^{(1)}/x_3^{(1)},x_2^{(2)}/x_3^{(2)},x_2^{(2n-5)}/x_3^{(2n-5)},x_2^{(2n-4)}/x_3^{(2n-4)},
x_2^{(2n-3)}/x_3^{(2n-3)},x_2^{(2n-2)}/x_3^{(2n-2)},x_2^{(2n-1)}/x_3^{(2n-1)},x_1^{(2m-1)}/x_3^{(2n)},x_1^{(2m-2)}/x_3^{(2n+1)},x_1^{(5)}/x_3^{(2m+2n-6)},x_1^{(4)}/x_3^{(2m+2n-5)},
x_1^{(3)}/x_3^{(2m+2n-4)},x_1^{(2)}/x_3^{(2m+2n-3)},x_1^{(1)}/x_3^{(2m+2n-2)},2/7}{
\path[edge, dashed] (\x) -- (\y);}

\node[ver](a1) at (-6.2,2.6){};
 \node[ver](a2) at (-7.3,2.6){};
  \node[ver](a3) at (6,-2.8){};
\foreach \x/\y in {3/8,x_1^{(2m-2)}/a1,x_3^{(2m+2n-6)}/a2,x_2^{(2)}/a3}{
\path[edge, dotted] (\x) -- (\y);} 

\draw[edge, dotted] plot [smooth,tension=0.5] coordinates{(x_1^{(5)}) (-7.2,1) (x_3^{(2m+2n-4)}) };
\draw[edge, dotted] plot [smooth,tension=0.5] coordinates{(x_1^{(2m)}) (-3,3.5) (x_3^{(2n+1)}) };
\draw[edge, dotted] plot [smooth,tension=0.5] coordinates{(x_1^{(2m+1)}) (-1,3.5) (x_3^{(2n)}) };
\draw[edge, dotted] plot [smooth,tension=0.5] coordinates{(x_1^{(2m+2)}) (-0.5,0) (x_3^{(2n-1)}) };
\draw[edge, dotted] plot [smooth,tension=0.5] coordinates{(x_1^{(1)}) (-0.5,-1.2)(1,3) (x_3^{(2n-2)}) };
\draw[edge, dotted] plot [smooth,tension=0.5] coordinates{(x_3^{(2m+2n-3)}) (-0.2,-1.3)(1.2,2.8) (x_2^{(2n-1)}) };
\draw[edge, dotted] plot [smooth,tension=0.5] coordinates{(x_3^{(2m+2n-2)}) (0.2,-1.3)(x_2^{(2n)}) };
\draw[edge, dotted] plot [smooth,tension=0.5] coordinates{(x_2^{(2n+1)}) (0.5,-3)(x_3^{(1)}) };
\draw[edge, dotted] plot [smooth,tension=0.5] coordinates{(x_2^{(2n+2)}) (1.5,-3)(x_3^{(2)}) };
\draw[edge, dotted] plot [smooth,tension=0.5] coordinates{(x_2^{(2n-5)}) (7.3,0)(x_3^{(2n-3)}) };
\draw[edge, dotted] plot [smooth,tension=0.5] coordinates{(x_1^{(2)}) (0,-4)(x_2^{(2n-4)}) };
\draw[edge, dotted] plot [smooth,tension=0.5] coordinates{(x_1^{(4)}) (0.5,3.5)(x_2^{(2n-2)}) };
\draw[edge, dotted] plot [smooth,tension=0.5] coordinates{(x_1^{(3)}) (0.5,1.2)(x_2^{(2n-3)}) };
\draw[edge, dotted] plot [smooth,tension=0.5] coordinates{(-7,2.2) (-7.7,1)(x_3^{(2m+2n-5)}) };
\draw[edge, dotted] plot [smooth,tension=0.5] coordinates{(-7,3) (-6,3.5)(x_1^{(2m-1)}) };
\draw[edge, dotted] plot [smooth,tension=0.5] coordinates{(x_2^{(1)}) (5.5,-3.5)(6.7,-3.2) };
\draw[edge, dotted] plot [smooth,tension=0.5] coordinates{ (7.2,-2.8)(7.7,-2.2)(x_3^{(2n-5)}) };
\draw[edge, dotted] plot [smooth,tension=0.5] coordinates{ (6.5,-2.5)(7.5,-1.5)(x_3^{(2n-4)}) };
\end{tikzpicture}
\caption{Crystallization of $M\langle m,n,2 \rangle$ for $m,n \geq 2$}\label{fig:polyhedral}
\end{figure}

Therefore, $\mathcal{R}=R\cup \{x_1^{m}x_2^{-n}\}$. Thus, $m^{(c)}_{12}=3$, $m^{(c)}_{13}=2m-1$, $m^{(c)}_{23}=2n-1$ for $2\leq c \leq 3$,
$(n_1,n_2,n_3)=(2m+2,2n+2,2m+2n-2)$ and  $G_i=C_{n_{i}}(x^{(1)}_i, \dots, x^{(n_i)}_i)$ for $1\leq i \leq 3$ as in Figure \ref{fig:polyhedral}. 
Choose $(n_i,n_j,n_l)=(n_3,n_1,n_2)$. Thus, $\{x^{(1)}_3x^{(1)}_2,\dots, x_3^{(2n-1)}x^{(2n-1)}_2\}$, 
$\{x^{(2m+2n-2)}_3x^{(1)}_1,\dots, x^{(2n)}_3x^{(2m-1)}_1\}$ 
and $\{x^{(2m+2)}_1x^{(2n+2)}_2, x^{(2m+1)}_1x^{(2n+1)}_2, x^{(2m)}_1x^{(2n)}_2\}$ are the sets of edges of color $2$. 
Here the only choice for the triplet $(q_1,q_2,q_3)$ is $(1,n-2,n-1)$ and for the path 
$P_5$ is $P_5(x_1^{(2)},  x_1^{(1)}, x_3^{(2n-2)}, x_3^{(2n-3)}, x_2^{(2n-5)}, x_2^{(2n-4)})$, they give a regular bipartite $4$-colored graph $(\Gamma, \gamma)$
which yields  $(\langle S \mid R\rangle, R)$  (cf. Lemma \ref{lemma:unique-polyhedra}). Therefore, by Theorem \ref{theorem:algorithm}, $(\Gamma, \gamma)$ (cf. Figure \ref{fig:polyhedral}) is a crystallization of the closed connected orientable $3$-manifold $M\langle m,n,2 \rangle$ (cf. Subsection \ref{subsec:quaternion}) and minimal with respect to $(\langle S \mid R\rangle, R)$.

Observe that (i) $\#V(\Gamma)= 4m+4n+2 =(2m+2)+(2n+2)+(2m+2n-2) = \lambda(\langle S \mid R\rangle, R)$, (ii) $\Gamma_{\{i,j\}}=(m+n-1)C_4 \sqcup C_6$ for  $0\leq i \leq 1$ and $2\leq j \leq3$
as $\#V(\Gamma)= 4(m+n-1)+2$, and (iii) the $m_{ij}^{(c)}$ edges of color $c$ between $G_i$ and $G_j$ yield $m_{ij}^{(c)}-1$  
bi-colored $4$-cycles  for $1 \leq i < j \leq 3$ and $2\leq c \leq3$.}
\end{eg}

\begin{lemma}\label{lemma:unique-polyhedra}
Let the presentation $(\langle S \mid R\rangle, R)$  and $q_1, q_2, q_3, P_5$ be as in Example $\ref{eg:polyhedral}$. 
Then, the choice of the triplet 
$(q_1,q_2,q_3)$ and the path $P_5$ for which the $4$-colored graph $(\Gamma,\gamma)$ yields $(\langle S \mid R\rangle, R)$, is unique.
\end{lemma}

\begin{proof}
Since $x_1^{m-1}x_2^{-1}x_1^{-1}x_2^{-1}$ (resp., $x_1^{m}x_2^{-n}$) is a relation which contains $x_1^{m-1}$ (resp., $x_1^{m}$), $m-2$ (resp., $m-1$) edges of color $3$ between
$G_1$ and $G_3$ are involved to yield  $x_1^{m-1}$ (resp., $x_1^{m}$). Since $\Gamma$ is bipartite, $m_{13}^{(3)}=2m-1$ and $G_1$ has $2m+2$ vertices, 
either white dots `$\circ$' vertices (resp., black dots  `$\bullet$' vertices) or  black dots `$\bullet$' vertices (resp., white dots `$\circ$' vertices) in $G_1$ are involved to yield $x_1^{m-1}$ (resp., $x_1^{m}$).
Again, the fact that $m^{(2)}_{12}=3$ and $\{x_2^{(2n+2)}, x_2^{(2n)}\}$ is a set of white dots `$\circ$' vertices implies that one of the vertices $x_2^{(2n+2)}, x_2^{(2n)}$ is the starting vertex for one of the two relations above. Thus, the black dot `$\bullet$' vertex $x_2^{(2n+1)}$ 
is the starting vertex for the other relation. Up to an automorphism, we can assume 
 that $x_2^{(2n)}$ and $x_2^{(2n+1)}$ are the starting vertices for the two relations above. Therefore, $x_1^{(2m)}$ and $x_1^{(2m+1)}$ are joined with vertices of $G_3$ by edges of color $3$.
Since the relation with starting vertex $x_2^{(2n+1)}$ contains  $x_1^{m-1}$ or $x_1^{m}$, either
$\{x_1^{(2m+1)}x_3^{(2n)},x_1^{(2m-1)}x_3^{(2n+2)}, \dots, x_1^{(7)}x_3^{(2m+2n-6)}\}$ or $\{x_1^{(2m+1)}x_3^{(2m+2n-2)}$, $x_1^{(1)}x_3^{(2m+2n-4)}$, 
$\dots, x_1^{(2m-7)}x_3^{(2n+4)}\}$
is the set of edges of color $3$. But, in the later case, $x_1^{(2m)}x_3^{(1)} \in \gamma^{-1} (3)$ as $x_1^{(2m)}$ and $x_1^{(2m+1)}$ are joined with vertices of $G_3$ by edges 
of color $3$. Then, the relation with starting vertex $x_2^{(2n)}$ contains $x_1x_2wx_2^{-1}$ for some $w \in F(S)$, which is a contradiction. Therefore, $x_1^{(2m+1)}x_3^{(2n)},x_1^{(2m-1)}x_3^{(2n+2)}, \dots, x_1^{(7)}x_3^{(2m+2n-6)} \in \gamma^{-1} (3)$ and hence
$x_1^{(2m)}x_3^{(2n+1)},x_1^{(2m-2)}x_3^{(2n+3)}, \dots, x_1^{(6)}x_3^{(2m+2n-5)} \in \gamma^{-1} (3)$. Since one relation contains $x_1^{m}$ and other contains $x_1^{m-1}$, 
$x_1^{(5)}x_3^{(2m+2n-4)} \in \gamma^{-1} (3)$ and hence $x_1^{(4)}, x_1^{(3)}, x_1^{(2)}$ are joined with $G_2$ with edges of color $3$. 
Thus, $x_1^{(2m+2)}x_3^{(2n-1)}, x_1^{(1)}x_3^{(2n-2)} \in \gamma^{-1} (3)$ as $m^{(3)}_{13}=2m-1$. Therefore, $x_2^{(2n+1)}$ is the starting vertex for the relation 
$x_1^{m}x_2^{-n}$.
In other words, we can say that $x_1^{(2m+1)}$ is the starting vertex for the relation $x_2^{n}x_1^{-m}$. Therefore, $x_1^{(2m+2)}$ is the starting vertex for the 
relation $x_2^{n-1}x_1^{-1}x_2^{-1}x_1^{-1}$. Thus, by  similar arguments as above, $x_2^{(2n-1)}x_3^{(2m+2n-3)},x_2^{(2n)}x_3^{(2m+2n-2)}, \dots, 
x_2^{(2n-5)}x_3^{(2n-3)}\in \gamma^{-1} (3)$.
Therefore, $x_1^{(2)}x_2^{(2n-4)},x_1^{(3)}x_2^{(2n-3)}, x_1^{(4)}x_2^{(2n-2)} \in \gamma^{-1} (3)$. 
Thus, the lemma follows.
\end{proof}

\begin{eg}[\textbf{\boldmath{Crystallization of $L(kq-1,q)$, for $k,q\geq 2$}}]\label{eg:lens1}
{\rm We know $\mathbb{Z}_{kq-1}$ has a presentation $\langle S \mid R \rangle$, where
$S=\{x_1,x_2\}$ and $R=\{x_1^{q}x_2^{-1},x_2^{k}x_1^{-1}\}$.
It is not difficult to prove that, $x_1^{q-1}x_2^{k-1}$ is the only independent
element in $\overline{R}$ of minimum weight.

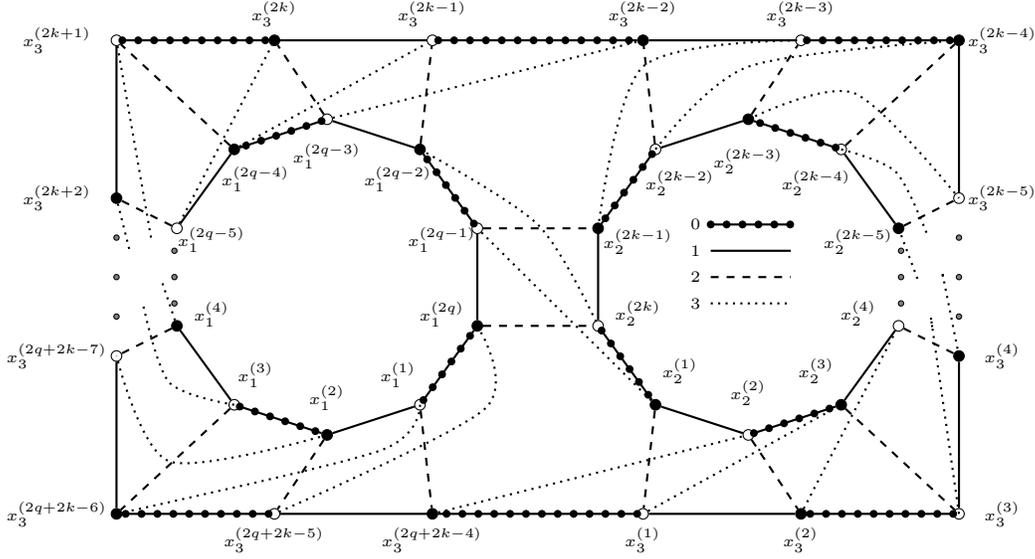
\begin{figure}[ht]
\tikzstyle{vert}=[circle, draw, fill=black!100, inner sep=0pt, minimum width=4pt] \tikzstyle{vertex}=[circle,
draw, fill=black!00, inner sep=0pt, minimum width=4pt] \tikzstyle{ver}=[] \tikzstyle{extra}=[circle, draw,
fill=black!50, inner sep=0pt, minimum width=2pt] \tikzstyle{edge} = [draw,thick,-] \centering
\begin{tikzpicture}[scale=0.7]

\begin{scope}[shift={(2,4)}]
\node[ver] (3) at (1,-4.5){\tiny{$3$}}; 
\node[ver] (2) at (1,-4){\tiny{$2$}};
\node[ver](1) at (1,-3.5){\tiny{$1$}}; 
\node[ver](0) at (1,-3){\tiny{$0$}}; 
\node[ver] (8) at (3,-4.5){}; 
\node[ver](7) at (3,-4){}; 
\node[ver](6) at (3,-3.5){}; 
\node[ver] (5) at (3,-3){};
\end{scope}

\begin{scope}[shift={(-4,0)}, rotate=54]
\foreach \x/\y in {72/x_1^{(2q-4)},144/x_1^{(4)},216/x_1^{(2)},288/x_1^{(2q)},0/x_1^{(2q-2)}}
{ \node[ver](\y) at (\x:2.3){\tiny{$\y$}};
    \node[vert] (\y) at (\x:3){};} 
    \foreach \x/\y in {108/x_1^{(2q-5)},180/x_1^{(3)},252/x_1^{(1)},324/x_1^{(2q-1)},36/x_1^{(2q-3)}}{ 
    \node[ver] (\y) at (\x:2.3){\tiny{$\y$}};
    \node[vertex] (\y) at (\x:3){};}
    \foreach \x/\y in {x_1^{(1)}/x_1^{(2)},x_1^{(3)}/x_1^{(4)},x_1^{(2q-3)}/x_1^{(2q-2)},x_1^{(2q-1)}/x_1^{(2q)},0/5,1/6}{
\path[edge] (\x) -- (\y);} 

\foreach \x/\y in {x_1^{(1)}/x_1^{(2q)},x_1^{(3)}/x_1^{(2)},x_1^{(2q-1)}/x_1^{(2q-2)},x_1^{(2q-3)}/x_1^{(2q-4)},x_1^{(2q-5)}/x_1^{(2q-4)}}
{\path[edge] (\x) -- (\y);} 

\foreach \x/\y in {x_1^{(1)}/x_1^{(2q)},x_1^{(3)}/x_1^{(2)},x_1^{(2q-1)}/x_1^{(2q-2)},x_1^{(2q-3)}/x_1^{(2q-4)},0/5}
{\draw [line width=3pt, line cap=round, dash pattern=on 0pt off 2\pgflinewidth]  (\x) -- (\y);} 
\end{scope}

\begin{scope}[shift={(4,0)}, rotate=-54]
\foreach \x/\y in {72/x_2^{(2k-5)},144/x_2^{(2k-3)},216/x_2^{(2k-1)},288/x_2^{(1)},0/x_2^{(3)}}
{ \node[ver](\y) at (\x:2.2){\tiny{$\y$}};
    \node[vert] (\y) at (\x:3){};} 
    \foreach \x/\y in {108/x_2^{(2k-4)},180/x_2^{(2k-2)},252/x_2^{(2k)},324/x_2^{(2)},36/x_2^{(4)}}{ 
    \node[ver] (\y) at (\x:2.2){\tiny{$\y$}};
    \node[vertex] (\y) at (\x:3){};}
    \foreach \x/\y in {x_2^{(2k)}/x_2^{(2k-1)},x_2^{(2k-2)}/x_2^{(2k-3)},x_2^{(4)}/x_2^{(3)},x_2^{(2)}/x_2^{(1)}}{
\path[edge] (\x) -- (\y);} 

\foreach \x/\y in {x_2^{(2k)}/x_2^{(1)},x_2^{(2k-2)}/x_2^{(2k-1)},x_2^{(2k-4)}/x_2^{(2k-3)},x_2^{(2k-5)}/x_2^{(2k-4)},x_2^{(2)}/x_2^{(3)}}
{\path[edge] (\x) -- (\y);} 

\foreach \x/\y in {x_2^{(2k)}/x_2^{(1)},x_2^{(2k-2)}/x_2^{(2k-1)},x_2^{(2k-4)}/x_2^{(2k-3)},x_2^{(2)}/x_2^{(3)}}
{\draw [line width=3pt, line cap=round, dash pattern=on 0pt off 2\pgflinewidth]  (\x) -- (\y);} 
\end{scope}

\begin{scope}[shift={(4,0)}, rotate=-90]
\foreach \x/\y in {-0.5/2.9,0/2.9,0.5/2.9}
{\node[extra] () at (\x,\y){};} 
\end{scope}
\begin{scope}[shift={(-4,0)}, rotate=90]
\foreach \x/\y in {-0.5/2.9,0/2.9,0.5/2.9}
{\node[extra] () at (\x,\y){};} 
\end{scope}

\begin{scope}[]
\foreach \x/\y/\z/\w in {8/4.5/8.7/x_3^{(2k-4)},8/-1.5/8.7/x_3^{(4)}}
{\node[ver](\w) at (\z,\y){\tiny{$~~\w$}};
\node[vert] (\w) at (\x,\y){};} 
\foreach \x/\y/\z/\w in {8/1.5/8.7/x_3^{(2k-5)},8/-4.5/8.7/x_3^{(3)}}
{\node[ver](\w) at (\z,\y){\tiny{$~~\w$}};
\node[vertex] (\w) at (\x,\y){};} 
\foreach \x/\y/\z/\w in {-8/4.5/-9/x_3^{(2k+1)},-8/-1.5/-9/x_3^{(2q+2k-7)}}
{\node[ver](\w) at (\z,\y){\tiny{$\w~~$}};
\node[vertex] (\w) at (\x,\y){};} 
\foreach \x/\y/\z/\w in {-8/1.5/-9/x_3^{(2k+2)},-8/-4.5/-9/x_3^{(2q+2k-6)}}
{\node[ver](\w) at (\z,\y){\tiny{$\w~~$}};
\node[vert] (\w) at (\x,\y){};} 

\foreach \x/\y/\z/\w in {-5/-4.5/-5/x_3^{(2q+2k-5)},5/4.5/5/x_3^{(2k-3)},-2/4.5/5/x_3^{(2k-1)},2/-4.5/-5/x_3^{(1)}}
{\node[ver](\w) at (\x,\z){\tiny{$\w$}};
\node[vertex] (\w) at (\x,\y){};} 
\foreach \x/\y/\z/\w in {-5/4.5/5/x_3^{(2k)},5/-4.5/-5/x_3^{(2)},-2/-4.5/-5/x_3^{(2q+2k-4)},2/4.5/5/x_3^{(2k-2)}}
{\node[ver](\w) at (\x,\z){\tiny{$\w$}};
\node[vert] (\w) at (\x,\y){};} 

\foreach \x/\y in {-8/-0.75,-8/0,-8/0.75,8/-0.75,8/0,8/0.75}
{\node[extra] () at (\x,\y){};} 

\foreach \x/\y in {x_3^{(1)}/x_3^{(2)},x_3^{(3)}/x_3^{(4)},x_3^{(2k-5)}/x_3^{(2k-4)},x_3^{(2k-3)}/x_3^{(2k-2)},x_3^{(2k-1)}/x_3^{(2k)},x_3^{(2k+1)}/x_3^{(2k+2)},x_3^{(2q+2k-5)}/x_3^{(2q+2k-4)},x_3^{(2q+2k-6)}/x_3^{(2q+2k-7)}}
{\path[edge] (\x) -- (\y);}

\foreach \x/\y in {x_3^{(1)}/x_3^{(2q+2k-4)},x_3^{(3)}/x_3^{(2)},x_3^{(2k-3)}/x_3^{(2k-4)},x_3^{(2k-1)}/x_3^{(2k-2)},x_3^{(2k+1)}/x_3^{(2k)},x_3^{(2q+2k-5)}/x_3^{(2q+2k-6)}}
{\path[edge] (\x) -- (\y);}

\foreach \x/\y in {x_3^{(1)}/x_3^{(2q+2k-4)},x_3^{(3)}/x_3^{(2)},x_3^{(2k-3)}/x_3^{(2k-4)},x_3^{(2k-1)}/x_3^{(2k-2)},x_3^{(2k+1)}/x_3^{(2k)},x_3^{(2q+2k-5)}/x_3^{(2q+2k-6)}}
{\draw [line width=3pt, line cap=round, dash pattern=on 0pt off 2\pgflinewidth]  (\x) -- (\y);}

\end{scope}

\foreach \x/\y in {x_1^{(2q-1)}/x_2^{(2k-1)},x_1^{(2q)}/x_2^{(2k)},x_2^{(1)}/x_3^{(1)},x_2^{(2)}/x_3^{(2)},x_2^{(3)}/x_3^{(3)},x_2^{(4)}/x_3^{(4)},
x_2^{(2k-5)}/x_3^{(2k-5)},x_2^{(2k-4)}/x_3^{(2k-4)},x_2^{(2k-3)}/x_3^{(2k-3)},x_2^{(2k-2)}/x_3^{(2k-2)},x_1^{(2q-2)}/x_3^{(2k-1)},x_1^{(2q-3)}/x_3^{(2k)},x_1^{(2q-4)}/x_3^{(2k+1)},
x_1^{(2q-5)}/x_3^{(2k+2)},x_1^{(4)}/x_3^{(2q+2k-7)},x_1^{(3)}/x_3^{(2q+2k-6)},x_1^{(2)}/x_3^{(2q+2k-5)},x_1^{(1)}/x_3^{(2q+2k-4)},2/7}{
\path[edge, dashed] (\x) -- (\y);} 

\draw[edge, dotted] plot [smooth,tension=0.5] coordinates{(x_1^{(2q)}) (-1,-2.5) (x_3^{(2q+2k-5)}) };
\draw[edge, dotted] plot [smooth,tension=0.5] coordinates{(x_1^{(1)}) (-3,-3.3) (x_3^{(2q+2k-6)}) };
\draw[edge, dotted] plot [smooth,tension=0.5] coordinates{(x_1^{(2)}) (-7,-3.5) (x_3^{(2q+2k-7)}) };
\draw[edge, dotted] plot [smooth,tension=0.5] coordinates{(x_1^{(2q-2)}) (0,1) (x_2^{(2k)}) };
\draw[edge, dotted] plot [smooth,tension=0.5] coordinates{(x_2^{(2k-1)}) (2,4) (x_3^{(2k-3)}) };
\draw[edge, dotted] plot [smooth,tension=0.5] coordinates{(x_2^{(2k-2)}) (4,4) (x_3^{(2k-4)}) };
\draw[edge, dotted] plot [smooth,tension=0.5] coordinates{(x_2^{(2k-3)}) (6,3.5) (x_3^{(2k-5)}) };
\draw[edge, dotted] plot [smooth,tension=0.5] coordinates{(x_1^{(3)}) (-7,-2) (-7.5,-0.5) };
\draw[edge, dotted] plot [smooth,tension=0.5] coordinates{(x_2^{(2k-4)}) (6.8,2) (7.4,0.8) };
\draw[edge, dotted] plot [smooth,tension=0.5] coordinates{(x_3^{(4)}) (7.7,0) (7.7,0) };
\draw[edge, dotted] plot [smooth,tension=0.5] coordinates{(x_2^{(2k-5)}) (7.2,0) (7.2,0) };
\draw[edge, dotted] plot [smooth,tension=0.5] coordinates{(x_3^{(3)}) (7.5,-1) (7.5,-1) };
\node[ver](a2) at (-7.2,0.1){};
\node[ver](a1) at (-7.3,0.5){};
 \node[ver](a3) at (-7.7,0.3){};
\foreach \x/\y in {3/8,x_1^{(2q-5)}/x_3^{(2k)},x_1^{(2q-4)}/x_3^{(2k-1)},x_1^{(2q-3)}/x_3^{(2k-2)},x_1^{(2q-1)}/x_2^{(1)},x_3^{(2q+2k-4)}/x_2^{(2)},x_3^{(1)}/x_2^{(3)},x_3^{(2)}/x_2^{(4)},
x_3^{(2k+1)}/a1,x_1^{(4)}/a2,x_3^{(2k+2)}/a3}{
\path[edge, dotted] (\x) -- (\y);} 

\end{tikzpicture}
\caption{Crystallization of $L(kq-1,q)$ for $k,q \geq 2$}\label{fig:lens1}
\end{figure}

Therefore, $\mathcal{R}=R\cup \{x_1^{q-1}x_2^{k-1}\}$. Thus,  $m^{(c)}_{12}=2$, $m^{(c)}_{13}=2q-2$, $m^{(c)}_{23}=2k-2$ for $2\leq c \leq 3$, 
$(n_1,n_2,n_3)=(2q,2k,2q+2k-4)$ and $G_i=C_{n_{i}}(x^{(1)}_i, \dots, x^{(n_i)}_i)$ for $1\leq i \leq 3$ as in Figure \ref{fig:lens1}. 
Choose $(n_i,n_j,n_l)=(n_3,n_1,n_2)$. Thus, $\{x^{(1)}_3x^{(1)}_2,\dots, x_3^{(2k-2)}x^{(2k-2)}_2\}$, 
$\{x^{(2q+2k-4)}_3x^{(1)}_1,\dots, x^{(2k-1)}_3x^{(2q-2)}_1\}$ 
and $\{x^{(2q-1)}_1x^{(2k-1)}_2, x^{(2q)}_1x^{(2k)}_2\}$ are the sets of edges of color $2$. 
Here the only choice for the triplet $(q_1,q_2,q_3)$ is $(q-1,k,k-1)$ and for the path 
$P_5$ is $P_5(x_1^{(2q-2)}, x_1^{(2q-3)}$, $x_3^{(2k-2)}, x_3^{(2k-3)}, x_2^{(2k-1)}, x_2^{(2k)})$, they give a regular bipartite $4$-colored graph $(\Gamma, \gamma)$
which yields  $(\langle S \mid R\rangle, R)$  (cf. Lemma \ref{lemma:unique-lens1}). 
Therefore, by Theorem \ref{theorem:algorithm}, $(\Gamma, \gamma)$ (cf. Figure \ref{fig:lens1}) is a crystallization of a
closed connected orientable $3$-manifold and minimal with respect to $(\langle S \mid R\rangle, R)$.
Here $(\Gamma, \gamma)$ is a crystallization of $L(kq-1,q)$ as $(\Gamma, \gamma)$ is isomorphic to the graph $\mathcal{M}_{k,q}$ (cf. \cite[Subsection 5.1]{bd14}) and $\mathcal{M}_{k,q}$ is a crystallization of $L(kq-1,q)$.

Observe that (i) $\#V(\Gamma)= 4(q+k-1) =2q+2k+2(q+k-2) = \lambda(\langle S \mid R\rangle, R)$, and
(ii) $\Gamma_{\{0,3\}}$ (resp., $\Gamma_{\{1,2\}}$) is a union of $(q+k-1)$ $4$-cycles and $\Gamma_{\{0,2\}}$ (resp., $\Gamma_{\{1,3\}}$) is of 
type $(q+k-4)C_4 \sqcup 2C_6$ as $\#V(\Gamma)= 4(q+k-1)$ and $m^{(c)}_{ij} \geq 1$ for $2\leq c \leq 3$ and $1 \leq i < j \leq 3$.}
\end{eg}

\begin{lemma}\label{lemma:unique-lens1}
Let the presentation $(\langle S \mid R\rangle, R)$  and $q_1, q_2, q_3, P_5$ be as in Example $\ref{eg:lens1}$. Then, the choice of the triplet 
$(q_1,q_2,q_3)$ and the path $P_5$ for which the $4$-colored graph $(\Gamma,\gamma)$ yields $(\langle S \mid R\rangle, R)$, is unique.
\end{lemma}

\begin{proof}

Clearly, either $x_2^{(2k)}$ or $x_2^{(2k-1)}$ is the starting vertex for  the relation $x_1^{q}x_2^{-1}$.  Up to an automorphism, we can assume
$x_2^{(2k)}$ is the starting vertex for the relation $x_1^{q}x_2^{-1}$. Since $\Gamma$ is bipartite,
all the black dots `$\bullet$' vertices in $G_1$ are involved to 
yield the relation $x_1^{q}x_2^{-1}$. Thus, either $\{x_1^{(2q)}x_3^{(2k-1)},x_1^{(2q-2)}x_3^{(2k+1)}, \dots$, 
$x_1^{(4)}x_3^{(2q+2k-5)}\}$ or $\{x_1^{(2q)}x_3^{(2q+2k-5)}, x_1^{(2)}x_3^{(2q+2k-7)}$, 
$\dots, x_1^{(2q-4)}x_3^{(2k-1)}\}$ is the set of edges of color $3$. Since the $2q-2$ edges of color $3$ between $G_1$ and 
$G_3$ form $2q-3$ bi-colored $4$-cycles  in $\Gamma_{\{0,1,3\}}$, in the first case, $x_1^{(2q-1)}x_3^{(2k)},x_1^{(2q-3)}x_3^{(2k+2)}$, 
$\dots, x_1^{(5)}x_3^{(2q+2k-6)} \in \gamma^{-1}(3)$. This gives 
a relation $x_1^{q-2}wx_2^{-1}$ other than $x_1^{q}x_2^{-1}$ for some $w\in F(S)$. But, there does not exist such a relation. 
So, $x_1^{(2q)}x_3^{(2q+2k-5)}$, $x_1^{(2)}x_3^{(2q+2k-7)}$, 
$\dots, x_1^{(2q-4)}x_3^{(2k-1)}$ $ \in \gamma^{-1}(3)$ and hence $x_1^{(1)}x_3^{(2q+2k-6)}, x_1^{(3)}x_3^{(2q+2k-8)},\dots$,
$x_1^{(2q-5)}$ $x_3^{(2k)} \in \gamma^{-1}(3)$. Since $x_2^{(2k)}$ is the starting vertex for the relation $x_1^{q}x_2^{-1}$, we have
$x_1^{(2q-2)}x_2^{(2k)} \in \gamma^{-1}(3)$ and 
hence $x_1^{(2q-1)}x_2^{(1)}, x_1^{(2q-3)}x_3^{(2k-2)} \in \gamma^{-1}(3)$ as 
all components of $\Gamma_{\{0,3\}}$ are $4$-cycle. 
By similar arguments as above, $x_2^{(2k-1)}x_3^{(2k-3)}$, $x_2^{(2k-2)}x_3^{(2k-4)},\dots$, $x_2^{(2)}x_3^{(2q+2k-4)} \in \gamma^{-1}(3)$ 
as $x_1^{(2q-1)}$ is the starting vertex for the relation $x_2^kx_1^{-1}$.
Thus, the lemma follows.
\end{proof}

\begin{eg}[\textbf{\boldmath{Crystallization of $L((k-1)q+1,q)$ for $k,q\geq 2$}}]\label{eg:lens2}
{\rm We know $\mathbb{Z}_{(k-1)q+1}$  has a presentation $\langle S \mid R \rangle$, where
$S=\{x_1,x_2\}$ and $R=\{x_1^{q}x_2^{-1},x_1^{q-1}x_2^{-k}\}$.

\begin{figure}[ht]
\tikzstyle{vert}=[circle, draw, fill=black!100, inner sep=0pt, minimum width=4pt] \tikzstyle{vertex}=[circle,
draw, fill=black!00, inner sep=0pt, minimum width=4pt] \tikzstyle{ver}=[] \tikzstyle{extra}=[circle, draw,
fill=black!50, inner sep=0pt, minimum width=2pt] \tikzstyle{edge} = [draw,thick,-] \centering
\begin{tikzpicture}[scale=0.7]

\begin{scope}[shift={(2,4)}]
\node[ver] (3) at (1,-4.5){\tiny{$3$}}; 
\node[ver] (2) at (1,-4){\tiny{$2$}};
\node[ver](1) at (1,-3.5){\tiny{$1$}}; 
\node[ver](0) at (1,-3){\tiny{$0$}}; 
\node[ver] (8) at (3,-4.5){}; 
\node[ver](7) at (3,-4){}; 
\node[ver](6) at (3,-3.5){}; 
\node[ver] (5) at (3,-3){};
\end{scope}

\begin{scope}[shift={(-4,0)}, rotate=54]
\foreach \x/\y in {72/x_1^{(2q-4)},144/x_1^{(4)},216/x_1^{(2)},288/x_1^{(2q)},0/x_1^{(2q-2)}}
{ \node[ver](\y) at (\x:2.3){\tiny{$\y$}};
    \node[vert] (\y) at (\x:3){};} 
    \foreach \x/\y in {108/x_1^{(2q-5)},180/x_1^{(3)},252/x_1^{(1)},324/x_1^{(2q-1)},36/x_1^{(2q-3)}}{ 
    \node[ver] (\y) at (\x:2.3){\tiny{$\y$}};
    \node[vertex] (\y) at (\x:3){};}
    \foreach \x/\y in {x_1^{(1)}/x_1^{(2)},x_1^{(3)}/x_1^{(4)},x_1^{(2q-3)}/x_1^{(2q-2)},x_1^{(2q-1)}/x_1^{(2q)},0/5,1/6}{
\path[edge] (\x) -- (\y);} 

\foreach \x/\y in {x_1^{(1)}/x_1^{(2q)},x_1^{(3)}/x_1^{(2)},x_1^{(2q-1)}/x_1^{(2q-2)},x_1^{(2q-3)}/x_1^{(2q-4)},x_1^{(2q-5)}/x_1^{(2q-4)}}
{\path[edge] (\x) -- (\y);} 

\foreach \x/\y in {x_1^{(1)}/x_1^{(2q)},x_1^{(3)}/x_1^{(2)},x_1^{(2q-1)}/x_1^{(2q-2)},x_1^{(2q-3)}/x_1^{(2q-4)},0/5}
{\draw [line width=3pt, line cap=round, dash pattern=on 0pt off 2\pgflinewidth]  (\x) -- (\y);} 
\end{scope}

\begin{scope}[shift={(4,0)}, rotate=-54]
\foreach \x/\y in {72/x_2^{(2k-5)},144/x_2^{(2k-3)},216/x_2^{(2k-1)},288/x_2^{(1)},0/x_2^{(3)}}
{ \node[ver](\y) at (\x:2.2){\tiny{$\y$}};
    \node[vert] (\y) at (\x:3){};} 
    \foreach \x/\y in {108/x_2^{(2k-4)},180/x_2^{(2k-2)},252/x_2^{(2k)},324/x_2^{(2)},36/x_2^{(4)}}{ 
    \node[ver] (\y) at (\x:2.2){\tiny{$\y$}};
    \node[vertex] (\y) at (\x:3){};}
    \foreach \x/\y in {x_2^{(2k)}/x_2^{(2k-1)},x_2^{(2k-2)}/x_2^{(2k-3)},x_2^{(4)}/x_2^{(3)},x_2^{(2)}/x_2^{(1)}}{
\path[edge] (\x) -- (\y);} 

\foreach \x/\y in {x_2^{(2k)}/x_2^{(1)},x_2^{(2k-2)}/x_2^{(2k-1)},x_2^{(2k-4)}/x_2^{(2k-3)},x_2^{(2k-5)}/x_2^{(2k-4)},x_2^{(2)}/x_2^{(3)}}
{\path[edge] (\x) -- (\y);} 

\foreach \x/\y in {x_2^{(2k)}/x_2^{(1)},x_2^{(2k-2)}/x_2^{(2k-1)},x_2^{(2k-4)}/x_2^{(2k-3)},x_2^{(2)}/x_2^{(3)}}
{\draw [line width=3pt, line cap=round, dash pattern=on 0pt off 2\pgflinewidth]  (\x) -- (\y);} 
\end{scope}

\begin{scope}[shift={(4,0)}, rotate=-90]
\foreach \x/\y in {-0.5/2.9,0/2.9,0.5/2.9}
{\node[extra] () at (\x,\y){};} 
\end{scope}
\begin{scope}[shift={(-4,0)}, rotate=90]
\foreach \x/\y in {-0.5/2.9,0/2.9,0.5/2.9}
{\node[extra] () at (\x,\y){};} 
\end{scope}

\begin{scope}[]
\foreach \x/\y/\z/\w in {8/4.5/8.7/x_3^{(2k-4)},8/-1.5/8.7/x_3^{(4)}}
{\node[ver](\w) at (\z,\y){\tiny{$~~\w$}};
\node[vert] (\w) at (\x,\y){};} 
\foreach \x/\y/\z/\w in {8/1.5/8.7/x_3^{(2k-5)},8/-4.5/8.7/x_3^{(3)}}
{\node[ver](\w) at (\z,\y){\tiny{$~~\w$}};
\node[vertex] (\w) at (\x,\y){};} 
\foreach \x/\y/\z/\w in {-8/4.5/-9/x_3^{(2k+1)},-8/-1.5/-9/x_3^{(2q+2k-7)}}
{\node[ver](\w) at (\z,\y){\tiny{$\w~~$}};
\node[vertex] (\w) at (\x,\y){};} 
\foreach \x/\y/\z/\w in {-8/1.5/-9/x_3^{(2k+2)},-8/-4.5/-9/x_3^{(2q+2k-6)}}
{\node[ver](\w) at (\z,\y){\tiny{$\w~~$}};
\node[vert] (\w) at (\x,\y){};} 

\foreach \x/\y/\z/\w in {-5/-4.5/-5/x_3^{(2q+2k-5)},5/4.5/5/x_3^{(2k-3)},-2/4.5/5/x_3^{(2k-1)},2/-4.5/-5/x_3^{(1)}}
{\node[ver](\w) at (\x,\z){\tiny{$\w$}};
\node[vertex] (\w) at (\x,\y){};} 
\foreach \x/\y/\z/\w in {-5/4.5/5/x_3^{(2k)},5/-4.5/-5/x_3^{(2)},-2/-4.5/-5/x_3^{(2q+2k-4)},2/4.5/5/x_3^{(2k-2)}}
{\node[ver](\w) at (\x,\z){\tiny{$\w$}};
\node[vert] (\w) at (\x,\y){};} 

\foreach \x/\y in {-8/-0.75,-8/0,-8/0.75,8/-0.75,8/0,8/0.75}
{\node[extra] () at (\x,\y){};} 

\foreach \x/\y in {x_3^{(1)}/x_3^{(2)},x_3^{(3)}/x_3^{(4)},x_3^{(2k-5)}/x_3^{(2k-4)},x_3^{(2k-3)}/x_3^{(2k-2)},x_3^{(2k-1)}/x_3^{(2k)},x_3^{(2k+1)}/x_3^{(2k+2)},x_3^{(2q+2k-5)}/x_3^{(2q+2k-4)},x_3^{(2q+2k-6)}/x_3^{(2q+2k-7)}}
{\path[edge] (\x) -- (\y);}

\foreach \x/\y in {x_3^{(1)}/x_3^{(2q+2k-4)},x_3^{(3)}/x_3^{(2)},x_3^{(2k-3)}/x_3^{(2k-4)},x_3^{(2k-1)}/x_3^{(2k-2)},x_3^{(2k+1)}/x_3^{(2k)},x_3^{(2q+2k-5)}/x_3^{(2q+2k-6)}}
{\path[edge] (\x) -- (\y);}

\foreach \x/\y in {x_3^{(1)}/x_3^{(2q+2k-4)},x_3^{(3)}/x_3^{(2)},x_3^{(2k-3)}/x_3^{(2k-4)},x_3^{(2k-1)}/x_3^{(2k-2)},x_3^{(2k+1)}/x_3^{(2k)},x_3^{(2q+2k-5)}/x_3^{(2q+2k-6)}}
{\draw [line width=3pt, line cap=round, dash pattern=on 0pt off 2\pgflinewidth]  (\x) -- (\y);}

\end{scope}

\foreach \x/\y in {x_1^{(2q-1)}/x_2^{(2k-1)},x_1^{(2q)}/x_2^{(2k)},x_2^{(1)}/x_3^{(1)},x_2^{(2)}/x_3^{(2)},x_2^{(3)}/x_3^{(3)},x_2^{(4)}/x_3^{(4)},
x_2^{(2k-5)}/x_3^{(2k-5)},x_2^{(2k-4)}/x_3^{(2k-4)},x_2^{(2k-3)}/x_3^{(2k-3)},x_2^{(2k-2)}/x_3^{(2k-2)},x_1^{(2q-2)}/x_3^{(2k-1)},x_1^{(2q-3)}/x_3^{(2k)},x_1^{(2q-4)}/x_3^{(2k+1)},
x_1^{(2q-5)}/x_3^{(2k+2)},x_1^{(4)}/x_3^{(2q+2k-7)},x_1^{(3)}/x_3^{(2q+2k-6)},x_1^{(2)}/x_3^{(2q+2k-5)},x_1^{(1)}/x_3^{(2q+2k-4)},2/7}{
\path[edge, dashed] (\x) -- (\y);} 

\draw[edge, dotted] plot [smooth,tension=0.5] coordinates{(x_1^{(2q)}) (-1,-2.5) (x_3^{(2q+2k-5)}) };
\draw[edge, dotted] plot [smooth,tension=0.5] coordinates{(x_1^{(1)}) (-3,-3.3) (x_3^{(2q+2k-6)}) };
\draw[edge, dotted] plot [smooth,tension=0.5] coordinates{(x_1^{(2)}) (-7,-3.5) (x_3^{(2q+2k-7)}) };
\draw[edge, dotted] plot [smooth,tension=0.5] coordinates{(x_1^{(2q-3)}) (0,2.5) (x_2^{(2k-1)}) };
\draw[edge, dotted] plot [smooth,tension=0.5] coordinates{(x_2^{(2k-2)}) (-3,4) (x_1^{(2q-4)}) };
\draw[edge, dotted] plot [smooth,tension=0.5] coordinates{(x_1^{(2q-1)}) (0,-2) (x_3^{(2q+2k-4)}) };
\draw[edge, dotted] plot [smooth,tension=0.5] coordinates{(x_1^{(3)}) (-7,-2) (-7.5,-0.5) };
\draw[edge, dotted] plot [smooth,tension=0.5] coordinates{(x_1^{(2q-2)}) (0,0) (x_3^{(1)}) };
\draw[edge, dotted] plot [smooth,tension=0.5] coordinates{(x_2^{(2k)}) (2,-3) (x_3^{(2)}) };
\draw[edge, dotted] plot [smooth,tension=0.5] coordinates{(x_2^{(1)}) (3,-3.2) (x_3^{(3)}) };
\draw[edge, dotted] plot [smooth,tension=0.5] coordinates{(x_2^{(2)}) (6,-3.5) (x_3^{(4)}) };
\draw[edge, dotted] plot [smooth,tension=0.5] coordinates{(x_3^{(2k-5)}) (7.8,0.5) (7.8,0.5) };
\draw[edge, dotted] plot [smooth,tension=0.5] coordinates{(x_2^{(3)}) (7,-2) (7.5,-0.8) };
\draw[edge, dotted] plot [smooth,tension=0.5] coordinates{(x_2^{(4)}) (7.2,0) (7.2,0) };
\draw[edge, dotted] plot [smooth,tension=0.5] coordinates{(x_3^{(2k-4)}) (7.5,1) (7.5,1) };
\node[ver](a2) at (-7.2,0.1){};
\node[ver](a1) at (-7.3,0.5){};
 \node[ver](a3) at (-7.7,0.3){};
\foreach \x/\y in {3/8,x_1^{(2q-5)}/x_3^{(2k)},x_3^{(2k-1)}/x_2^{(2k-3)},x_3^{(2k-2)}/x_2^{(2k-4)},x_3^{(2k-3)}/x_2^{(2k-5)},
x_3^{(2k+1)}/a1,x_1^{(4)}/a2,x_3^{(2k+2)}/a3}{
\path[edge, dotted] (\x) -- (\y);} 

\end{tikzpicture}
\caption{Crystallization of $L((k-1)q+1,q)$ for $k,q \geq 2$}\label{fig:lens2}
\end{figure}
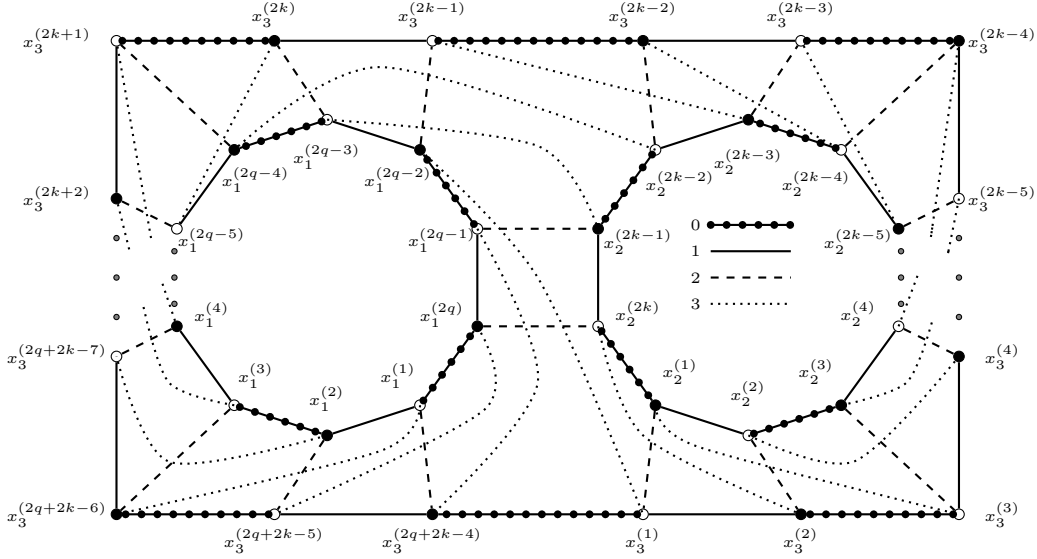

It is not difficult to prove that, $x_1x_2^{k}$ is the only independent element in $\overline{R}$ of minimum weight.
Therefore, $\mathcal{R}=R\cup \{x_1x_2^{k}\}$. Thus, $m^{(c)}_{12}=2$, $m^{(c)}_{13}=2q-2$, $m^{(c)}_{23}=2k-2$ for $2\leq c \leq 3$, $(n_1,n_2,n_3)=(2q,2k,2q+2k-4)$.
Choose $(n_i,n_j,n_k)=(n_3,n_1,n_2)$. Therefore, the $3$-colored graph with colors $0,1,2$ as in previous example. Here, the only choice for the triplet
$(q_1,q_2,q_3)$ is $(q-1,k,1)$ and for the path 
$P_5$ is $P_5(x_1^{(2q-3)},  x_1^{(2q-2)}, x_3^{(1)}, x_3^{(2)}, x_2^{(2k)}, x_2^{(2k-1)})$, they give a regular bipartite $4$-colored graph
$(\Gamma,\gamma)$ 
which yields  $(\langle S \mid R\rangle, R)$ (cf. Lemma \ref{lemma:unique-lens2}). 
Therefore, by Theorem \ref{theorem:algorithm}, $(\Gamma, \gamma)$ (cf. Figure \ref{fig:lens2}) is a crystallization of a
closed connected orientable $3$-manifold and minimal with respect to $(\langle S \mid R\rangle, R)$.
Here $(\Gamma, \gamma)$ is a $4(k+q-1)$-vertex crystallization of $L((k-1)q-1,q)$ as $(\Gamma, \gamma)$ is isomorphic to the graph  $\mathcal{N}_{k-1,q}$  (cf. \cite[Subsection 5.2]{bd14}) and  $\mathcal{N}_{k-1,q}$  is a crystallization of $ L((k-1)q-1,q)$.

}
\end{eg}

\begin{lemma}\label{lemma:unique-lens2}
Let the presentation $(\langle S \mid R\rangle, R)$  and $q_1, q_2, q_3, P_5$ be as in Example $\ref{eg:lens2}$. 
Then, the choice of the triplet 
$(q_1,q_2,q_3)$ and the path $P_5$  for which the $4$-colored graph $(\Gamma,\gamma)$ yields $(\langle S \mid R\rangle, R)$, is unique.
\end{lemma}

\begin{proof}
Without loss of generality, we choose $x_2^{(2k-1)}$ as the starting vertex for the relation $x_1^{q}x_2^{-1}$.
Thus, either $\{x_1^{(2q-1)}x_3^{(2k)}$, $x_1^{(2q-3)}x_3^{(2k+2)}$, $\dots$, $x_1^{(3)}x_3^{(2q+2k-4)}\}$ or $\{x_1^{(2q-1)}x_3^{(2q+2k-4)}$, $x_1^{(1)}x_3^{(2q+2k-6)}$, 
$\dots$, $x_1^{(2q-5)}x_3^{(2k)}\}$ is the set of edges of color $3$. By similar arguments as in the proof of Lemma \ref{lemma:unique-lens1}, we get
a relation $x_2^{k}x_1^{-1}$ in the first case. Therefore, we have to choose the second case. Again, by similar arguments as in the proof of Lemma \ref{lemma:unique-lens1}, the lemma follows.
\end{proof}

\begin{eg}[\textbf{\boldmath{Crystallization of a hyperbolic $3$-manifold}}]\label{eg:hyperbolic}
{\rm Let $\langle S \mid R\rangle$ be a presentation of a group, where $S=\{x_1,x_2\}$ and $R=$
$\{x_1^4x_2^{-1}x_1x_2^{-1}x_1^4x_2^{-1}x_1x_2^{-3}x_1x_2^{-1}, x_1^3x_2^2x_1^3x_2^{-1}x_1$ $x_2^{-1}x_1^4x_2^{-1}x_1x_2^{-1}\}$.
It is not difficult to prove that, $x_2^2x_1^3x_2^2x_1^{-1}x_2x_1^{-1} \in \overline{R}$ is the only independent element of minimum weight.

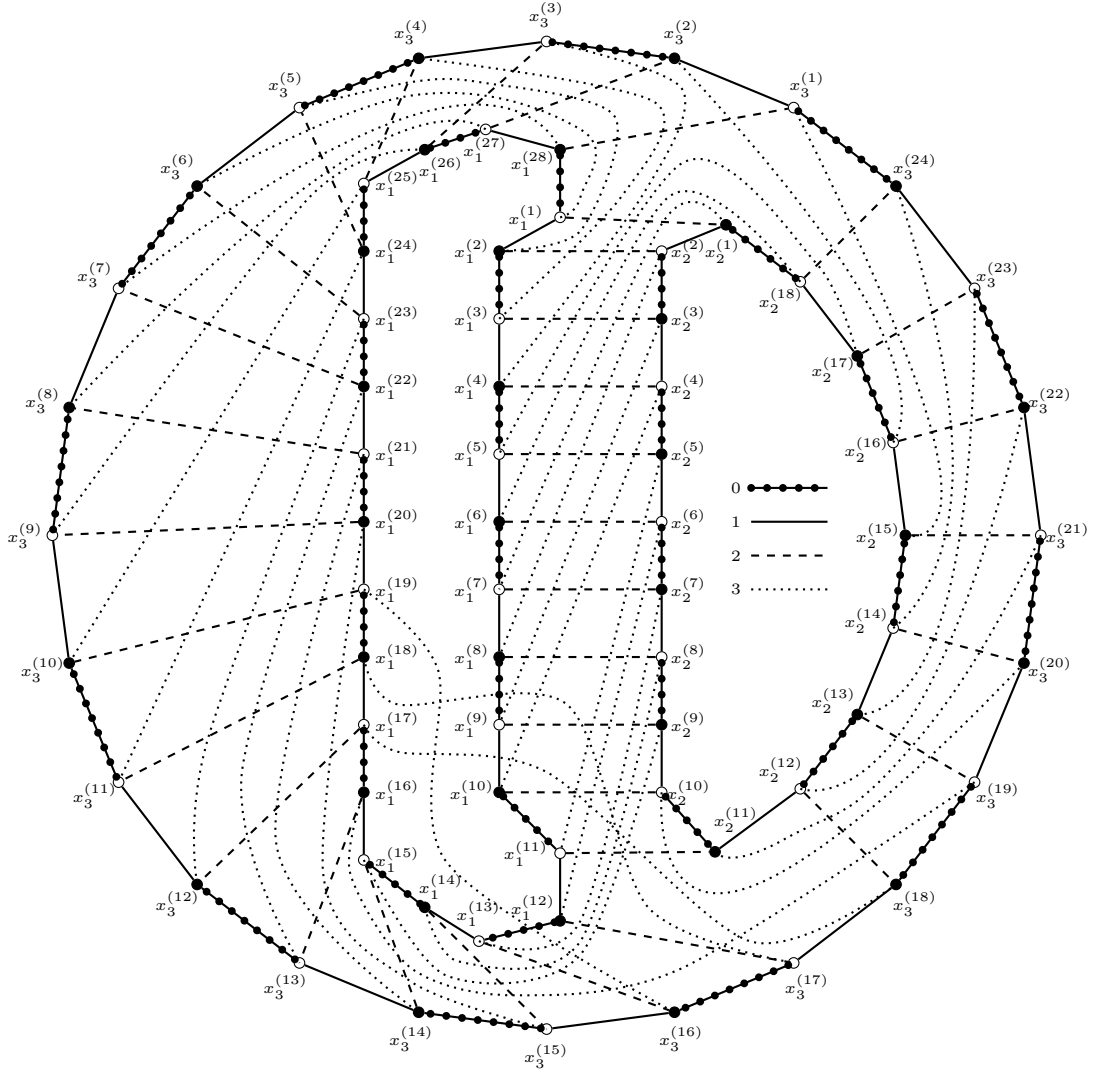
\begin{figure}[ht]
\tikzstyle{vert}=[circle, draw, fill=black!100, inner sep=0pt, minimum width=4pt]
\tikzstyle{vertex}=[circle, draw, fill=black!00, inner sep=0pt, minimum width=4pt]
\tikzstyle{ver}=[]
\tikzstyle{extra}=[circle, draw, fill=black!50, inner sep=0pt, minimum width=2pt]
\tikzstyle{edge} = [draw,thick,-]

\centering
\begin{tikzpicture}[scale=0.9]

\begin{scope}[]
\foreach \x/\y in {0/x_{3}^{(21)},30/x_{3}^{(23)},60/x_{3}^{(1)},90/x_{3}^{(3)},120/x_{3}^{(5)},150/x_{3}^{(7)},180/x_{3}^{(9)},210/x_{3}^{(11)},240/x_{3}^{(13)},270/x_{3}^{(15)},300/x_{3}^{(17)},330/x_{3}^{(19)}}{
\node[ver] (\y) at (\x:7.7){\tiny{$\y$}};
\node[vertex] (\y) at (\x:7.3){};
}
\foreach \x/\y in {15/x_{3}^{(22)},45/x_{3}^{(24)},75/x_{3}^{(2)},105/x_{3}^{(4)},135/x_{3}^{(6)},165/x_{3}^{(8)},195/x_{3}^{(10)},225/x_{3}^{(12)},255/x_{3}^{(14)},285/x_{3}^{(16)},315/x_{3}^{(18)},345/x_{3}^{(20)}}{
\node[ver] (\y) at (\x:7.7){\tiny{$\y$}};
\node[vert] (\y) at (\x:7.3){};
}
\foreach \x/\y in {15/x_2^{(16)},45/x_2^{(18)},315/x_2^{(12)},345/x_2^{(14)}}{
\node[ver] (\y) at (\x:4.9){\tiny{$\y$}};
\node[vertex] (\y) at (\x:5.3){};
}
\foreach \x/\y in {0/x_2^{(15)},30/x_2^{(17)},330/x_2^{(13)}}{
\node[ver] (\y) at (\x:4.9){\tiny{$\y$}};
\node[vert] (\y) at (\x:5.3){};
}
\foreach \x/\y/\z/\w in {1.7/3.2/2.1/x_2^{(3)},1.7/1.2/2.1/x_2^{(5)},1.7/-0.8/2.1/x_2^{(7)},1.7/-2.8/2.1/x_2^{(9)}}{
\node[ver] (\w) at (\z,\y){\tiny{$\w$}};
\node[vert] (\w) at (\x,\y){};
}
\foreach \x/\y/\z/\w in {1.7/4.2/2.1/x_2^{(2)},1.7/2.2/2.1/x_2^{(4)},1.7/0.2/2.1/x_2^{(6)},1.7/-1.8/2.1/x_2^{(8)},1.7/-3.8/2.1/x_2^{(10)}}{
\node[ver] (\w) at (\z,\y){\tiny{$\w$}};
\node[vertex] (\w) at (\x,\y){};
}
\node[ver] () at (2.6,4.2){\tiny{$x_2^{(1)}$}};
\node[ver] () at (2.8,-4.2){\tiny{$x_2^{(11)}$}};
\node[vert] (x_2^{(1)}) at (60:5.3){};
\node[vert] (x_2^{(11)}) at (298:5.3){};
\foreach \x/\y/\z/\w in {0.2/4.7/-0.3/x_1^{(1)},-0.7/3.2/-1.1/x_{1}^{(3)},-0.7/1.2/-1.1/x_{1}^{(5)},-0.7/-0.8/-1.1/x_1^{(7)},-0.7/-2.8/-1.1/x_1^{(9)},0.2/-4.7/-0.3/x_1^{(11)}}{
\node[ver] (\w) at (\z,\y){\tiny{$\w$}};
\node[vertex] (\w) at (\x,\y){};
}
\foreach \x/\y/\z/\w in {-0.7/4.2/-1.1/x_1^{(2)},-0.7/2.2/-1.1/x_1^{(4)},-0.7/0.2/-1.1/x_1^{(6)},-0.7/-1.8/-1.1/x_1^{(8)},-0.7/-3.8/-1.1/x_1^{(10)}}{
\node[ver] (\w) at (\z,\y){\tiny{$\w$}};
\node[vert] (\w) at (\x,\y){};
}
\foreach \x/\y/\z/\w/\q in {0.2/5.7/-0.2/x_1^{(28)}/5.5, 0.2/-5.7/-0.2/x_1^{(12)}/-5.5}{
\node[ver] (\w) at (\z,\q){\tiny{$\w$}};
\node[vert] (\w) at (\x,\y){};
}
\foreach \x/\y/\z/\w in {-2.7/3.2/-2.2/x_1^{(23)},-2.7/1.2/-2.2/x_1^{(21)},-2.7/-0.8/-2.2/x_1^{(19)},-2.7/-2.8/-2.2/x_1^{(17)},-2.7/-4.8/-2.2/x_1^{(15)},
-2.7/5.2/-2.2/x_1^{(25)}}{
\node[ver] (\w) at (\z,\y){\tiny{$\w$}};
\node[vertex] (\w) at (\x,\y){};
}
\foreach \x/\y/\z/\w in {-2.7/4.2/-2.2/x_1^{(24)},-2.7/2.2/-2.2/x_1^{(22)},-2.7/0.2/-2.2/x_1^{(20)},-2.7/-1.8/-2.2/x_1^{(18)},-2.7/-3.8/-2.2/x_1^{(16)}}{
\node[ver] (\w) at (\z,\y){\tiny{$\w$}};
\node[vert] (\w) at (\x,\y){};
}

\foreach \x/\y/\z/\w in {-0.9/6/5.7/x_1^{(27)},-1/-6/-5.6/x_{1}^{(13)}}{
\node[ver] (\w) at (\x,\z){\tiny{$\w$}};
\node[vertex] (\w) at (\x,\y){};
}
\foreach \x/\y/\z/\w in {-1.8/-5.5/-5.2/x_1^{(14)},-1.8/5.7/5.4/x_1^{(26)}}{
\node[ver] (\w) at (\x,\z){\tiny{$~~~~~\w$}};
\node[vert] (\w) at (\x,\y){};
}
\end{scope}

\begin{scope}[shift={(2.8,-0.8)}]
\node[ver] (3) at (0,0){\tiny{$3$}}; 
\node[ver] (2) at (0,0.5){\tiny{$2$}};
\node[ver](1) at (0,1){\tiny{$1$}}; 
\node[ver](0) at (0,1.5){\tiny{$0$}}; 
\node[ver] (7) at (1.5,0){}; 
\node[ver](6) at (1.5,0.5){}; 
\node[ver](5) at (1.5,1){}; 
\node[ver] (4) at (1.5,1.5){};
\end{scope}

\foreach \x/\y in {x_1^{(2)}/x_{1}^{(3)},x_{1}^{(3)}/x_1^{(4)},x_1^{(4)}/x_{1}^{(5)},x_{1}^{(5)}/x_1^{(6)},x_1^{(6)}/x_1^{(7)},x_1^{(7)}/x_1^{(8)},x_1^{(8)}/x_1^{(9)},x_1^{(9)}/x_1^{(10)},x_1^{(10)}/x_1^{(11)},x_1^{(11)}/x_1^{(12)},
x_1^{(12)}/x_{1}^{(13)},x_{1}^{(13)}/x_1^{(14)},x_1^{(14)}/x_1^{(15)},x_1^{(15)}/x_1^{(16)},x_1^{(16)}/x_1^{(17)},x_1^{(17)}/x_1^{(18)},x_1^{(18)}/x_1^{(19)},x_1^{(19)}/x_1^{(20)},
x_1^{(20)}/x_1^{(21)},x_1^{(21)}/x_1^{(22)},x_1^{(22)}/x_1^{(23)},x_1^{(23)}/x_1^{(24)},x_1^{(24)}/x_1^{(25)},x_1^{(25)}/x_1^{(26)},x_1^{(26)}/x_1^{(27)},x_1^{(27)}/x_1^{(28)},x_1^{(28)}/x_1^{(1)},x_1^{(1)}/x_1^{(2)},0/4,1/5}{
\path[edge] (\x) -- (\y);}

\foreach \x/\y in {x_1^{(2)}/x_{1}^{(3)},x_1^{(4)}/x_{1}^{(5)},x_1^{(6)}/x_1^{(7)},x_1^{(8)}/x_1^{(9)},x_1^{(10)}/x_1^{(11)},
x_1^{(12)}/x_{1}^{(13)},x_1^{(14)}/x_1^{(15)},x_1^{(17)}/x_1^{(16)},x_1^{(19)}/x_1^{(18)},x_1^{(21)}/x_1^{(20)},
x_1^{(23)}/x_1^{(22)},x_1^{(25)}/x_1^{(24)},x_1^{(26)}/x_1^{(27)},x_1^{(28)}/x_1^{(1)},0/4}{
\draw [line width=3pt, line cap=round, dash pattern=on 0pt off 2\pgflinewidth] 	(\x) -- (\y);}

\foreach \x/\y in {x_2^{(2)}/x_2^{(3)},x_2^{(3)}/x_2^{(4)},x_2^{(4)}/x_2^{(5)},x_2^{(5)}/x_2^{(6)},x_2^{(6)}/x_2^{(7)},x_2^{(7)}/x_2^{(8)},x_2^{(8)}/x_2^{(9)},x_2^{(9)}/x_2^{(10)},x_2^{(10)}/x_2^{(11)},x_2^{(11)}/x_2^{(12)},
x_2^{(12)}/x_2^{(13)},x_2^{(13)}/x_2^{(14)},x_2^{(14)}/x_2^{(15)},x_2^{(15)}/x_2^{(16)},x_2^{(16)}/x_2^{(17)},x_2^{(17)}/x_2^{(18)},x_2^{(18)}/x_2^{(1)},
x_2^{(1)}/x_2^{(2)}}{
\path[edge] (\x) -- (\y);}

\foreach \x/\y in {x_2^{(2)}/x_2^{(3)},x_2^{(4)}/x_2^{(5)},x_2^{(6)}/x_2^{(7)},x_2^{(8)}/x_2^{(9)},x_2^{(10)}/x_2^{(11)},
x_2^{(12)}/x_2^{(13)},x_2^{(14)}/x_2^{(15)},x_2^{(16)}/x_2^{(17)},x_2^{(18)}/x_2^{(1)}}{
\draw [line width=3pt, line cap=round, dash pattern=on 0pt off 2\pgflinewidth] 	(\x) -- (\y);}

\foreach \x/\y in {x_{3}^{(21)}/x_{3}^{(22)},x_{3}^{(22)}/x_{3}^{(23)},x_{3}^{(23)}/x_{3}^{(24)},x_{3}^{(24)}/x_{3}^{(1)},x_{3}^{(1)}/x_{3}^{(2)},x_{3}^{(2)}/x_{3}^{(3)},x_{3}^{(3)}/x_{3}^{(4)},x_{3}^{(4)}/x_{3}^{(5)},x_{3}^{(5)}/x_{3}^{(6)},x_{3}^{(6)}/x_{3}^{(7)},
x_{3}^{(7)}/x_{3}^{(8)},x_{3}^{(8)}/x_{3}^{(9)},x_{3}^{(9)}/x_{3}^{(10)},x_{3}^{(10)}/x_{3}^{(11)},x_{3}^{(11)}/x_{3}^{(12)},x_{3}^{(12)}/x_{3}^{(13)},x_{3}^{(13)}/x_{3}^{(14)},
x_{3}^{(14)}/x_{3}^{(15)},x_{3}^{(15)}/x_{3}^{(16)},x_{3}^{(16)}/x_{3}^{(17)},x_{3}^{(17)}/x_{3}^{(18)},x_{3}^{(18)}/x_{3}^{(19)},x_{3}^{(19)}/x_{3}^{(20)},x_{3}^{(20)}/x_{3}^{(21)}}{
\path[edge] (\x) -- (\y);}

\foreach \x/\y in {x_{3}^{(23)}/x_{3}^{(22)},x_{3}^{(1)}/x_{3}^{(24)},x_{3}^{(3)}/x_{3}^{(2)},x_{3}^{(5)}/x_{3}^{(4)},x_{3}^{(7)}/x_{3}^{(6)},
x_{3}^{(9)}/x_{3}^{(8)},x_{3}^{(11)}/x_{3}^{(10)},x_{3}^{(13)}/x_{3}^{(12)},x_{3}^{(15)}/x_{3}^{(14)},
x_{3}^{(17)}/x_{3}^{(16)},x_{3}^{(19)}/x_{3}^{(18)},x_{3}^{(21)}/x_{3}^{(20)}}{
\draw [line width=3pt, line cap=round, dash pattern=on 0pt off 2\pgflinewidth] 	(\x) -- (\y);}

\foreach \x/\y in {x_1^{(2)}/x_2^{(2)},x_{1}^{(3)}/x_2^{(3)},x_1^{(4)}/x_2^{(4)},x_{1}^{(5)}/x_2^{(5)},x_1^{(6)}/x_2^{(6)},x_1^{(7)}/x_2^{(7)},x_1^{(8)}/x_2^{(8)},x_1^{(9)}/x_2^{(9)},x_1^{(10)}/x_2^{(10)},x_2^{(18)}/x_{3}^{(24)},x_2^{(17)}/x_{3}^{(23)},
x_2^{(16)}/x_{3}^{(22)},x_2^{(15)}/x_{3}^{(21)},x_2^{(12)}/x_{3}^{(18)},x_2^{(13)}/x_{3}^{(19)},x_2^{(14)}/x_{3}^{(20)},x_1^{(11)}/x_2^{(11)},x_1^{(1)}/x_2^{(1)},x_1^{(28)}/x_{3}^{(1)},x_1^{(12)}/x_{3}^{(17)},
x_{1}^{(13)}/x_{3}^{(16)},x_1^{(14)}/x_{3}^{(15)},x_1^{(15)}/x_{3}^{(14)},x_1^{(16)}/x_{3}^{(13)},x_1^{(17)}/x_{3}^{(12)},x_1^{(18)}/x_{3}^{(11)},x_1^{(19)}/x_{3}^{(10)},x_1^{(20)}/x_{3}^{(9)},
x_1^{(21)}/x_{3}^{(8)},x_1^{(22)}/x_{3}^{(7)},x_1^{(23)}/x_{3}^{(6)},x_1^{(24)}/x_{3}^{(5)},x_1^{(25)}/x_{3}^{(4)},x_1^{(26)}/x_{3}^{(3)},x_1^{(27)}/x_{3}^{(2)},2/6}{
\path[edge, dashed] (\x) -- (\y);}

\foreach \x/\y in {x_1^{(25)}/x_{3}^{(10)},x_1^{(24)}/x_{3}^{(11)},x_1^{(11)}/x_2^{(5)},x_1^{(10)}/x_2^{(4)},x_1^{(9)}/x_2^{(3)},x_1^{(8)}/x_2^{(2)},x_1^{(12)}/x_2^{(6)},3/7}{
\path[edge, dotted] (\x) -- (\y);}
\draw[edge, dotted] plot [smooth,tension=1] coordinates{(x_1^{(28)})(-2,6.2)(x_{3}^{(7)})};
\draw[edge, dotted] plot [smooth,tension=1] coordinates{(x_1^{(27)})(-3.2,5.4)(x_{3}^{(8)})};
\draw[edge, dotted] plot [smooth,tension=0.5] coordinates{(x_1^{(26)})(-3.5,5)(x_{3}^{(9)})};
\draw[edge, dotted] plot [smooth,tension=0.5] coordinates{(x_{1}^{(13)})(0.5,-5.7)(x_2^{(7)})};
\draw[edge, dotted] plot [smooth,tension=0.5] coordinates{(x_1^{(14)})(-1.3,-6.3)(0.5,-6)(x_2^{(8)})};
\draw[edge, dotted] plot [smooth,tension=0.5] coordinates{(x_1^{(15)})(-1.5,-6.5)(0.5,-6.2)(x_2^{(9)})};
\draw[edge, dotted] plot [smooth,tension=1] coordinates{(x_2^{(10)})(3,-5.2)(x_{3}^{(20)})};
\draw[edge, dotted] plot [smooth,tension=1] coordinates{(x_2^{(11)})(3.5,-4.5)(5.5,-3)(x_{3}^{(21)})};
\draw[edge, dotted] plot [smooth,tension=1] coordinates{(x_2^{(12)})(5.5,-2.5)(x_{3}^{(22)})};
\draw[edge, dotted] plot [smooth,tension=1] coordinates{(x_2^{(13)})(5.9,-1)(x_{3}^{(23)})};
\draw[edge, dotted] plot [smooth,tension=1] coordinates{(x_2^{(14)})(6,1)(x_{3}^{(24)})};
\draw[edge, dotted] plot [smooth,tension=1] coordinates{(x_2^{(15)})(5.7,1.5)(x_{3}^{(1)})};
\draw[edge, dotted] plot [smooth,tension=1] coordinates{(x_2^{(16)})(4.7,3.5)(x_{3}^{(2)})};
\draw[edge, dotted] plot [smooth,tension=0.5] coordinates{(x_2^{(17)})(4,4)(2,5.5)(x_{1}^{(5)})};
\draw[edge, dotted] plot [smooth,tension=0.5] coordinates{(x_2^{(18)})(1.8,5)(x_1^{(6)})};
\draw[edge, dotted] plot [smooth,tension=0.5] coordinates{(x_2^{(1)})(1.7,4.6)(x_1^{(7)})};
\draw[edge, dotted] plot [smooth,tension=0.5] coordinates{(x_1^{(4)})(1.4,5)(2,6.6)(x_{3}^{(3)})};
\draw[edge, dotted] plot [smooth,tension=0.5] coordinates{(x_{1}^{(3)})(1.1,5)(1.6,6.4)(x_{3}^{(4)})};
\draw[edge, dotted] plot [smooth,tension=0.7] coordinates{(x_1^{(1)})(0.5,5)(0.5,6)(-2,6.5)(x_{3}^{(6)})};
\draw[edge, dotted] plot [smooth,tension=0.5] coordinates{(x_1^{(2)})(0.5,4.5)(1,6.2)(-1,6.7)(-2,6.7)(x_{3}^{(5)})};
\draw[edge, dotted] plot [smooth,tension=0.5] coordinates{(x_1^{(16)})(-2.9,-5)(-2,-6.4)(-0.5,-6.8)(2,-6.4)(x_{3}^{(19)})};
\draw[edge, dotted] plot [smooth,tension=0.5] coordinates{(x_1^{(17)})(-2.2,-3.2)(0,-3.5)(3,-6)(x_{3}^{(18)})};
\draw[edge, dotted] plot [smooth,tension=0.5] coordinates{(x_1^{(18)})(-2.2,-2.5)(0.3,-2.5)(1.5,-5.5)(x_{3}^{(17)})};
\draw[edge, dotted] plot [smooth,tension=0.5] coordinates{(x_1^{(19)})(-1.6,-2)(-1.5,-4.8)(x_{3}^{(16)})};
\draw[edge, dotted] plot [smooth,tension=0.5] coordinates{(x_1^{(20)})(-3.3,-5)(-2,-6.6)(x_{3}^{(15)})};
\draw[edge, dotted] plot [smooth,tension=0.5] coordinates{(x_1^{(21)})(-4,-4.5)(x_{3}^{(14)})};
\draw[edge, dotted] plot [smooth,tension=0.5] coordinates{(x_1^{(22)})(-4.5,-3.5)(x_{3}^{(13)})};
\draw[edge, dotted] plot [smooth,tension=0.5] coordinates{(x_1^{(23)})(-5,-3)(x_{3}^{(12)})};
\end{tikzpicture}
\caption{Crystallization of a hyperbolic $3$-manifold}\label{fig:hyperbolic}
\end{figure}

Therefore, $\mathcal{R}=R\cup \{x_2^2x_1^3x_2^2x_1^{-1}x_2x_1^{-1}\}$. Thus, $m^{(c)}_{12}=11$, $m^{(c)}_{13}=17$, $m^{(c)}_{23}=7$ for 
$2 \leq c \leq 3$, $(n_1,n_2,n_3)=(28,18,24)$ and $G_i=C_{n_{i}}(x^{(1)}_i, \dots, x^{(n_i)}_i)$ for $1\leq i \leq 3$ as in  Figure \ref{fig:hyperbolic}. 
Choose $(n_i,n_j,n_l)=(n_1,n_3,n_2)$. Thus, $\{x^{(1)}_1x^{(1)}_2,\dots, x_1^{(11)}x^{(11)}_2\}$, $\{x^{(28)}_1x^{(1)}_3,\dots, x^{(12)}_1x^{(17)}_3\}$ 
and $\{x^{(24)}_3x^{(18)}_2,\dots, x^{(18)}_3x^{(12)}_2\}$ are the sets of edges of color $2$. Here the only choice for the triplet $(q_1,q_2,q_3)$ is $(8,5,10)$ and for the path $P_5$ is  
$P_5(x_3^{(20)}, x_3^{(19)}, x_1^{(16)}, x_1^{(15)},x_2^{(9)}, x_2^{(10)})$, they give a  bipartite $4$-colored graph $(\Gamma, \gamma)$ (cf. Figure \ref{fig:hyperbolic}) which yields  $(\langle S \mid R\rangle, R)$ (cf. Lemma \ref{lemma:unique-hyperbolic}). Since $\langle S \mid R\rangle$ is not a free product of two non trivial group, $|\mathcal{K}(\Gamma)|$ is prime.
Now, $\langle S \mid R\rangle\cong \langle x_1, x_2 \mid x_1^4x_2^{-1}x_1x_2^{-1}x_1^4x_2^{-1}x_1x_2^{-3}x_1x_2^{-1}, x_2^2x_1^3x_2^2x_1^{-1}x_2x_1^{-1} \rangle$
$\cong \langle x_1, x_2 \mid  x_1x_2^{3}x_1x_2x_1^4x_2x_1x_2x_1^4x_2, x_1x_2x_1x_2^2x_1^{-3}x_2^2 \rangle$ which is the presentation of the 
fundamental group of a closed, connected orientable prime hyperbolic $3$-manifold (see the presentation in \texttt{http://www.dms.umontreal.ca/$\sim$math/Logiciels/Magma/text414.htm}). 
Therefore,  by Proposition \ref{prop:hyperbolic}, $|\mathcal{K}(\Gamma)|$ is homeomorphic to the hyperbolic $3$-manifold as in the list.
Thus, using the algorithm we get a crystallization of the hyperbolic $3$-manifold from a given presentation.

Observe that (i) $\#V(\Gamma)= 70=28+26+16 = \lambda(\langle S \mid R\rangle, R)$, and (ii)
$\Gamma_{\{i,j\}}=16C_4 \sqcup C_6$ for $0\leq i \leq 1, 2\leq j \leq3$ as $\#V(\Gamma)=70=4 \times 17+2$.
}
\end{eg}

\begin{lemma}\label{lemma:unique-hyperbolic}
Let the presentation $(\langle S \mid R\rangle, R)$  and $q_1, q_2, q_3, P_5$ be as in Example $\ref{eg:hyperbolic}$. 
Then, the choice of the triplet 
$(q_1,q_2,q_3)$ and the path $P_5$  for which the $4$-colored graph $(\Gamma,\gamma)$ yields $(\langle S \mid R\rangle, R)$, is unique.
\end{lemma}

\begin{proof}
From the discussions in the proofs of previous lemmas it is clear that, if $x_j^{\varepsilon_1} x_i^{\pm m} x_j^{\varepsilon_2}$ is a part of a relation for some $\varepsilon_1, \varepsilon_2 \in \{1,-1\}$, $m \geq 2$ and $1\leq i \neq j \leq 2$ then to yield $x_j^{\varepsilon_1} x_i^{\pm m} x_j^{\varepsilon_2}$, there are exactly $m-1$ vertices of $G_3$ have both the 2-adjacent vertices and the 3-adjacent vertices in $G_i$. There are three different words of type $x_2^{\varepsilon_1} x_1^4 x_2^{\varepsilon_2}$
and three different words of type $x_2^{\varepsilon_1} x_1^3 x_2^{\varepsilon_2}$ in the relations in $\mathcal{R}$. Therefore,
exactly $(3(4-1)+3(3-1)=)$ 15 vertices of $G_3$ have both the 2-adjacent vertices and the 3-adjacent vertices in $G_1$. By similar arguments as above,
$((3-1)+3(2-1)=)$ 5 vertices of $G_3$ have both the 2-adjacent vertices and the 3-adjacent vertices in $G_2$. Since there are
eight different words of type $x_1  x_2^{-1} x_1$ and one word of type $x_1^{-1}  x_2 x_1^{-1}$ in the relations in $\mathcal{R}$, exactly $(8+1=)$ 9 vertices of $G_2$ have both the 2-adjacent vertices and the 3-adjacent vertices in  $G_1$.
Since $m_{13}^{(2)}=m_{13}^{(3)}=17$ and $15$ vertices of $G_3$ have both the 2-adjacent vertices and the 3-adjacent vertices in $G_1$,
$x_3^{(1+m)}, \dots,x_3^{(15+m)}$, $0\leq m \leq 2$ are the choices of these vertices. If $m=1$ then, either $x_3^{(1)}$
or $x_3^{(17)}$ is joined to $G_1$ with an edge of color $3$ as $m^{(3)}_{13}=17$. Since exactly $15$ vertices of $G_3$ are joined to $G_1$ with both edges of colors $2$ and $3$, $m \neq 1$.
Up to an automorphism, we can assume $x_3^{(1)}$ is not such a vertex, i.e., $m = 2$. 
Therefore,  $x_3^{(1)}$ and $x_3^{(2)}$ are joined to $G_2$ with edges of color $3$. Let $x^{(3)}_3x^{(26+n)}_1,\dots, x^{(19)}_3x^{(10+n)}_1 \in \gamma^{-1}(3)$ for some integer $n$. If $-4\leq n \leq 4$ then there is a relation containing $x_1^5$ and if $n \geq 8$ or
$n \leq -8$ then there is no relation containing $x_1^4$ and therefore both cases are not possible. Therefore, $n = \pm 6$.
If $m=2$ and $n=-6$ then $x^{(3)}_3x^{(20)}_1,\dots, x^{(19)}_3x^{(4)}_1 \in \gamma^{-1}(3)$. This contradicts that $9$ vertices of $G_2$ have both the 2-adjacent vertices and the 3-adjacent vertices in  $G_1$. Therefore, $m=2$ and $n=6$ and hence
$x^{(3)}_3x^{(4)}_1,x^{(4)}_3x^{(3)}_1,\dots, x^{(19)}_3x^{(16)}_1 \in \gamma^{-1}(3)$.
 Since $x^{(18)}_3$ and $x^{(19)}_3$ are already joined to $G_1$ with edges of color $3$, the remaining
vertices $x^{(20)}_3,\dots, x^{(24)}_3$ are joined to $G_2$ with both edges of colors $2$ and $3$. Since
the relations with starting vertices $x^{(18)}_3$ and $x^{(19)}_3$  yield $x_2^2w$ for some $w \in F(S)$,  
$x^{(20)}_3x^{(14+k)}_2,\dots, x^{(24)}_3x^{(18+k)}_2 \in \gamma^{-1}(3)$ for $k=-4$. So, $x^{(15)}_1x^{(9)}_2 \in \gamma^{-1}(3)$. Thus, the lemma follows.
\end{proof}

\begin{remark} 
{\rm If $(\Gamma, \gamma)$ is a crystallization of a $3$-manifold then  the regular genus of $\Gamma$ is the integer
$\rho(\Gamma) = \min \{g_{01}, g_{02}, g_{03}\} -1$ (cf. \cite[Section 4]{bcg13}). The crystallizations constructed in Example \ref{eg:polyhedral} for $n=k=2$ and in Examples
\ref{eg:lens1} and \ref{eg:lens2} are crystallizations of handle-free manifolds. Thus, by \cite[Proposition 4]{ca99} and from the catalogue in \cite{bcg13}, 
these crystallizations are vertex-minimal regular genus two crystallizations
when the number of vertices of the crystallizations are at most $42$. The crystallizations constructed in Examples \ref{eg:lens1} and \ref{eg:lens2} are
vertex-minimal for all known cases. In fact, the crystallizations of $L((k-1)q+1, q)$ are vertex-minimal when $(k-1)q+1$ are even
(cf. \cite{bd14, cc14, sw13}).
}
\end{remark}

\subsection{Non existence of some crystallizations}

Here we consider the cases where Algorithm $1$ determines the non existence of any crystallization for a pair $(\langle S \mid R\rangle, R)$.

\begin{eg}[\textbf{\boldmath{For a presentation of $\mathbb{Z}_6$}}]\label{eg:no}
{\rm Let $(\langle S \mid R\rangle, R)$ be a presentation of the cyclic group $\mathbb{Z}_6$, where $S=\{x_1,x_2\}$ and $R= \{x_1^3x_2^{-1}, x_1^3x_2\}$.
Clearly, $x_2^2$ is the only independent element in $\overline{R}$ of minimum weight. Therefore,  $\mathcal{R}=R\cup \{x_2^2\}$ and
let $(\Gamma,\gamma)$ be  a crystallization realizing the above presentation. Thus, $m^{(c)}_{12}=1$, $m^{(c)}_{13}=5$, $m^{(c)}_{23}=3$ for $2\leq c \leq3$ and $(n_1,n_2,n_3)=(6,4,8)$. Therefore,
by choosing $(n_i,n_j,n_l)=(n_3,n_1,n_2)$, we have $\Gamma_{\{0,1,2\}}$ as in Figure \ref{fig:Z6}.

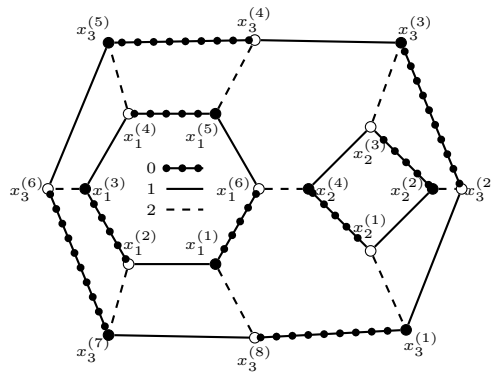
\begin{figure}[ht]
\tikzstyle{vert}=[circle, draw, fill=black!100, inner sep=0pt, minimum width=4pt]
\tikzstyle{vertex}=[circle, draw, fill=black!00, inner sep=0pt, minimum width=4pt]
\tikzstyle{ver}=[]
\tikzstyle{extra}=[circle, draw, fill=black!50, inner sep=0pt, minimum width=2pt]
\tikzstyle{edge} = [draw,thick,-]

\centering
\begin{tikzpicture}[scale=0.55]

\begin{scope}[]
\foreach \x/\y in {0/x_{3}^{(2)},180/x_{3}^{(6)}}{
\node[ver] (\y) at (\x:5.5){\tiny{$\y$}};
\node[vertex] (\y) at (\x:5){};
}

\foreach \x/\y in {90/x_{3}^{(4)},270/x_{3}^{(8)}}{
\node[ver] (\y) at (\x:4.1){\tiny{$\y$}};
\node[vertex] (\y) at (\x:3.6){};
}
\foreach \x/\y in {45/x_{3}^{(3)},135/x_{3}^{(5)},225/x_{3}^{(7)},317/x_{3}^{(1)}}{
\node[ver] (\y) at (\x:5.5){\tiny{$\y$}};
\node[vert] (\y) at (\x:5){};
}
\end{scope}

\begin{scope}[shift={(-2,0)}]
\foreach \x/\y in {0/x_{1}^{(6)},120/x_{1}^{(4)},240/x_{1}^{(2)}}{
\node[ver] (\y) at (\x:1.5){\tiny{$\y$}};
\node[vertex] (\y) at (\x:2.1){};
}
\foreach \x/\y in {60/x_{1}^{(5)},180/x_{1}^{(3)},300/x_{1}^{(1)}}{
\node[ver] () at (\x:1.5){\tiny{$\y$}};
\node[vert] (\y) at (\x:2.1){};
}
\end{scope}

\begin{scope}[shift={(2.8,0)}]
\foreach \x/\y in {90/x_{2}^{(3)},270/x_{2}^{(1)}}{
\node[ver] (\y) at (\x:0.9){\tiny{$\y$}};
\node[vertex] (\y) at (\x:1.5){};
}
\foreach \x/\y in {0/x_{2}^{(2)},180/x_{2}^{(4)}}{
\node[ver] () at (\x:0.9){\tiny{$\y$}};
\node[vert] (\y) at (\x:1.5){};
}
\end{scope}

\begin{scope}[shift={(-2.5,-1)}]
 
\node[ver] (2) at (0,0.5){\tiny{$2$}};
\node[ver](1) at (0,1){\tiny{$1$}}; 
\node[ver](0) at (0,1.5){\tiny{$0$}};  
\node[ver](6) at (1.5,0.5){}; 
\node[ver](5) at (1.5,1){}; 
\node[ver] (4) at (1.5,1.5){};
\end{scope}

\foreach \x/\y in {x_{1}^{(1)}/x_{1}^{(2)},x_{1}^{(2)}/x_{1}^{(3)},x_{1}^{(3)}/x_{1}^{(4)},x_{1}^{(4)}/x_{1}^{(5)},x_{1}^{(5)}/x_{1}^{(6)},x_{1}^{(6)}/x_{1}^{(1)},
x_{2}^{(1)}/x_{2}^{(2)},x_{2}^{(2)}/x_{2}^{(3)},x_{2}^{(3)}/x_{2}^{(4)},x_{2}^{(4)}/x_{2}^{(1)},
x_{3}^{(1)}/x_{3}^{(2)},x_{3}^{(2)}/x_{3}^{(3)},x_{3}^{(3)}/x_{3}^{(4)},x_{3}^{(4)}/x_{3}^{(5)},x_{3}^{(5)}/x_{3}^{(6)},x_{3}^{(6)}/x_{3}^{(7)},x_{3}^{(7)}/x_{3}^{(8)},x_{3}^{(8)}/x_{3}^{(1)},0/4,1/5}{
\path[edge] (\x) -- (\y);}

\foreach \x/\y in {x_{1}^{(2)}/x_{1}^{(3)},x_{1}^{(4)}/x_{1}^{(5)},x_{1}^{(6)}/x_{1}^{(1)},x_{2}^{(2)}/x_{2}^{(3)},x_{2}^{(4)}/x_{2}^{(1)},
x_{3}^{(2)}/x_{3}^{(3)},x_{3}^{(4)}/x_{3}^{(5)},x_{3}^{(6)}/x_{3}^{(7)},x_{3}^{(8)}/x_{3}^{(1)},0/4}{
\draw [line width=3pt, line cap=round, dash pattern=on 0pt off 2\pgflinewidth] 	(\x) -- (\y);}

\foreach \x/\y in {x_{1}^{(1)}/x_{3}^{(8)},x_{1}^{(2)}/x_{3}^{(7)},x_{1}^{(3)}/x_{3}^{(6)},x_{1}^{(4)}/x_{3}^{(5)},x_{1}^{(5)}/x_{3}^{(4)},x_{1}^{(6)}/x_{2}^{(4)},
x_{2}^{(1)}/x_{3}^{(1)},x_{2}^{(2)}/x_{3}^{(2)},x_{2}^{(3)}/x_{3}^{(3)},2/6}{
\path[edge, dashed] (\x) -- (\y);}
\end{tikzpicture}
\caption{The graph $\Gamma_{\{0,1,2\}}$}\label{fig:Z6}
\end{figure}

If $x_3^{(2q_3-1)}x_3^{(2q_3)}= x_3^{(7)}x_3^{(8)}$, $x_3^{(1)}x_3^{(2)}$ or $x_3^{(5)}x_3^{(6)}$ then, for each of the two choices of the path $P_5$,
either $x_3^{(1)}$ or $x_3^{(3)}$ is joined with $G_1$ with edges of color $3$. Since the graph is bipartite, neither $x_2^{(2)}x_3^{(1)}$ nor $x_2^{(2)}x_3^{(3)}$ can be edge
of color 3 in $\Gamma$.
Therefore, no components of $\Gamma_{\{2,3\}}$ yield the relation $x_2^2$. For the same reason, we can not choose $x_3^{(2q_3-1)}x_3^{(2q_3)}= x_3^{(3)}x_3^{(4)}$
and the path $P_5(x_1^{(2q_1-1)}, x_1^{(2q_1)}, x_3^{(3)}, x_3^{(4)}$, $x_2^{(2q_2)}, x_2^{(2q_2-1)})$. Therefore, the only remaining choice  is $x_3^{(2q_3-1)}x_3^{(2q_3)}= x_3^{(3)}x_3^{(4)}$
and the path $P_5(x_1^{(2q_1)}, x_1^{(2q_1-1)}, x_3^{(4)}, x_3^{(3)}, x_2^{(2q_2-1)}, x_2^{(2q_2)})$. Thus, $x_3^{(2)}$ is joined with $G_2$ and  hence $x_2^{(4)}x_3^{(2)}
\in \gamma^{-1}(3)$ (since $\Gamma$ is bipartite and has no double edge). Then, $P_3(x_1^{(6)}, x_2^{(4)},x_3^{(2)}, x_2^{(2)})$ is a part of a component of 
$\Gamma_{\{2,3\}}$, which yields the word $x_1^{-1}x_2^2$. Since $x_1^{-1}x_2^2$ is not a part of relations in $\mathcal{R}$, this choice also is not possible.
Thus, there is no crystallization of a closed connected $3$-manifold which yields $(\langle S \mid R\rangle, R)$ and is minimal with respect to $(\langle S \mid R\rangle, R)$.
 
}\end{eg}

\begin{eg}[\textbf{\boldmath{For a presentation of $\mathbb{Z}_{mn+n+1}$, $ m,n\geq 1$}}]\label{eg:no2}
{\rm Let $(\langle S \mid R\rangle, R)$ be the presentation of $\mathbb{Z}_{mn+n+1}$, where $S=\{x_1,x_2\}$ and $R= \{ x_1^{m+1}x_2^{-1}, x_1x_2^n\}$ and $m,n \geq 1$.
It is not difficult to prove that, $\{x_1^{m_1}x_2^{-1}x_1^{m_2}x_2^{-n} : m_1+m_2=m, m_1,m_2 \geq 1\}$ is the set of all independent elements in $\overline{R}$ of minimum weight. So, $\mathcal{R}=R\cup \{x_1^{m_1}x_2^{-1}x_1^{m_2}x_2^{-n} \}$ and let $(\Gamma,\gamma)$ be a
crystallization realizing the above presentation. Thus, $m^{(c)}_{12}=3$, $m^{(c)}_{13}=2m-1$, $m^{(c)}_{23}=2n-1$ for $2 \leq c \leq 3$ and $(n_1,n_2,n_3)=(2m+2,2n+2,2m+2n-2)$. By choosing $(n_i,n_j,n_l)=(n_3,n_1,n_2)$, we have $\Gamma_{\{0,1,2\}}$ as in Figure \ref{fig:polyhedral}.
From the discussions in the proof of Lemma \ref{lemma:unique-hyperbolic}, it is clear that, since there is a relation $x_1^{m+1} x_2^{-1}$
and there are two words $x_2^{-1} x_1^{m_1} x_2^{-1}$, $x_2^{-1} x_1^{m_2} x_2^{-1}$  in an other relation in $\mathcal{R}$,
exactly $m+m_1-1+m_2-1=2m-2$ vertices of $G_3$ have both the 2-adjacent vertices and the 3-adjacent vertices in  $G_1$. 
But, to yield the relation $x_1^{m+1}x_2^{-1}$, exactly $m$ vertices of $G_3$ have both the 2-adjacent vertices and the 3-adjacent vertices in  $G_1$.
Since $\Gamma$ is bipartite, $\{x_3^{(2n)},x_3^{(2n+2)},\dots, x_3^{(2m+2n-2)}\}$ is the only possible set of those $m$ vertices.
Therefore, $x_3^{(2n+1)},x_3^{(2n+3)},\dots, x_3^{(2m+2n-3)}$ are also joined to $G_1$ with edges of color $3$
as all the $m^{(3)}_{13}$ edges of color $3$ between $G_1$ and $G_3$ yield $m^{(3)}_{13}-1$  bi-colored $4$-cycles in $\Gamma_{\{0,1,3\}}$. Thus, we get $m+m-1=2m-1$ vertices of $G_3$ have both the 2-adjacent vertices and the 3-adjacent vertices in  $G_1$, which is a contradiction. Therefore, there is no choice for the triplet $(q_1,q_2,q_3)$ which yields the relations. Thus, there is no crystallization of a closed connected $3$-manifold which yields $(\langle S \mid R\rangle, R)$ and is minimal with respect to $(\langle S \mid R\rangle, R)$.
 
}\end{eg}

\section{Generalization of Algorithm 1}

In Section \ref{sec:algorithm}, we have computed crystallizations of $3$-manifolds from a given presentation $(\langle S \mid R\rangle, R)$ with two generators and two relations.
For such a presentation, $\Gamma_{\{0,1,2\}}$ and $\Gamma_{\{0,1,3\}}$ were unique up to an isomorphism. But, if the given presentation $(\langle S \mid R\rangle, R)$
has the number of generators and relations greater than two then  $\Gamma_{\{0,1,2\}}$ and $\Gamma_{\{0,1,3\}}$ may have many choices. But, there are some classes of presentations, for which  $\Gamma_{\{0,1,2\}}$ and 
$\Gamma_{\{0,1,3\}}$ are unique up to an isomorphism.  

In this section, we generalize  Algorithm $1$
for a presentation  $(\langle S \mid R\rangle, R)$ with three generators and a certain  class of relations. 
Let $r \in \overline{R}$ be an element of minimum weight and $ R' = R \cup \{r\}$.
Let $m^{(c)}_{ij}:= \sum_{w\in R'} w^{(c)}_{ij}$ for $1\leq i < j \leq 4$ and $2 \leq c \leq 3$, where  $w^{(c)}_{ij}$ as in
Subsection \ref{subsec:presentation}.
Let $C_R:=\{w_1, \dots, w_k\}$ be the set of all independent elements in $\overline{R}$ such that (i) the weight of $w$ is  minimum and (ii) 
for each $R\cup\{w_l\}$, we have $m^{(c)}_{ij}= \sum_{w\in R\cup\{w_l\}} w^{(c)}_{ij} \geq 1$, where $1\leq i < j \leq 4$ and $2 \leq c \leq 3$. 
Let $(\Gamma, \gamma)$ be a crystallization of a $3$-manifold such that $(\Gamma, \gamma)$ is minimal with respect to
$(\langle S \mid R\rangle, R)$ and yields $\mathcal{R} = R \cup \{w\}$, where $w \in C_R$.  Without loss of generality, let $\Gamma_{\{0,1\}}= \sqcup_{i=1}^4G_i$ such that $G_i$ represents the generator
$x_i$ for $1 \leq i \leq 3$ and $G_4$ represents $x_4$ (cf. Eq. \eqref{tildar} for construction of $\tilde{r}$).
Let $n_i$ be the total number of appearances of $x_i$ in the four
relations in $\mathcal{R}$ for $1\leq i \leq 3$ and $n_4=\lambda(\langle S \mid R \rangle, R)-(n_1+n_2+n_3)$. 
Then, the total number of vertices in $G_i$ should be $n_i$. Assume 
$G_i=C_{n_{i}}(x^{(1)}_i, \dots, x^{(n_i)}_i)$ for $1\leq i \leq 4$.
Clearly, each $n_i$ is even and $n_{1}+n_{2}+n_{3}+n_4=\#V(\Gamma)$. Without loss of generality, we can assume that $x^{(2j-1)}_ix^{(2j)}_i 
\in \gamma^{-1}(1)$ and $x^{(2j)}_ix^{(2j+1)}_i \in \gamma^{-1}(0)$
with $x^{(n_i+1)}_i=x^{(1)}_i$ for $1 \leq j \leq n_i/2$ and $1\leq i \leq 4$. 
Here and after, the additions and subtractions
at the point `$\ast$' in $x_i^{(\ast)}$ are modulo $n_i$ for $1\leq i \leq 4$.
Let the colors $2$ and $3$ be the colors `$i$' and `$j$' respectively 
as in construction of $\tilde{r}$ for $r\in \mathcal{R}$ (cf. Eq. \eqref{tildar}).
Then, the number of edges of color $c$ between $G_i$ and $G_j$ is $m_{ij}^{(c)}$
for $2\leq c \leq 3$ and $1 \leq i < j \leq 4$. Therefore, $m:=\sum_{1 \leq i < j \leq 4} m_{ij}^{(c)} = \#V(\Gamma)/2$.
Now, the maximum number of bi-colored $4$-cycles in $\Gamma_{\{0,1,c\}}$
with two edges of color $c$ is $\sum_{1 \leq i < j \leq 4} (m_{ij}^{(c)}-1) =m-6$. Again, by  Proposition \ref{prop:gagliardi79a},
$4+g_{0c}+g_{1c}= \#V(\Gamma)/2+2$, i.e., $g_{0c}+g_{1c}=m-2$. Since  $G_i$ and $G_j$ 
are connected by an edge of color $c$ for $1 \leq i < j \leq 4$ and $2\leq c \leq 3$, we have at least four distinct 
bi-colored paths $P_5$ with some edges of color $c$ in $\Gamma_{\{0,1,c\}}$ which touch
$G_i, G_j, G_l$ for all distinct $i,j,l \in \{1,2,3,4\}$. Therefore, we must have $m-6$ bi-colored $4$-cycles and four bi-colored $6$-cycles with some edges
of color $c$ in $\Gamma_{\{0,1,c\}}$. Thus, $m_{ij}^{(c)}$ edges of color $c$ between  $G_i$ and $G_j$ yield $m_{ij}^{(c)}-1$  bi-colored $4$-cycles  in $\Gamma_{\{0,1,c\}}$ for $1 \leq i < j \leq 4$ and
$2\leq c\leq3$. Therefore, two bi-colored $6$-cycles with some edges
of color $c$ in $\Gamma_{\{0,1,c\}}$ give unique choices for the remaining edges of color $c$.
Without loss, we can assume $\Gamma_{\{0,2\}}$ has a $6$-cycle $C_6(x^{(1)}_1, x^{(n_1)}_1, x_4^{(n_4)}, x_4^{(1)}, x^{(n_2)}_2,  x^{(1)}_2)$. 
Then, join $x^{(1)}_1x^{(1)}_2,\dots, x^{(m_{12}^{(2)})}_1x^{(m_{12}^{(2)})}_2$ by edges of color $2$. Without loss of generality, choose $x_3^{(p)} \in G_3$ such that
$C_6(x^{(m_{12}^{(2)})}_1, x^{(m_{12}^{(2)}+1)}_1, x_3^{(p)}, x_3^{(p+1)}, x^{(m_{12}^{(2)}+1)}_2,  x^{(m_{12}^{(2)})}_2)$ is a bi-colored cycle with three edges of color $2$. Therefore, we have a unique choice 
for  $\Gamma_{\{0,1,2\}}$ up to an isomorphism. The choices of two $6$-cycles with three edges
of color $3$ in $\Gamma_{\{0,1,3\}}$ give all possible $4$-colored graphs. If some graphs yield $(\langle S \mid R \rangle, R)$
then these satisfy all the properties of Proposition \ref{prop:gagliardi79a} and hence they are crystallizations of some $3$-manifolds.
By  similar arguments as in the proof of Theorem \ref{theorem:algorithm}, if $M$ is a closed connected prime manifold with fundamental group $(\langle S \mid R \rangle, R)$ and $(\Gamma,\gamma)$ is a crystallization, constructed from the pair
$(\langle S \mid R \rangle, R)$ then $(\Gamma,\gamma)$ is a crystallization of $M$.

\subsection{Algorithm 2}\label{subsec:algorithm2} 
We now present an algorithm for a presentation  $(\langle S \mid R\rangle, R)$ with  $\#S=\#R=3$ and $C_R \neq \emptyset$. This algorithm gives all crystallizations which yield the relation set $R\cup\{w\}$, where $w \in C_R$ and are minimum with respect to  $(\langle S \mid R\rangle, R)$.
\begin{enumerate}[(i)]
\item  Find the set $\{w_i \in \overline{R}, 1\leq i \leq k\}$ of independent words such that $\lambda(w_i)$ is minimum
and for each $R\cup\{w_l\}$, we have $m^{(c)}_{ij}= \sum_{w\in R\cup\{w_l\}} w^{(c)}_{ij} \geq 1$, where $1\leq i < j \leq 4$ and $2 \leq c \leq 3$.
Let $\mathcal{R} = R \cup \{w_1\}$ and consider a class of graphs $\mathcal{C}$  which is empty.

\item For $\mathcal{R}$, $(a)$ find $m_{ij}^{(c)}$ for $2\leq c \leq 3$ and $1 \leq i < j \leq 4$ and $(b)$ find $n_1, n_2, n_3, n_4$.

\item Consider four bi-colored cycles $G_i=C_{n_{i}}(x^{(1)}_i, \dots, x^{(n_i)}_i)$ for $1\leq i \leq 4$ such that $x^{(2j-1)}_ix^{(2j)}_i$ has color $1$ and 
$x^{(2j)}_ix^{(2j+1)}_i$ has color $0$ with the consideration $x^{(n_i+1)}_i=x^{(1)}_i$ for $1 \leq j \leq n_i/2$ and $1\leq i \leq 4$.

\item The sets $\{x^{(1)}_1x^{(1)}_2,\dots, x^{(m_{12}^{(2)})}_1x^{(m_{12}^{(2)})}_2\}$, $\{x^{(n_1)}_1x^{(n_4)}_4,\dots, x^{(n_1+1-m_{14}^{(2)})}_1x^{(n_4+1-m_{14}^{(2)})}_4\}$
and $\{x^{(n_2)}_2x^{(1)}_4,\dots, x^{(n_2+1-m_{24}^{(2)})}_4x^{(m_{24}^{(2)})}_4\}$
contain edges of color $2$. Without loss of generality, choose $x_3^{(p)} \in G_3$ such that
$C_6(x^{(m_{12}^{(2)})}_1, x^{(m_{12}^{(2)}+1)}_1, x_3^{(p)}, x_3^{(p+1)}, x^{(m_{12}^{(2)}+1)}_2,  x^{(m_{12}^{(2)})}_2)$ is a bi-colored cycle with three edges of 
color $2$. Therefore, the edges of the sets $\{x^{(m_{12}^{(2)}+1)}_2 x_3^{(p+1)}$,
$\dots$, $x^{(m_{12}^{(2)}+m_{23}^{(2)})}_2x_3^{(p+m_{23}^{(2)})}\}$, 
$\{x_3^{(p+m_{23}^{(2)}+1)} x_4^{(m_{24}^{(2)}+1)}$, $\dots$, $x_3^{(p+m_{23}^{(2)}+m_{34}^{(2)})}$ $x_4^{(m_{24}^{(2)}+m_{34}^{(2)})}\}$ and 
$\{x^{(m_{12}^{(2)}+1)}_1 x_3^{(p)}, \dots, x^{(m_{12}^{(2)}+m_{13}^{(2)})}_1 x_3^{(p+1-m_{13}^{(2)})}\}$ have also color $2$.

\item For each $1\leq q_1\leq n_1$, choose $x_2^{(q_2)} \in G_2$
such that  $\{x^{(q_1)}_1x^{(q_2)}_2,\dots, x^{(q_1-1+m_{12}^{(3)})}_1$ $x^{(q_2-1+m_{12}^{(3)})}_2\}$ $\subset \gamma^{-1}(3)$. Then, 
choose $x_3^{(q_3)} \in G_3$ and $x_4^{(q_4)} \in G_4$ such that, either
$\{x^{(q_1-1)}_1x^{(q_3)}_3, x^{(q_1+m_{12}^{(3)})}_1x^{(q_4)}_4\}$ or $\{x^{(q_1-1)}_1x^{(q_4)}_4, x^{(q_1+m_{12}^{(3)})}_1 x^{(q_3)}_3\}$ contains edges of color $3$.
There are $n_1 \times \frac{n_2}{2} \times 2 \times \frac{n_3}{2} \times \frac{n_4}{2}=\frac{n_1n_2n_3n_4}{4}$ choices for choosing these vertices and edges. Then, for each choice, join the remaining vertices by edges of color 3 as there is unique way to choose the remaining edges with the known $m_{12}^{(3)}+2$ edges of color 3. If some graphs yield $(\langle S \mid R \rangle, R)$ then put them in the class $\mathcal{C}$.

\item If $\mathcal{R} = R \cup\{w_i\}$, for some $i \in \{1, \dots, k-1\}$, choose $\mathcal{R} = R \cup\{w_{i+1}\}$ and go to step (ii). 
If $\mathcal{R} = R \cup\{w_{k}\}$ then $\mathcal{C}$ is the collection of all crystallizations which yield  $(\langle S \mid R \rangle, R)$ and 
are minimal with respect to $(\langle S \mid R \rangle, R)$. If $\mathcal{C}$ is empty then such a crystallization does not exist.
\end{enumerate}

\subsection{Constructions of crystallizations of \boldmath{ $M \langle m,n,k \rangle$}}
Recall that $M \langle m,n,k \rangle$  is the  closed connected orientable $3$-manifold with the fundamental group $\langle m,n,k \rangle$ which has a presentation 
$(\langle S \mid R_{mnk} \rangle, R_{mnk})$, where $S=\{x_1,x_2,x_3\}$ and $R_{mnk}=$
$\{x_1^{m-1}x_3^{-1}x_2^{-1}$, $x_2^{n-1}x_1^{-1}x_3^{-1},\ x_3^{k-1}x_2^{-1}x_1^{-1}\}$.
It is not difficult to prove that, $x_1^{m-2}x_3^{-1}x_2^{n-2}x_1^{-1}x_3^{k-2}x_2^{-1}$ is the only independent element in $\overline{R_{mnk}}$ of minimum weight. Observe that
\begin{eqnarray}
\lambda(x_1^{m-2}x_3^{-1}x_2^{n-2}x_1^{-1}x_3^{k-2}x_2^{-1}) = \left\{ \begin{array}{lcl}
2m & \mbox{if} & k=n=2, m \geq 3, \\
2m+2n-6 & \mbox{if} &  k=2, m,n \geq 3, \\
2m+2n+2k-12 & \mbox{if} & m,n,k \geq 3.
\end{array}\right. \nonumber
\end{eqnarray}

\noindent Therefore, 
\begin{eqnarray}
\lambda(\langle S \mid R_{mnk} \rangle, R_{mnk})= \left\{ \begin{array}{lcl}
4(m+2) & \mbox{if} & k=n=2, m \geq 3, \\
4(m+n)-2 & \mbox{if} &  k=2, m,n \geq 3, \\
4(m+n+k-3) & \mbox{if} & m,n,k \geq 3.
\end{array}\right. \nonumber
\end{eqnarray}

Since $m_{ij}^{(c)} \geq 1$ for the set $R_{mnk} \cup \{x_1^{m-2}x_3^{-1}x_2^{n-2}x_1^{-1}x_3^{k-2}x_2^{-1}\}$,  where $1\leq i < j\leq 4$ and $2\leq c\leq 3$,
we have $x_1^{m-2}x_3^{-1}x_2^{n-2}x_1^{-1}x_3^{k-2}x_2^{-1} \in C_{R_{mnk}}$. Thus, we can apply Algorithm $2$. Recall that $M \langle m,2,2 \rangle \cong S^3/Q_{4m}$.

\begin{theo} \label{theorem:q4n}
For $m \geq 3$, $S^3/Q_{4m}$ has a crystallization with $4(m+2)$ vertices which
is unique and minimal with respect to  $(\langle S \mid R_{m22} \rangle, R_{m22})$.
\end{theo}

\begin{proof}
Since $C_{R_{m22}} \neq \emptyset$ for $m \geq 3$, we can apply  Algorithm $2$. Let $\mathcal{R}= R_{m22} \cup\{x_1^{m-2}x_3^{-1}x_1^{-1}$ $x_2^{-1}\}$.
Thus, $m_{13}^{(2)}=m_{23}^{(2)}=m_{24}^{(2)}=m_{12}^{(3)}=m_{23}^{(3)}=m_{34}^{(3)}=1$, 
$m_{12}^{(2)}=m_{34}^{(2)}=m_{13}^{(3)}=m_{24}^{(3)}=2$ and $m_{14}^{(2)}=m_{14}^{(3)}=2m-3$. Observe that, 
$(n_1,n_2,n_3, n_4)=(2m,4,4,2m)$ and  $G_i=C_{n_{i}}(x^{(1)}_i, \dots, x^{(n_i)}_i)$ for $1\leq i \leq 4$ as in Figure \ref{fig:1}. 
Choose $x_3^{(p)}= x_3^{(2)}$ as in Algorithm $2$, then the $3$-colored graph with the color set $\{0,1,2\}$ is as in Figure \ref{fig:1}, which is unique up to an isomorphism. 
For the choices $(q_1,q_2,q_3,q_4)=(5,3,1,2)$ and $(x^{(q_1-1)}_1x^{(q_4)}_4, x^{(q_1+m_{12}^{(3)})}_1 x^{(q_3)}_3) = (x^{(4)}_1x^{(2)}_4, x^{(6)}_1 x^{(1)}_3)$, we get a  $4$-colored graph which yields 
$(\langle S \mid R_{m22} \rangle, R_{m22})$. Therefore, for each $m \geq 3$, we get 
a crystallization $(\Gamma, \gamma)$ of the $3$-manifold $S^3/Q_{4m}$.

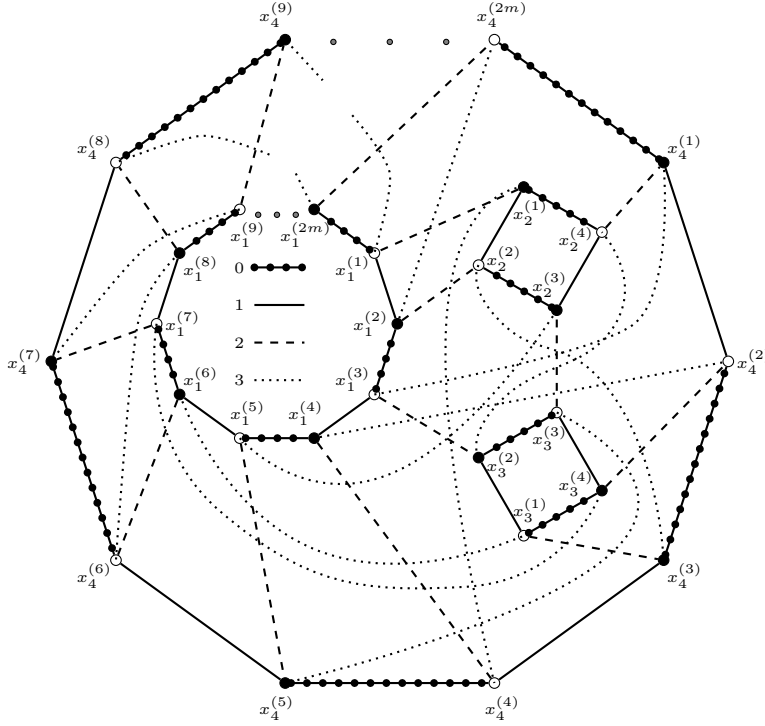
\begin{figure}[ht]
\tikzstyle{vert}=[circle, draw, fill=black!100, inner sep=0pt, minimum width=4pt] 
\tikzstyle{vertex}=[circle, draw, fill=black!00, inner sep=0pt, minimum width=4pt] 
\tikzstyle{ver}=[] 
\tikzstyle{extra}=[circle, draw, fill=black!50, inner sep=0pt, minimum width=2pt] 
\tikzstyle{edge} = [draw,thick,-] 
\centering
\begin{tikzpicture}[scale=0.5]
\begin{scope}[shift={(-3,1)}]
\foreach \x/\y in {72/x_1^{(2m)},144/x_1^{(8)},216/x_1^{(6)},288/x_1^{(4)},0/x_1^{(2)}}
{ \node[ver](\y) at (\x:2.5){\tiny{$\y$}};
    \node[vert] (\y) at (\x:3.2){};} 
    \foreach \x/\y in {108/x_1^{(9)},180/x_1^{(7)},252/x_1^{(5)},324/x_1^{(3)},36/x_1^{(1)}}{ 
    \node[ver] (\y) at (\x:2.5){\tiny{$\y$}};
    \node[vertex] (\y) at (\x:3.2){};}
    \foreach \x/\y in {x_1^{(5)}/x_1^{(6)},x_1^{(7)}/x_1^{(8)},x_1^{(1)}/x_1^{(2)},x_1^{(3)}/x_1^{(4)}}{
\path[edge] (\x) -- (\y);} 

\foreach \x/\y in {x_1^{(5)}/x_1^{(4)},x_1^{(7)}/x_1^{(6)},x_1^{(9)}/x_1^{(8)},x_1^{(3)}/x_1^{(2)},x_1^{(1)}/x_1^{(2m)}}
{\path[edge] (\x) -- (\y);} 

\foreach \x/\y in {x_1^{(5)}/x_1^{(4)},x_1^{(7)}/x_1^{(6)},x_1^{(9)}/x_1^{(8)},x_1^{(3)}/x_1^{(2)},x_1^{(1)}/x_1^{(2m)}}
{\draw [line width=3pt, line cap=round, dash pattern=on 0pt off 2\pgflinewidth]  (\x) -- (\y);} 
\end{scope}

\begin{scope}[shift={(4,3)}, rotate=60]
\node[vert] (x^{(1)}_2) at (1.2,1.2){};
\node[vertex] (x^{(2)}_2) at (-1.2,1.2){};
\node[vert] (x^{(3)}_2) at (-1.2,-1.2){};
\node[vertex] (x^{(4)}_2) at (1.2,-1.2){};
\node[ver] () at (0.7,0.7){\tiny{$x^{(1)}_2$}};
\node[ver] () at (-0.7,0.7){\tiny{$x^{(2)}_2$}};
\node[ver] () at (-0.7,-0.7){\tiny{$x^{(3)}_2$}};
\node[ver] () at (0.7,-0.7){\tiny{$x^{(4)}_2$}};
\end{scope}

\begin{scope}[shift={(4,-3)}, rotate=30]
\node[vertex] (x^{(3)}_3) at (1.2,1.2){};
\node[vert] (x^{(2)}_3) at (-1.2,1.2){};
\node[vertex] (x^{(1)}_3) at (-1.2,-1.2){};
\node[vert] (x^{(4)}_3) at (1.2,-1.2){};
\node[ver] () at (0.7,0.7){\tiny{$x^{(3)}_3$}};
\node[ver] () at (-0.7,0.7){\tiny{$x^{(2)}_3$}};
\node[ver] () at (-0.7,-0.7){\tiny{$x^{(1)}_3$}};
\node[ver] () at (0.7,-0.7){\tiny{$x^{(4)}_3$}};
\end{scope}

\begin{scope}[]
\foreach \x/\y in {72/x^{(2m)}_4,144/x^{(8)}_4,216/x^{(6)}_4,288/x^{(4)}_4,0/x^{(2)}_4}
{ \node[ver](\y) at (\x:9.7){\tiny{$\y$}};
    \node[vertex] (\y) at (\x:9){};} 
    \foreach \x/\y in {108/x^{(9)}_4,180/x^{(7)}_4,252/x^{(5)}_4,324/x^{(3)}_4,36/x^{(1)}_4}{ 
    \node[ver] (\y) at (\x:9.7){\tiny{$\y$}};
    \node[vert] (\y) at (\x:9){};}
    \foreach \x/\y in {x^{(5)}_4/x^{(6)}_4,x^{(7)}_4/x^{(8)}_4,x^{(1)}_4/x^{(2m)}_4,x^{(1)}_4/x^{(2)}_4,x^{(3)}_4/x^{(4)}_4}{
\path[edge] (\x) -- (\y);} 

\foreach \x/\y in {x^{(5)}_4/x^{(4)}_4,x^{(7)}_4/x^{(6)}_4,x^{(9)}_4/x^{(8)}_4,x^{(3)}_4/x^{(2)}_4}
{\path[edge] (\x) -- (\y);} 

\foreach \x/\y in {x^{(5)}_4/x^{(4)}_4,x^{(7)}_4/x^{(6)}_4,x^{(9)}_4/x^{(8)}_4,x^{(3)}_4/x^{(2)}_4,x^{(1)}_4/x^{(2m)}_4}
{\draw [line width=3pt, line cap=round, dash pattern=on 0pt off 2\pgflinewidth]  (\x) -- (\y);} 
\end{scope}

\begin{scope}[shift={(-5,5.5)}]
\node[ver] (3) at (1,-6){\tiny{$3$}}; 
\node[ver] (2) at (1,-5){\tiny{$2$}};
\node[ver](1) at (1,-4){\tiny{$1$}}; 
\node[ver](0) at (1,-3){\tiny{$0$}}; 
\node[ver] (8) at (3,-6){}; 
\node[ver](7) at (3,-5){}; 
\node[ver](6) at (3,-4){}; 
\node[ver] (5) at (3,-3){};
\end{scope}

\foreach \x/\y in {x^{(1)}_2/x^{(4)}_2,x^{(3)}_2/x^{(2)}_2,x^{(1)}_2/x^{(2)}_2,x^{(3)}_2/x^{(4)}_2,x^{(1)}_3/x^{(4)}_3,x^{(3)}_3/x^{(2)}_3,x^{(1)}_3/x^{(2)}_3,x^{(3)}_3/x^{(4)}_3,0/5,1/6}{
\path[edge] (\x) -- (\y);} 
\foreach \x/\y in {x^{(1)}_3/x^{(4)}_3,x^{(3)}_3/x^{(2)}_3,x^{(1)}_2/x^{(4)}_2,x^{(3)}_2/x^{(2)}_2,0/5}
{\draw [line width=3pt, line cap=round, dash pattern=on 0pt off 2\pgflinewidth]  (\x) -- (\y);} 
    
\foreach \x/\y in {x_1^{(4)}/x^{(4)}_4,x_1^{(5)}/x^{(5)}_4,x_1^{(6)}/x^{(6)}_4,x_1^{(7)}/x^{(7)}_4,x_1^{(8)}/x^{(8)}_4,x_1^{(9)}/x^{(9)}_4,x_1^{(2m)}/x^{(2m)}_4,x_1^{(2)}/x^{(2)}_2,
x_1^{(3)}/x^{(2)}_3,x^{(1)}_3/x^{(3)}_4,x^{(4)}_3/x^{(2)}_4,x^{(3)}_3/x^{(3)}_2,x^{(4)}_2/x^{(1)}_4,x^{(1)}_2/x_1^{(1)},2/7}{
\path[edge, dashed] (\x) -- (\y);} 

\foreach \x/\y in {-1.5/8.5,0/8.5,1.5/8.5,-3.5/3.9,-3/3.9,-2.5/3.9}
{\node[extra] () at (\x,\y){};}

\foreach \x/\y in {3/8,x_1^{(4)}/x^{(2)}_4,x_1^{(2)}/x^{(2m)}_4}{
\path[edge,  dotted] (\x) -- (\y);}
     
\draw[edge, dotted] plot [smooth,tension=1] coordinates{(x^{(2)}_3) (3,-1) (6,1) (x^{(4)}_2) };
\draw[edge, dotted] plot [smooth,tension=1] coordinates{(x_1^{(5)}) (0,-3) (x^{(3)}_2) };
\draw[edge, dotted] plot [smooth,tension=1] coordinates{(x^{(2)}_2) (3,1) (6,-1) (x^{(3)}_4) };
\draw[edge, dotted] plot [smooth,tension=1] coordinates{(x^{(1)}_2) (1.5,1) (x^{(4)}_4) };
\draw[edge, dotted] plot [smooth,tension=1] coordinates{(x_1^{(6)}) (-2,-5) (x^{(1)}_3) };
\draw[edge, dotted] plot [smooth,tension=1] coordinates{(x_1^{(7)}) (-4.5,-4) (2,-6) (x^{(4)}_3) };
\draw[edge, dotted] plot [smooth,tension=1] coordinates{(x^{(3)}_3) (6,-5) (x^{(5)}_4) };
\draw[edge, dotted] plot [smooth,tension=0.5] coordinates{(x_1^{(3)}) (6.5,1) (x^{(1)}_4) };
\draw[edge, dotted] plot [smooth,tension=0.5] coordinates{(x_1^{(8)}) (-6.7,1) (x^{(6)}_4) };
\draw[edge, dotted] plot [smooth,tension=0.5] coordinates{(x_1^{(9)}) (-6.5,3) (x^{(7)}_4) };
\draw[edge, dotted] plot [smooth,tension=0.5] coordinates{(-3.2,5.5) (-5,6) (x^{(8)}_4) };
\draw[edge, dotted] plot [smooth,tension=0.5] coordinates{(-2.5,5) (-2,4) (x_1^{(2m)}) };

\draw[edge, dotted] plot [smooth,tension=0.1] coordinates{(-1.8,7.5) (-1.8,7.5) (x^{(9)}_4) };
\draw[edge, dotted] plot [smooth,tension=0.5] coordinates{(-1,6.55) (0,5) (x_1^{(1)}) };

\end{tikzpicture}
\caption{Crystallization of $S^3/Q_{4m}$ for  $m \geq 3$}\label{fig:1}
\end{figure}

Now, we show that the crystallization $(\Gamma, \gamma)$ is unique.
Here we choose the pair of colors $(2,3)=(i,j)$ as in the construction of $\tilde{r}$ for $r\in \mathcal{R}$ (cf. Eq. \eqref{tildar}). 
From the construction of $\tilde{r}$ for $r\in \mathcal{R}$, it is clear that, either $x^{(1)}_2$ or $x^{(2)}_2$ is the 
starting vertex $v_1$ of the component of $\Gamma_{\{2,3\}}$,
which yields the relation $x_1^{m-1}x_3^{-1}x_2^{-1}$ (resp., $x_1^{m-2}x_3^{-1}x_1^{-1}x_2^{-1}$). If possible let $x^{(1)}_2$ be the starting vertex to yield the relation
$x_1^{m-1}x_3^{-1}x_2^{-1}$ then $x^{(2)}_2$  is the starting vertex to yield the relation $x_1^{m-2}x_3^{-1}x_1^{-1}x_2^{-1}$ and 
$x^{(1)}_1 x^{(2m-1)}_4, x^{(2m-1)}_1 x^{(2m-3)}_4,\dots, 
x^{(7)}_1 x^{(5)}_4 \in \gamma^{-1}(3)$ (as $\Gamma$ is bipartite). Since $m_{14}^{(3)}$ edges of color $3$ yield $m_{14}^{(3)}-1$ bi-colored $4$-cycle
in $\Gamma_{\{0,1,3\}}$, we have $x^{(2m)}_1 x^{(2m-2)}_4$, $x^{(2m-2)}_1 x^{(2m-4)}_4,\dots, 
x^{(8)}_1 x^{(6)}_4 \in \gamma^{-1}(3)$. Since $x^{(2)}_1 x^{(2m)}_4 \in \gamma^{-1}(3)$, the component of $\Gamma_{\{2,3\}}$ with starting vertex $x^{(2)}_2$ yields
a relation $x_1^{m-1}wx_2^{-1}$ for some $w \in F(S)$, which is not possible. Thus, $x^{(2)}_2$ is the starting vertex to yield the relation
$x_1^{m-1}x_3^{-1}x_2^{-1}$ and $x^{(2)}_1 x^{(2m)}_4, x^{(1)}_1 x^{(2m-1)}_4,\dots$, 
$x^{(8)}_1 x^{(6)}_4\in \gamma^{-1}(3)$ and  $x^{(7)}_1 x^{(5)}_4 \not \in \gamma^{-1}(3)$. Therefore, $x^{(3)}_1 x^{(1)}_4, x^{(4)}_1 x^{(2)}_4\in \gamma^{-1}(3)$ as $m_{14}^{(3)}=2m-3$. 
To yield the relation $x_1^{m-1}x_3^{-1}x_2^{-1}$, we have $x^{(6)}_1 x^{(1)}_3, x^{(2)}_2 x^{(3)}_4 \in \gamma^{-1}(3)$.
Since $x^{(3)}_1 x^{(1)}_4 \in \gamma^{-1}(3)$, we have $x^{(4)}_2x^{(2)}_3 \in \gamma^{-1}(3)$ and hence $C_4(x_1^{(3)},x^{(2)}_3,x^{(4)}_2,x^{(1)}_4)$ is the component of $\Gamma_{\{2,3\}}$ which yields the relation
$x_3x_2^{-1}x_1^{-1}$ with starting vertex $v_1=x^{(3)}_1$. Now, there is unique way to choose the remaining edges of color $3$ as in Figure \ref{fig:1}. 
Since $x_1^{m-2}x_3^{-1}x_1^{-1}x_2^{-1}$ is the only independent element in $\overline{R_{m22}}$ of minimum weight, the theorem follows.
\end{proof}
\begin{remark}

{\rm
For $m=2$, there is no crystallization of $S^3/Q_{4m}$ with $4(m+2)=16$ vertices (cf. \cite{bd14}).}
\end{remark}

\begin{theo} \label{theorem:mn2}
For $m,n \geq 3$, $M \langle m,n,2 \rangle$ has a crystallization with $4(m+n)-2$ vertices which
is unique and minimal with respect to  $(\langle S \mid R_{mn2} \rangle, R_{mn2})$.
\end{theo}

\begin{proof}
Since $C_{R_{mn2}} \neq \emptyset$ for $m,n \geq 3$, we can apply  Algorithm $2$. Let $\mathcal{R}= R_{mn2} \cup\{x_1^{m-2}x_3^{-1}x_2^{n-2}x_1^{-1}x_2^{-1}\}$.
Thus, $m_{13}^{(2)}=m_{23}^{(2)}=m_{23}^{(3)}=m_{34}^{(3)}=1$, $m_{24}^{(2)}=2n-4, m_{24}^{(3)}=2n-3$,
$m_{12}^{(2)}=m_{23}^{(2)}=m_{12}^{(3)}=m_{13}^{(3)}=2$, $m_{14}^{(2)}=2m-3$ and $m_{14}^{(3)}=2m-4$. Observe that, 
$(n_1,n_2,n_3, n_4)=(2m,2n,4,2(m+n-3))$ and  $G_i=C_{n_{i}}(x^{(1)}_i, \dots, x^{(n_i)}_i)$ for $1\leq i \leq 4$ as in Figure \ref{fig:mn2}. 
Choose $x_3^{(p)}= x_3^{(2)}$ as in Algorithm $2$, then the $3$-colored graph with the color set $\{0,1,2\}$ is as in Figure \ref{fig:mn2}, which is unique up to an isomorphism. 
For the choices $(q_1,q_2,q_3,q_4)=(4,2n-2,1,1)$ and $(x^{(q_1-1)}_1x^{(q_4)}_4, x^{(q_1+m_{12}^{(3)})}_1 x^{(q_3)}_3)= (x^{(3)}_1x^{(1)}_4, x^{(6)}_1 x^{(1)}_3)$, we get a  $4$-colored graph which yields 
$(\langle S \mid R_{mn2} \rangle, R_{mn2})$. Therefore, for each $m,n \geq 3$, we get 
a crystallization $(\Gamma, \gamma)$ of the $3$-manifold $M \langle m,n,2 \rangle$.

\begin{figure}[ht]
\tikzstyle{vert}=[circle, draw, fill=black!100, inner sep=0pt, minimum width=4pt] 
\tikzstyle{vertex}=[circle, draw, fill=black!00, inner sep=0pt, minimum width=4pt] 
\tikzstyle{ver}=[] 
\tikzstyle{extra}=[circle, draw, fill=black!50, inner sep=0pt, minimum width=2pt] 
\tikzstyle{edge} = [draw,thick,-] 
\centering
\begin{tikzpicture}[scale=0.5]
\begin{scope}[shift={(-4.5,2)}]
\foreach \x/\y in {72/x_1^{(2m)},144/x_1^{(8)},216/x_1^{(6)},288/x_1^{(4)},0/x_1^{(2)}}
{ \node[ver](\y) at (\x:2.5){\tiny{$\y$}};
    \node[vert] (\y) at (\x:3.2){};} 
    \foreach \x/\y in {108/x_1^{(9)},180/x_1^{(7)},252/x_1^{(5)},324/x_1^{(3)},36/x_1^{(1)}}{ 
    \node[ver] (\y) at (\x:2.5){\tiny{$\y$}};
    \node[vertex] (\y) at (\x:3.2){};}
    \foreach \x/\y in {x_1^{(5)}/x_1^{(6)},x_1^{(7)}/x_1^{(8)},x_1^{(1)}/x_1^{(2)},x_1^{(3)}/x_1^{(4)}}{
\path[edge] (\x) -- (\y);} 

\foreach \x/\y in {x_1^{(5)}/x_1^{(4)},x_1^{(7)}/x_1^{(6)},x_1^{(9)}/x_1^{(8)},x_1^{(3)}/x_1^{(2)},x_1^{(1)}/x_1^{(2m)}}
{\path[edge] (\x) -- (\y);} 

\foreach \x/\y in {x_1^{(5)}/x_1^{(4)},x_1^{(7)}/x_1^{(6)},x_1^{(9)}/x_1^{(8)},x_1^{(3)}/x_1^{(2)},x_1^{(1)}/x_1^{(2m)}}
{\draw [line width=3pt, line cap=round, dash pattern=on 0pt off 2\pgflinewidth]  (\x) -- (\y);} 
\foreach \x/\y in {-0.5/3,0/3,0.5/3}
{\node[extra] () at (\x,\y){};}

\end{scope}
\begin{scope}[shift={(4.5,2)}, rotate=-144]
\foreach \x/\y in {72/x_2^{(5)},144/x_2^{(2n-3)},216/x_2^{(2n-1)},288/x_2^{(1)},0/x_2^{(3)}}
{ \node[ver](\y) at (\x:2.5){\tiny{$\y~~$}};
    \node[vert] (\y) at (\x:3.2){};} 
    \foreach \x/\y in {108/x_2^{(2n-4)},180/x_2^{(2n-2)},252/x_2^{(2n)},324/x_2^{(2)},36/x_2^{(4)}}{ 
    \node[ver] (\y) at (\x:2.5){\tiny{$\y~~$}};
    \node[vertex] (\y) at (\x:3.2){};}
    \foreach \x/\y in {x_2^{(2n)}/x_2^{(2n-1)},x_2^{(2n-2)}/x_2^{(2n-3)},x_2^{(4)}/x_2^{(3)},x_2^{(2)}/x_2^{(1)}}{
\path[edge] (\x) -- (\y);} 

\foreach \x/\y in {x_2^{(2n)}/x_2^{(1)},x_2^{(2n-2)}/x_2^{(2n-1)},x_2^{(2n-4)}/x_2^{(2n-3)},x_2^{(2)}/x_2^{(3)},x_2^{(4)}/x_2^{(5)}}
{\path[edge] (\x) -- (\y);} 

\foreach \x/\y in {x_2^{(2n)}/x_2^{(1)},x_2^{(2n-2)}/x_2^{(2n-1)},x_2^{(2n-4)}/x_2^{(2n-3)},x_2^{(2)}/x_2^{(3)},x_2^{(4)}/x_2^{(5)}}
{\draw [line width=3pt, line cap=round, dash pattern=on 0pt off 2\pgflinewidth]  (\x) -- (\y);} 
\foreach \x/\y in {-0.5/3,0/3,0.5/3}
{\node[extra] () at (\x,\y){};} 

\end{scope}

\begin{scope}[shift={(0,-4)}, rotate=135]
\foreach \x/\y in {180/x_3^{(4)},0/x_3^{(2)}}
{ \node[ver](\y) at (\x:1.3){\tiny{$\y$}};
    \node[vert] (\y) at (\x:2){};} 
    \foreach \x/\y in {270/x_3^{(3)},90/x_3^{(1)}}{ 
    \node[ver] (\y) at (\x:1.3){\tiny{$\y$}};
    \node[vertex] (\y) at (\x:2){};}
    \foreach \x/\y in {x_3^{(1)}/x_3^{(2)},x_3^{(3)}/x_3^{(4)}}{
\path[edge] (\x) -- (\y);} 

\foreach \x/\y in {x_3^{(1)}/x_3^{(4)},x_3^{(3)}/x_3^{(2)}}
{\path[edge] (\x) -- (\y);} 

\foreach \x/\y in {x_3^{(1)}/x_3^{(4)},x_3^{(3)}/x_3^{(2)}}
{\draw [line width=3pt, line cap=round, dash pattern=on 0pt off 2\pgflinewidth]  (\x) -- (\y);} 

\end{scope}

\begin{scope}[rotate=-50]
\foreach \x/\y in {0/x_4^{(2n-4)},50.42/x_4^{(4)},100.48/x_4^{(2)},154.28/x_4^{(2m+2n-6)},
204.70/x_4^{(2n+2)},255.12/x_4^{(2n)},305.54/x_4^{(2n-2)}}
{ \node[ver](\y) at (\x:11.3){\tiny{$\y$}};
    \node[vertex] (\y) at (\x:10.5){};}

\foreach \x/\y in {25.21/x_4^{(5)},75.63/x_4^{(3)},126.05/x_4^{(1)},176.47/x_4^{(2n+3)},
226.89/x_4^{(2n+1)},277.31/x_4^{(2n-1)},327.73/x_4^{(2n-3)}}
{ \node[ver](\y) at (\x:11.3){\tiny{$\y$}};
    \node[vert] (\y) at (\x:10.5){};} 
    
\foreach \x in {160,166,172,5.5,11.5,17.5}
{\node[extra] () at (\x:10.5){};} 
\end{scope} 

\foreach \x/\y in {x_4^{(2n+2)}/x_4^{(2n+1)},x_4^{(2n)}/x_4^{(2n-1)},x_4^{(1)}/x_4^{(2)},x_4^{(3)}/x_4^{(4)},
x_4^{(2n-3)}/x_4^{(2n-2)},x_4^{(2n-2)}/x_4^{(2n-1)}}{\path[edge] (\x) -- (\y);}  

\foreach \x/\y in {x_4^{(2n+3)}/x_4^{(2n+2)},x_4^{(2n+1)}/x_4^{(2n)},x_4^{(2m+2n-6)}/x_4^{(1)},x_4^{(2)}/x_4^{(3)},
x_4^{(4)}/x_4^{(5)},x_4^{(2n-4)}/x_4^{(2n-3)}}{\path[edge] (\x) -- (\y);}
\foreach \x/\y in {x_4^{(2n+3)}/x_4^{(2n+2)},x_4^{(2n+1)}/x_4^{(2n)},x_4^{(2m+2n-6)}/x_4^{(1)},x_4^{(2)}/x_4^{(3)},
x_4^{(4)}/x_4^{(5)},x_4^{(2n-4)}/x_4^{(2n-3)}}{\draw [line width=3pt, line cap=round, dash pattern=on 0pt off 2\pgflinewidth]  (\x) -- (\y);} 

\foreach \x/\y in {x_1^{(1)}/x_2^{(1)},x_1^{(2)}/x_2^{(2)},x_1^{(3)}/x_3^{(2)},x_1^{(4)}/x_4^{(2n-2)},x_2^{(3)}/x_3^{(3)},x_3^{(1)}/x_4^{(2n-3)},
x_1^{(2m)}/x_4^{(2m+2n-6)},x_1^{(9)}/x_4^{(2n+3)},x_1^{(8)}/x_4^{(2n+2)},x_1^{(7)}/x_4^{(2n+1)},
x_1^{(6)}/x_4^{(2n)},x_1^{(5)}/x_4^{(2n-1)},x_2^{(2n)}/x_4^{(1)},x_2^{(2n-1)}/x_4^{(2)},x_2^{(2n-2)}/x_4^{(3)},x_2^{(2n-3)}/x_4^{(4)},
x_2^{(2n-4)}/x_4^{(5)},x_2^{(5)}/x_4^{(2n-4)},x_3^{(4)}/x_2^{(4)}}{\path[edge, dashed] (\x) -- (\y);}   

\draw[edge, dotted] plot [smooth,tension=1] coordinates{(x_1^{(3)}) (0,3) (x_4^{(1)})}; 
\draw[edge, dotted] plot [smooth,tension=1] coordinates{(x_1^{(2)}) (-1,5) (x_4^{(2m+2n-6)})};   
\draw[edge, dotted] plot [smooth,tension=1] coordinates{(x_1^{(9)}) (-8,4)  (x_4^{(2n+1)})};   
\draw[edge, dotted] plot [smooth,tension=1] coordinates{(x_1^{(8)}) (-8.5,2) (x_4^{(2n)})}; 
\draw[edge, dotted] plot [smooth,tension=1] coordinates{(x_1^{(6)}) (-5,-4) (x_3^{(1)})}; 
\draw[edge, dotted] plot [smooth,tension=1] coordinates{(x_1^{(7)}) (-5,-5) (x_3^{(4)})}; 
\draw[edge, dotted] plot [smooth,tension=1] coordinates{(x_3^{(2)}) (1,3.5) (x_2^{(2n)})};
\draw[edge, dotted] plot [smooth,tension=1] coordinates{(x_2^{(1)}) (4,2) (3,-6)(x_4^{(2n-2)})};
\draw[edge, dotted] plot [smooth,tension=1] coordinates{(x_3^{(3)}) (2,-6) (x_4^{(2n-1)})};
\draw[edge, dotted] plot [smooth,tension=1] coordinates{(x_2^{(2)}) (-3,-4) (x_4^{(2n-3)})};
\draw[edge, dotted] plot [smooth,tension=1] coordinates{(x_2^{(3)}) (3,-3) (x_4^{(2n-4)})};
\draw[edge, dotted] plot [smooth,tension=1] coordinates{(x_2^{(4)}) (6.5,-3.2) };
\draw[edge, dotted] plot [smooth,tension=1] coordinates{(x_4^{(5)}) (8.5,-4)};
\draw[edge, dotted] plot [smooth,tension=1] coordinates{(x_2^{(5)}) (6,-1.3)};
\draw[edge, dotted] plot [smooth,tension=1] coordinates{(7.5,-1.5) (9,-1) (x_4^{(4)})};
\draw[edge, dotted] plot [smooth,tension=1] coordinates{(x_2^{(2n-4)}) (9,3) (x_4^{(3)})};
\draw[edge, dotted] plot [smooth,tension=1] coordinates{(x_2^{(2n-3)}) (8,5) (x_4^{(2)})};
\draw[edge, dotted] plot [smooth,tension=0.5] coordinates{(x_1^{(5)}) (1,5.5) (x_2^{(2n-1)})};
\draw[edge, dotted] plot [smooth,tension=0.5] coordinates{(x_1^{(4)})(1,1) (5,1)(x_2^{(2n-2)})};
\draw[edge, dotted] plot [smooth,tension=1] coordinates{(x_4^{(2n+2)}) (-7,6) (-5,5.8)};
\draw[edge, dotted] plot [smooth,tension=0.5] coordinates{(x_4^{(2n+3)})(-5,8)};
\draw[edge, dotted] plot [smooth,tension=0.5] coordinates{(x_1^{(2m)})(-4,5.5)};
\draw[edge, dotted] plot [smooth,tension=1] coordinates{(x_1^{(1)}) (-2,5.5)(-3.5,7.5)};
\end{tikzpicture}
\caption{Crystallization of $M\langle m,n,2 \rangle$ for $m,n \geq 3$}\label{fig:mn2}
\end{figure}
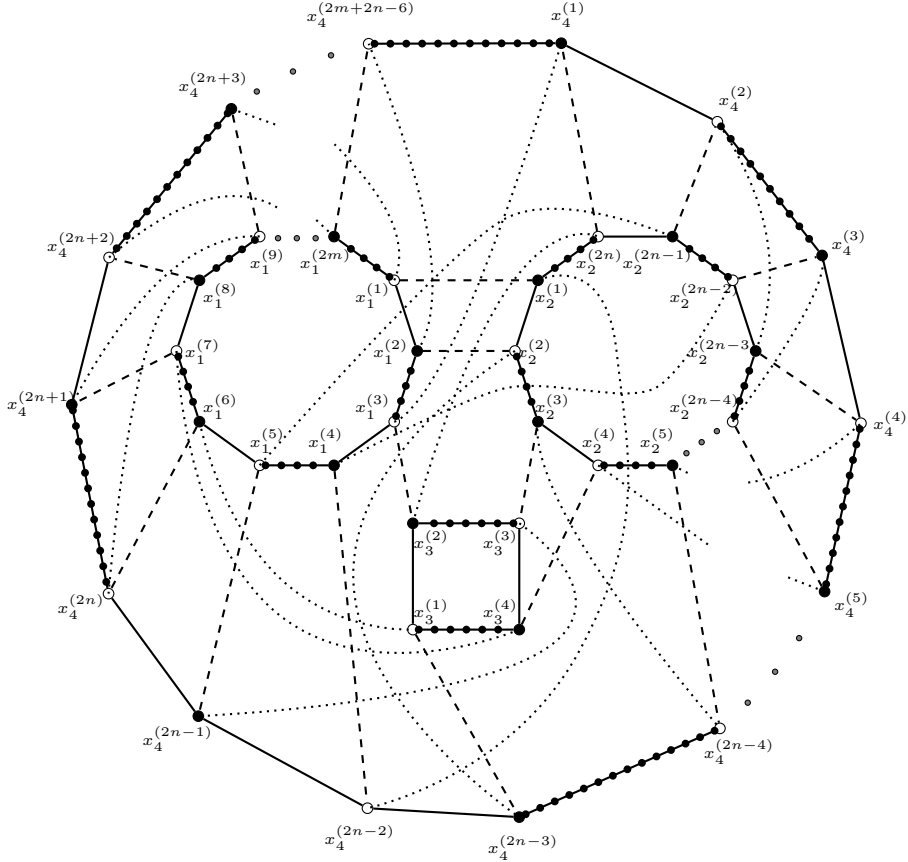

Now, we show that the crystallization $(\Gamma, \gamma)$ is unique.
Here we choose the pair of colors $(2,3)=(i,j)$ as in construction of $\tilde{r}$ for $r\in \mathcal{R}$ (cf. Eq. \eqref{tildar}).
By  similar arguments as in the proof of Theorem \ref{theorem:q4n}, $x^{(2)}_2$ is the starting vertex to yield the relation
$x_1^{m-1}x_3^{-1}x_2^{-1}$ and hence $x^{(2)}_1 x^{(2m+2n-6)}_4, x^{(1)}_1 x^{(2m+2n-5)}_4,\dots, 
x^{(8)}_1 x^{(2n)}_4 \in \gamma^{-1}(3)$ and  $x^{(7)}_1 x^{(2n-1)}_4 \not \in \gamma^{-1}(3)$. Therefore, $x^{(3)}_1 x^{(1)}_4\in \gamma^{-1}(3)$ as $m_{14}^{(3)}=2m-4$. 
Since $m_{34}^{(2)}=1$, to yield the relation $x_1^{m-1}x_3^{-1}x_2^{-1}$, we have $x^{(6)}_1 x^{(1)}_3, x^{(2)}_2 x^{(2n-3)}_4 \in \gamma^{-1}(3)$.
Similarly, $x^{(3)}_3$ is the starting vertex to yield the relation
$x_2^{n-1}x_1^{-1}x_3^{-1}$ and hence $x^{(3)}_2 x^{(2n-4)}_4, x^{(4)}_2 x^{(2n-5)}_4,\dots, 
x^{(2n-3)}_2 x^{(2)}_4$ $\in \gamma^{-1}(3)$ and  $x^{(2n-2)}_2 x^{(1)}_4 \not \in \gamma^{-1}(3)$. 
Therefore, $x^{(2)}_2 x^{(2n-3)}_4, x^{(1)}_2 x^{(2n-2)}_4 \in \gamma^{-1}(3)$ as $m_{24}^{(3)}$ $=2n-3$ and hence
$x^{(7)}_1 x^{(4)}_3, x^{(4)}_1 x^{(2n-2)}_2, x^{(1)}_2 x^{(2n-2)}_4 \in \gamma^{-1}(3)$ to yield the relation $x_1^{m-2}x_3^{-1}x_2^{n-2}x_1^{-1}$ $x_2^{-1}$.
Now, there is unique way to choose the remaining edges of color $3$ as in Figure \ref{fig:mn2}. 
Since $x_1^{m-2}x_3^{-1}x_2^{n-2}x_1^{-1}x_2^{-1}$ is the only independent element in $\overline{R_{mn2}}$ of minimum weight, the theorem follows.
\end{proof}

\begin{theo} \label{theorem:mnk}
For $m,n,k \geq 3$, $M \langle m,n,k \rangle$ has a crystallization with $4(m+n+k-3)$ vertices which
is unique and minimal with respect to  $(\langle S \mid R_{mnk} \rangle, R_{mnk})$.
\end{theo}

\begin{figure}[ht]
\tikzstyle{vert}=[circle, draw, fill=black!100, inner sep=0pt, minimum width=4pt] 
\tikzstyle{vertex}=[circle, draw, fill=black!00, inner sep=0pt, minimum width=4pt] 
\tikzstyle{ver}=[] 
\tikzstyle{extra}=[circle, draw, fill=black!50, inner sep=0pt, minimum width=2pt] 
\tikzstyle{edge} = [draw,thick,-] 
\centering
\begin{tikzpicture}[scale=0.55]
\begin{scope}[shift={(-4.5,2)}]
\foreach \x/\y in {72/x_1^{(2m)},144/x_1^{(8)},216/x_1^{(6)},288/x_1^{(4)},0/x_1^{(2)}}
{ \node[ver](\y) at (\x:2.5){\tiny{$\y$}};
    \node[vert] (\y) at (\x:3.2){};} 
    \foreach \x/\y in {108/x_1^{(9)},180/x_1^{(7)},252/x_1^{(5)},324/x_1^{(3)},36/x_1^{(1)}}{ 
    \node[ver] (\y) at (\x:2.5){\tiny{$\y$}};
    \node[vertex] (\y) at (\x:3.2){};}
    \foreach \x/\y in {x_1^{(5)}/x_1^{(6)},x_1^{(7)}/x_1^{(8)},x_1^{(1)}/x_1^{(2)},x_1^{(3)}/x_1^{(4)}}{
\path[edge] (\x) -- (\y);} 

\foreach \x/\y in {x_1^{(5)}/x_1^{(4)},x_1^{(7)}/x_1^{(6)},x_1^{(9)}/x_1^{(8)},x_1^{(3)}/x_1^{(2)},x_1^{(1)}/x_1^{(2m)}}
{\path[edge] (\x) -- (\y);} 

\foreach \x/\y in {x_1^{(5)}/x_1^{(4)},x_1^{(7)}/x_1^{(6)},x_1^{(9)}/x_1^{(8)},x_1^{(3)}/x_1^{(2)},x_1^{(1)}/x_1^{(2m)}}
{\draw [line width=3pt, line cap=round, dash pattern=on 0pt off 2\pgflinewidth]  (\x) -- (\y);} 
\foreach \x/\y in {-0.5/3,0/3,0.5/3}
{\node[extra] () at (\x,\y){};}

\end{scope}
\begin{scope}[shift={(4.5,2)}, rotate=-144]
\foreach \x/\y in {72/x_2^{(5)},144/x_2^{(2n-3)},216/x_2^{(2n-1)},288/x_2^{(1)},0/x_2^{(3)}}
{ \node[ver](\y) at (\x:2.5){\tiny{$\y~~$}};
    \node[vert] (\y) at (\x:3.2){};} 
    \foreach \x/\y in {108/x_2^{(2n-4)},180/x_2^{(2n-2)},252/x_2^{(2n)},324/x_2^{(2)},36/x_2^{(4)}}{ 
    \node[ver] (\y) at (\x:2.5){\tiny{$\y~~$}};
    \node[vertex] (\y) at (\x:3.2){};}
    \foreach \x/\y in {x_2^{(2n)}/x_2^{(2n-1)},x_2^{(2n-2)}/x_2^{(2n-3)},x_2^{(4)}/x_2^{(3)},x_2^{(2)}/x_2^{(1)}}{
\path[edge] (\x) -- (\y);} 

\foreach \x/\y in {x_2^{(2n)}/x_2^{(1)},x_2^{(2n-2)}/x_2^{(2n-1)},x_2^{(2n-4)}/x_2^{(2n-3)},x_2^{(2)}/x_2^{(3)},x_2^{(4)}/x_2^{(5)}}
{\path[edge] (\x) -- (\y);} 

\foreach \x/\y in {x_2^{(2n)}/x_2^{(1)},x_2^{(2n-2)}/x_2^{(2n-1)},x_2^{(2n-4)}/x_2^{(2n-3)},x_2^{(2)}/x_2^{(3)},x_2^{(4)}/x_2^{(5)}}
{\draw [line width=3pt, line cap=round, dash pattern=on 0pt off 2\pgflinewidth]  (\x) -- (\y);} 
\foreach \x/\y in {-0.5/3,0/3,0.5/3}
{\node[extra] () at (\x,\y){};} 

\end{scope}

\begin{scope}[shift={(0,-4.5)}, rotate=108]
\foreach \x/\y in {72/x_3^{(2k)},144/x_3^{(8)},216/x_3^{(6)},288/x_3^{(4)},0/x_3^{(2)}}
{ \node[ver](\y) at (\x:2.5){\tiny{$\y$}};
    \node[vert] (\y) at (\x:3.2){};} 
    \foreach \x/\y in {108/x_3^{(9)},180/x_3^{(7)},252/x_3^{(5)},324/x_3^{(3)},36/x_3^{(1)}}{ 
    \node[ver] (\y) at (\x:2.5){\tiny{$\y$}};
    \node[vertex] (\y) at (\x:3.2){};}
    \foreach \x/\y in {x_3^{(5)}/x_3^{(6)},x_3^{(7)}/x_3^{(8)},x_3^{(1)}/x_3^{(2)},x_3^{(3)}/x_3^{(4)}}{
\path[edge] (\x) -- (\y);} 

\foreach \x/\y in {x_3^{(5)}/x_3^{(4)},x_3^{(7)}/x_3^{(6)},x_3^{(9)}/x_3^{(8)},x_3^{(3)}/x_3^{(2)},x_3^{(1)}/x_3^{(2k)}}
{\path[edge] (\x) -- (\y);} 

\foreach \x/\y in {x_3^{(5)}/x_3^{(4)},x_3^{(7)}/x_3^{(6)},x_3^{(9)}/x_3^{(8)},x_3^{(3)}/x_3^{(2)},x_3^{(1)}/x_3^{(2k)}}
{\draw [line width=3pt, line cap=round, dash pattern=on 0pt off 2\pgflinewidth]  (\x) -- (\y);} 
\foreach \x/\y in {-0.5/3,0/3,0.5/3}
{\node[extra] () at (\x,\y){};} 

\end{scope}

\begin{scope}[rotate=-20]
\foreach \x/\y in {0/x_4^{(2n-4)},40/x_4^{(4)},80/x_4^{(2)},120/x_4^{(2m+2n+2k-12)},
160/x_4^{(2n+2k-4)},200/x_4^{(2n+2k-6)},240/x_4^{(2n+2k-8)},280/x_4^{(2n)},320/x_4^{(2n-2)}}
{ \node[ver](\y) at (\x:11.8){\tiny{$\y$}};
    \node[vertex] (\y) at (\x:10.5){};}

\foreach \x/\y in {20/x_4^{(5)},60/x_4^{(3)},100/x_4^{(1)},140/x_4^{(2n+2k-3)},
180/x_4^{(2n+2k-5)},220/x_4^{(2n+2k-7)},260/x_4^{(2n+1)},300/x_4^{(2n-1)},340/x_4^{(2n-3)}}
{ \node[ver](\y) at (\x:11.8){\tiny{$\y$}};
    \node[vert] (\y) at (\x:10.5){};} 
    
\foreach \x in {125,130,135,245,250,255,5,10,15}
{\node[extra] () at (\x:10.5){};} 
\end{scope} 

\foreach \x/\y in {x_4^{(2n+2k-4)}/x_4^{(2n+2k-5)},x_4^{(2n+2k-6)}/x_4^{(2n+2k-7)},x_4^{(1)}/x_4^{(2)},x_4^{(3)}/x_4^{(4)},
x_4^{(2n-3)}/x_4^{(2n-2)},x_4^{(2n-1)}/x_4^{(2n)}}{\path[edge] (\x) -- (\y);}  

\foreach \x/\y in {x_4^{(2n+2k-3)}/x_4^{(2n+2k-4)},x_4^{(2n+2k-5)}/x_4^{(2n+2k-6)},x_4^{(2m+2n+2k-12)}/x_4^{(1)},x_4^{(2)}/x_4^{(3)},
x_4^{(4)}/x_4^{(5)},x_4^{(2n-4)}/x_4^{(2n-3)},x_4^{(2n-2)}/x_4^{(2n-1)},x_4^{(2n)}/x_4^{(2n+1)},x_4^{(2n+2k-8)}/x_4^{(2n+2k-7)}}{\path[edge] (\x) -- (\y);}
\foreach \x/\y in {x_4^{(2n+2k-3)}/x_4^{(2n+2k-4)},x_4^{(2n+2k-5)}/x_4^{(2n+2k-6)},x_4^{(2m+2n+2k-12)}/x_4^{(1)},x_4^{(2)}/x_4^{(3)},
x_4^{(4)}/x_4^{(5)},x_4^{(2n-4)}/x_4^{(2n-3)},x_4^{(2n-2)}/x_4^{(2n-1)},x_4^{(2n)}/x_4^{(2n+1)},x_4^{(2n+2k-8)}/x_4^{(2n+2k-7)}}{\draw [line width=3pt, line cap=round, dash pattern=on 0pt off 2\pgflinewidth]  (\x) -- (\y);} 

\foreach \x/\y in {x_1^{(1)}/x_2^{(1)},x_1^{(2)}/x_2^{(2)},x_1^{(3)}/x_3^{(2)},x_1^{(4)}/x_3^{(1)},x_2^{(3)}/x_3^{(3)},x_2^{(4)}/x_3^{(4)},
x_1^{(2m)}/x_4^{(2m+2n+2k-12)},x_1^{(9)}/x_4^{(2n+2k-3)},x_1^{(8)}/x_4^{(2n+2k-4)},x_1^{(7)}/x_4^{(2n+2k-5)},
x_1^{(6)}/x_4^{(2n+2k-6)},x_1^{(5)}/x_4^{(2n+2k-7)},x_2^{(2n)}/x_4^{(1)},x_2^{(2n-1)}/x_4^{(2)},x_2^{(2n-2)}/x_4^{(3)},x_2^{(2n-3)}/x_4^{(4)},
x_2^{(2n-4)}/x_4^{(5)},x_2^{(5)}/x_4^{(2n-4)},x_3^{(5)}/x_4^{(2n-3)},x_3^{(6)}/x_4^{(2n-2)},x_3^{(7)}/x_4^{(2n-1)},x_3^{(8)}/x_4^{(2n)},
x_3^{(9)}/x_4^{(2n+1)},x_3^{(2k)}/x_4^{(2n+2k-8)}}{\path[edge, dashed] (\x) -- (\y);}   

\draw[edge, dotted] plot [smooth,tension=1] coordinates{(x_1^{(3)}) (0,3) (x_4^{(1)})}; 
\draw[edge, dotted] plot [smooth,tension=1] coordinates{(x_1^{(2)}) (-1,5) (x_4^{(2m+2n+2k-12)})};   
\draw[edge, dotted] plot [smooth,tension=1] coordinates{(x_1^{(9)})  (x_4^{(2n+2k-5)})};   
\draw[edge, dotted] plot [smooth,tension=1] coordinates{(x_1^{(8)})  (x_4^{(2n+2k-6)})}; 
\draw[edge, dotted] plot [smooth,tension=0.5] coordinates{(x_3^{(3)}) (-1.5,-1) (x_4^{(2n+2k-7)})}; 
\draw[edge, dotted] plot [smooth,tension=1] coordinates{(x_3^{(2)}) (-4,-2.5) (x_4^{(2n+2k-8)})};
\draw[edge, dotted] plot [smooth,tension=1] coordinates{(x_3^{(9)}) (0,-9.5)(x_4^{(2n-1)})};
\draw[edge, dotted] plot [smooth,tension=1] coordinates{(x_3^{(8)}) (1,-8.5) (x_4^{(2n-2)})};
\draw[edge, dotted] plot [smooth,tension=1] coordinates{(x_2^{(2)}) (1,0) (x_4^{(2n-3)})};
\draw[edge, dotted] plot [smooth,tension=1] coordinates{(x_2^{(3)}) (5,-3) (x_4^{(2n-4)})};
\draw[edge, dotted] plot [smooth,tension=1] coordinates{(x_2^{(4)}) (5,-2) (7.5,-1.5)};
\draw[edge, dotted] plot [smooth,tension=1] coordinates{(x_4^{(5)}) (9,-1)};
\draw[edge, dotted] plot [smooth,tension=1] coordinates{(x_2^{(5)}) (7,-0.6)};
\draw[edge, dotted] plot [smooth,tension=1] coordinates{(8,0) (9.5,2) (x_4^{(4)})};
\draw[edge, dotted] plot [smooth,tension=1] coordinates{(x_2^{(2n-4)}) (9,3) (x_4^{(3)})};
\draw[edge, dotted] plot [smooth,tension=1] coordinates{(x_2^{(2n-3)}) (8,5) (x_4^{(2)})};
\draw[edge, dotted] plot [smooth,tension=0.5] coordinates{(x_1^{(7)})(-5.5,3) (-1,1) (x_3^{(4)})};
\draw[edge, dotted] plot [smooth,tension=0.5] coordinates{(x_1^{(6)})(-4,-3.3)(2,-5) (x_3^{(5)})};
\draw[edge, dotted] plot [smooth,tension=0.5] coordinates{(x_3^{(6)})(3.8,-4) (5,2.5) (x_2^{(2n)})};
\draw[edge, dotted] plot [smooth,tension=0.5] coordinates{(x_3^{(7)})(0,-6.5) (0,2) (x_2^{(1)})};
\draw[edge, dotted] plot [smooth,tension=0.5] coordinates{(x_1^{(5)}) (1,5.5) (x_2^{(2n-1)})};
\draw[edge, dotted] plot [smooth,tension=0.5] coordinates{(x_1^{(4)})(1,1) (5,1)(x_2^{(2n-2)})};
\draw[edge, dotted] plot [smooth,tension=0.5] coordinates{(x_4^{(2n+2k-4)})(-5,5.8)};
\draw[edge, dotted] plot [smooth,tension=0.5] coordinates{(x_4^{(2n+2k-3)})(-4.5,8.5)};
\draw[edge, dotted] plot [smooth,tension=0.5] coordinates{(x_1^{(2m)})(-4,5.5)};
\draw[edge, dotted] plot [smooth,tension=1] coordinates{(x_1^{(1)}) (-2,5.5)(-3.5,7.5)};
\draw[edge, dotted] plot [smooth,tension=1] coordinates{(x_4^{(2n)})(-2.8,-9) (-3.4,-6.5)};
\draw[edge, dotted] plot [smooth,tension=0.5] coordinates{(x_3^{(2k)}) (-3.4,-5.5)};
\draw[edge, dotted] plot [smooth,tension=0.5] coordinates{(x_4^{(2n+1)})(-5,-8)};
\draw[edge, dotted] plot [smooth,tension=0.5] coordinates{(x_3^{(1)})(-4,-4) (-4.6,-6)};
\end{tikzpicture}
\caption{Crystallization of $M\langle m,n,k \rangle$ for $m,n,k \geq 3$}\label{fig:mnk}
\end{figure}
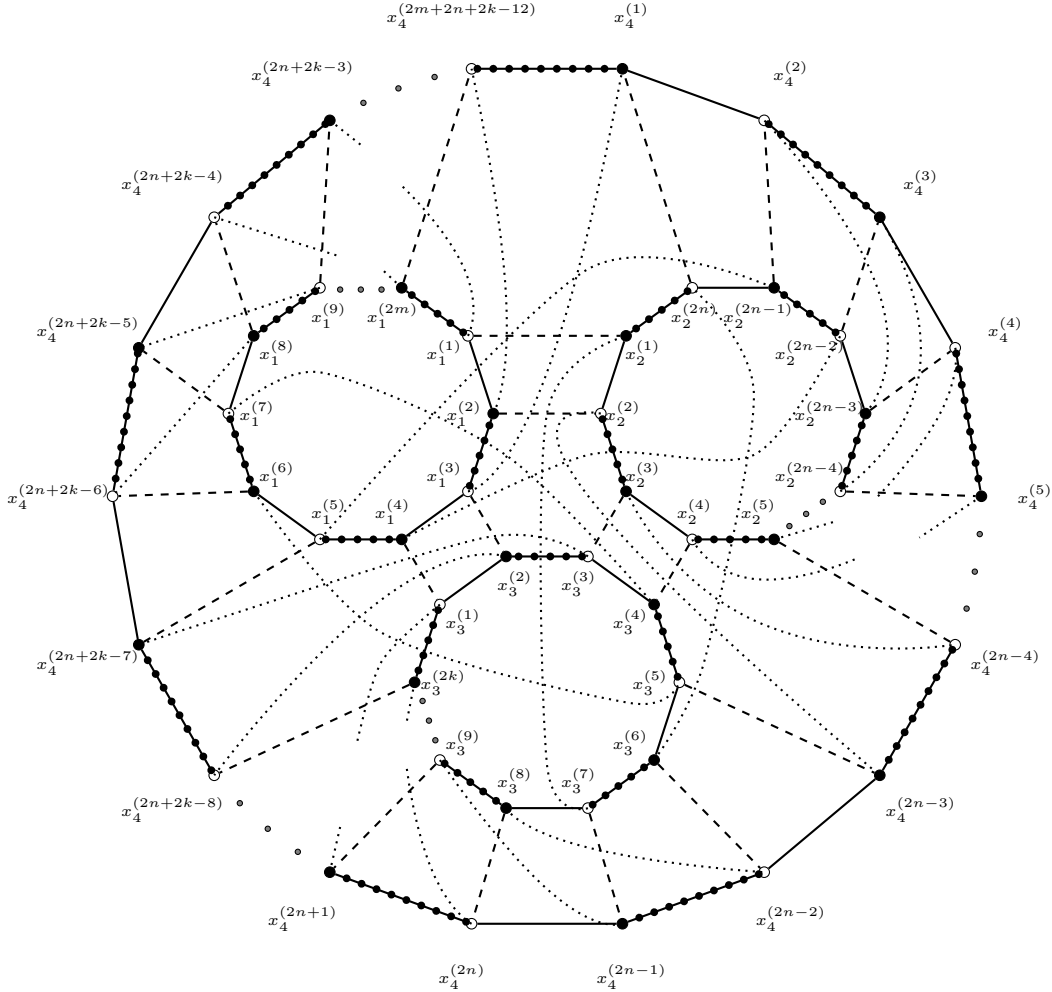
\begin{proof}
Since $C_{R_{mnk}} \neq \emptyset$ for $m,n,k \geq 3$, we can apply  Algorithm $2$. Let $\mathcal{R}= R_{mnk} \cup\{x_1^{m-2}x_3^{-1}x_2^{n-2}x_1^{-1}x_3^{k-2}x_2^{-1}\}$.
Thus, $m_{12}^{(2)}=m_{23}^{(2)}=m_{13}^{(2)}=m_{12}^{(3)}=m_{23}^{(3)}=m_{13}^{(3)}=2$,
$m_{14}^{(2)}=m_{14}^{(3)}=2m-4, m_{24}^{(2)}=m_{24}^{(3)}=2n-4$ and $m_{34}^{(2)}=m_{34}^{(3)}=2k-4$.
Again, we have 
$(n_1,n_2,n_3, n_4)=(2m,2n,2k,2(m+n-6))$ and  $G_i=C_{n_{i}}(x^{(1)}_i, \dots, x^{(n_i)}_i)$ for $1\leq i \leq 4$ as in Figure \ref{fig:mnk}. 
Choose $x_3^{(p)}= x_3^{(2)}$ as in Algorithm $2$, then the $3$-colored graph with the color set $\{0,1,2\}$ is as in Figure \ref{fig:mnk}, which is unique up to an isomorphism. 
For the choices $(q_1,q_2,q_3,q_4)=(4,2n-2,5,1)$ and $(x^{(q_1-1)}_1x^{(q_4)}_4, x^{(q_1+m_{12}^{(3)})}_1 x^{(q_3)}_3)= (x^{(3)}_1x^{(1)}_4, x^{(6)}_1 x^{(6)}_3)$, we get a  $4$-colored graph which yields 
$(\langle S \mid R_{mnk} \rangle, R_{mnk})$. Therefore, for each $m,n,k \geq 3$, we get 
a crystallization $(\Gamma, \gamma)$ of the $3$-manifold $M \langle m,n,k \rangle$.

Now, we show that the crystallization $(\Gamma, \gamma)$ is unique.
Here we choose the pair of colors $(2,3)=(i,j)$ as in the construction of $\tilde{r}$ for $r\in \mathcal{R}$ (cf. Eq. \eqref{tildar}). By similar arguments as in the proofs of  Theorems \ref{theorem:q4n} and \ref{theorem:mn2}, $x^{(2)}_2$,  $x^{(3)}_3$,  $x^{(3)}_1$  are the starting vertices to yield the relations
$x_1^{m-1}x_3^{-1}x_2^{-1}$, $x_2^{n-1}x_1^{-1}x_3^{-1}$, $x_3^{m-1}x_2^{-1}x_1^{-1}$ respectively. Therefore, as in the proofs of previous theorems,
$\{x^{(3)}_3 x^{(2n+2k-7)}_4$, $x^{(2)}_3 x^{(2n+2k-8)}_4,\dots, x^{(8)}_3 x^{(2n-2)}_4\}$,
$\{x^{(2)}_2 x^{(2n-3)}_4$, $ x^{(3)}_2 x^{(2n-4)}_4,\dots, x^{(2n-3)}_2 x^{(2)}_4\}$, 
$\{x^{(3)}_1 x^{(1)}_4$, $x^{(2)}_1 x^{(2m+2n+2k-12)}_4,\dots, x^{(8)}_1 x^{(2n+2k-6)}_4\} \subset \gamma^{-1}(3)$.
Again, $x^{(7)}_1 x^{(4)}_3, x^{(4)}_1 x^{(2n-2)}_2, x^{(1)}_2 x^{(7)}_3 \in$  $\gamma^{-1}(3)$ to yield the relation $x_1^{m-2}x_3^{-1}x_2^{n-2}x_1^{-1}x_3^{k-2}x_2^{-1}$.
Now, there is unique way to choose the remaining edges of color $3$ as in Figure \ref{fig:mnk}. 
Since $x_1^{m-2}x_3^{-1}x_2^{n-2}x_1^{-1}x_3^{k-2}x_2^{-1}$ is the only independent  element in $\overline{R_{mnk}}$ of minimum weight, the theorem follows.
\end{proof}

\begin{remark} \label{remark:minimal}
{\rm  By \cite[Proposition 4]{ca99}, the vertex-minimal crystallizations of prime and handle-free $3$-manifolds with at most $30$ vertices are known (cf. \cite{bccgm11,
cc08, li95}). Thus, our crystallizations of the $3$-manifolds  $M \langle m,2,2 \rangle$ for `$m \geq 3$' and $M \langle m,n,k \rangle$ for `$m \geq 4$
and $n,k \geq 3$' are minimal crystallizations when the number of vertices of the crystallizations are at most $30$. There are no known crystallizations in the literature of  
$M \langle m,2,2 \rangle$ for `$m \geq 3$' and of $M \langle m,n,k \rangle$ for `$m \geq 4$ and $n,k \geq 3$' which have less number of vertices than our constructed ones.
}
\end{remark}

\noindent {\bf Acknowledgement:} The author would like to thank Paola Bandieri and Basudeb Datta for helpful comments. Furthermore, the author would also like to thank the anonymous referees for insightful and helpful remarks.
The author 
is supported by CSIR, India for SPM Fellowship and the UGC Centre for Advanced Studies.


{\footnotesize

\begin{thebibliography}{99}

\bibitem{afw13} M. Aschenbrenner, S. Friedl and H. Wilton, 3-manifold groups, arXiv:1205.0202v3, 2013, 149 pages.

\bibitem{bcg13} Paola Bandieri, Paola Cristofori and Carlo Gagliardi, A census of genus two 3-manifolds up to 42 coloured
tetrahedra, {\em Discrete Math.} {\bf 310} (2010), 2469--2481.

\bibitem{bccgm11}P. Bandieri, M. R. Casali,  P. Cristofori,  L. Grasselli and M. Mulazzani, Computational
aspects of crystallization theory: complexity, catalogues and classification of 3-manifolds,
{\em Atti Sem. Mat. Fis. Univ. Modena} {\bf 58} (2011), 11--45.

\bibitem{bd14} B. Basak and B. Datta, Minimal crystallizations of 3-manifolds,
{\em Electron. J Combin.} {\bf 21} (1) (2014), \#P1.61, 1--25.

\bibitem{bm08}
J. A. Bondy and U. S. R. Murty, {\em Graph Theory: Graduate Texts in Mathematics},
Springer, New York, 2008, xii+651 pages.
\bibitem{ca99}
M. R. Casali, Classification of non-orientable 3-manifolds admitting decompositions into $\leq 26$ coloured tetrahedra,
{\em Acta Appl. Math.} {\bf 54} (1998), 75--97.

\bibitem{cc08}
M. R. Casali and P. Cristofori, A catalogue of orientable 3-manifolds triangulated by 30 coloured tetrahedra,
{\em J. Knot Theory Ramification} {\bf 17} (2008), 1--23.

\bibitem{cc14}
M. R. Casali and P. Cristofori, A note about complexity of lens spaces, {\em Forum Math.} (2014),
DOI:10.1515/forum-2013-0185, published online February 19, 2014, 14 pages.

\bibitem{cgp80}
A. Cavicchioli, L. Grasselli and M. Pezzana, Su di una decomposizione normale per le $n$-variet\`{a} chiuse, {\em
Boll. Un. Mat. Ital.} {\bf 17-B} (1980), 1146-1165.


\bibitem{ep61} D. B. A. Epstein, Finite presentations of groups and 3-manifolds, {\em Quart. J. Math. Oxford}
{\bf 12} (1961), 205--212.


\bibitem{fgg86}
M. Ferri, C. Gagliardi and L. Grasselli, A graph-theoretic representation of PL-manifolds -- A survey on
crystallizations, {\em Acquationes Math.} {\bf 31} (1986), 121--141.

\bibitem{ga79a}
C. Gagliardi, A combinatorial characterization of 3-manifold crystallizations, {\em Boll. Un. Mat. Ital.} {\bf
16-A} (1979), 441--449.

\bibitem{ga79b}
C. Gagliardi, How to deduce the fundamental group of a closed $n$-manifold from a contracted triangulation, {\em
J. Combin. Inform. System Sci.} {\bf 4} (1979), 237--252.

\bibitem{li95}
S. Lins, {\em Gems, computers and attractors for $3$-manifolds},  Series on Knots and Everything,  World Scientific Publishing Co., Inc., River Edge, NJ. {\bf 5} (1995)  xvi+450 pages.

\bibitem{mil75}
J. Milnor, On the 3-dimensional Brieskorn manifolds $M(p,q,r)$, Knots, groups, and $3$-manifolds, {\em Ann. of Math. Studies, Princeton Univ. Press, Princeton, N. J.} {\bf 84} (1975), 175--225.

\bibitem{pe03}
G. Perelman, Finite extinction time for the solutions to the Ricci flow on certain three-manifolds,
arXiv:\,math/0307245v1, 2003, 7 pages.

\bibitem{pe74}
M. Pezzana, Sulla struttura topologica delle variet\`{a} compatte, {\em Atti Sem. Mat. Fis. Univ. Modena} {\bf
23} (1974), 269--277.

\bibitem{th80}
W. Thurston, {\em The Geometry and Topology of Three-Manifolds}, Princeton lecture notes on geometric structures on $3$-manifolds (1980), 360 pages,
(\texttt{http://library.msri.org/books/gt3m/PDF/}).

\bibitem{sw13}
E. Swartz, The average dual surface of a cohomology class and minimal simplicial decompositions of infinitely
many lens spaces, arXiv:1310.1991v2, 2013, 6 pages.

\end{thebibliography}
}
\end{document}